\documentclass[english]{article}
\usepackage[T1]{fontenc}
\usepackage[utf8]{inputenc}
\usepackage{babel}
\usepackage{float}
\usepackage{units}
\usepackage{dsfont}
\usepackage{amsmath}
\usepackage{amsthm}
\usepackage{amssymb}
\usepackage[numbers]{natbib}
\usepackage[pdfusetitle,
 bookmarks=true,bookmarksnumbered=false,bookmarksopen=false,
 breaklinks=false,pdfborder={0 0 1},backref=false,colorlinks=false]
 {hyperref}

\makeatletter

\floatstyle{ruled}
\newfloat{algorithm}{tbp}{loa}
\providecommand{\algorithmname}{Algorithm}
\floatname{algorithm}{\protect\algorithmname}

\theoremstyle{plain}
\newtheorem{assumption}{\protect\assumptionname}[section]
\theoremstyle{remark}
\newtheorem{rem}{\protect\remarkname}[section]
\theoremstyle{plain}
\newtheorem{thm}{\protect\theoremname}[section]
\theoremstyle{plain}
\newtheorem{lem}{\protect\lemmaname}[section]

\@ifundefined{date}{}{\date{}}
\usepackage[paperwidth=8.5in, paperheight=11in, margin=1in]{geometry}
\usepackage{enumitem}
\usepackage{algorithmic}

\makeatother

\providecommand{\assumptionname}{Assumption}
\providecommand{\lemmaname}{Lemma}
\providecommand{\remarkname}{Remark}
\providecommand{\theoremname}{Theorem}

\begin{document}
\global\long\def\E{\mathbb{E}}%
\global\long\def\F{\mathcal{F}}%
\global\long\def\N{\mathbb{N}}%
\global\long\def\R{\mathbb{R}}%
\global\long\def\X{\boldsymbol{X}}%
\global\long\def\bg{\boldsymbol{g}}%
\global\long\def\bv{\boldsymbol{v}}%
\global\long\def\bx{\boldsymbol{x}}%
\global\long\def\by{\boldsymbol{y}}%
\global\long\def\bz{\boldsymbol{z}}%
\global\long\def\sx{x}%
\global\long\def\sy{y}%
\global\long\def\sz{z}%
\global\long\def\p{p}%
\global\long\def\bxi{\boldsymbol{\xi}}%
\global\long\def\bzero{\boldsymbol{0}}%
\global\long\def\defeq{\triangleq}%
\global\long\def\argmin{\mathrm{argmin}}%
\global\long\def\dom{\mathrm{dom}}%
\global\long\def\interior{\mathrm{int}}%
\global\long\def\mydots{\dots}%
\global\long\def\d{\mathrm{d}}%
\global\long\def\Breg{\mathrm{D}}%
\global\long\def\SMD{\textsf{SMD}}%
\global\long\def\ASMD{\textsf{ASMD}}%
\global\long\def\SGD{\textsf{SGD}}%
\global\long\def\SGDM{\textsf{SGDM}}%
\global\long\def\NSGD{\textsf{NSGD}}%
\global\long\def\AdaGrad{\textsf{AdaGrad}}%
\global\long\def\AdaGradNorm{\textsf{AdaGrad-Norm}}%
\global\long\def\Adam{\textsf{Adam}}%
\global\long\def\Mini{\textsf{Mini-batch SGD}}%
\global\long\def\metric{\textsf{METRIC}}%
\global\long\def\rate{\textsf{RATE}}%

\title{In-Expectation Convergence of Stochastic Gradient Methods under Heavy-Tailed
Noise}
\author{Zijian Liu\thanks{Stern School of Business, New York University, zl3067@stern.nyu.edu.}}
\maketitle
\begin{abstract}
Many stochastic gradient methods are believed not to converge when
the noise in stochastic gradients has only a finite $p$-th moment
for $p\in\left(1,2\right)$, a setting known as the heavy-tailed noise
assumption. However, some recent studies have found that Stochastic
Gradient Descent ($\textsf{SGD}$), without any modification to its
update rule, can surprisingly converge in expectation for convex problems
with bounded domains, highlighting the potential of classical stochastic
gradient methods. Inspired by this recent progress, we provide a comprehensive
study of stochastic optimization under heavy-tailed noise and establish
new in-expectation convergence results for Stochastic Mirror Descent
($\textsf{SMD}$) and Accelerated Stochastic Mirror Descent ($\textsf{ASMD}$)
in convex optimization, and for $\textsf{SGD}$ and Stochastic Gradient
Descent with Momentum ($\textsf{SGDM}$) in nonconvex optimization.
Notably, our results not only hold without algorithmic changes but
also avoid restrictive assumptions, such as bounded domains, imposed
in prior work. More importantly, our analysis provides a new, elegant,
and powerful framework for studying heavy-tailed stochastic optimization,
opening a new route to understanding first-order stochastic gradient
methods.
\end{abstract}

\section{Introduction\label{sec:introduction}}

The heavy-tailed noise condition, which means that the noise in stochastic
gradients has only a finite $p$-th moment for $p\in\left(1,2\right]$,
was introduced in the pioneering works \citep{pmlr-v97-simsekli19a,NEURIPS2020_b05b57f6},
and has been observed in different areas of machine learning, including
deep learning \citep{ahn2024linear,pmlr-v238-battash24a} and reinforcement
learning \citep{pmlr-v139-garg21b}. However, under heavy-tailed noise,
many stochastic gradient methods, such as the classical Stochastic
Gradient Descent ($\SGD$) \citep{10.1214/aoms/1177729586} algorithm,
are believed to fail to converge when $\p\in\left(1,2\right)$. For
instance, \citep{NEURIPS2020_b05b57f6} provides a simple example
in which $\E[\left\Vert \nabla F(\bx_{t})\right\Vert _{2}^{2}]$ can
be $+\infty$ at an arbitrary point $\bx_{t}$ in the optimization
trajectory, where $F$ denotes the objective function. Motivated by
this issue, various techniques have been proposed to address the non-convergence
phenomenon in stochastic gradient methods. In particular, two mechanisms,
namely gradient clipping \citep{pmlr-v28-pascanu13,NEURIPS2020_b05b57f6,cutkosky2021high,NEURIPS2022_349956de,liu2023stochastic,pmlr-v195-liu23c,NEURIPS2023_4c454d34,pmlr-v202-sadiev23a,pmlr-v235-liu24bo,pmlr-v235-gorbunov24a,parletta2024improved,pmlr-v238-puchkin24a,pmlr-v258-armacki25a}
and gradient normalization \citep{nesterov1984minimization,pmlr-v258-hubler25a,liu2025nonconvex,JMLR:v26:24-1991},
are known to handle heavy-tailed noise both in theory and practice.
However, to the best of our knowledge, existing mechanisms for handling
heavy-tailed noise inevitably introduce additional tuning parameters
(e.g., the clipping parameter in gradient clipping) and incur extra
computational costs (e.g., computing the norm of the stochastic gradient). 

Although, as discussed above, many stochastic gradient methods are
expected to fail to converge in expectation, three insightful works
\citep{pmlr-v178-vural22a,fatkhullin2025can,liu2025online} have challenged
this popular view and obtained interesting results in both convex
and nonconvex optimization through the study of two classical first-order
methods, $\SGD$ and Stochastic Mirror Descent ($\SMD$) \citep{nemirovskij1983problem,BECK2003167}.

\paragraph{Convex optimization.}

All three works, \citep{pmlr-v178-vural22a,fatkhullin2025can,liu2025online},
provide new results for convex optimization, which we discuss below.
Note that \citep{fatkhullin2025can} and \citep{liu2025online} also
investigate strongly convex optimization. To save space, we defer
the discussion of the strongly convex case to Section \ref{sec:cvx}.

\begin{itemize}[leftmargin=*]

\item \citep{pmlr-v178-vural22a} studies $\SMD$ under heavy-tailed
stochastic gradients rather than heavy-tailed noise and shows that,
for Lipschitz objectives, the average iterate of $\SMD$ converges
at the optimal rate $1/T^{1-\frac{1}{\p}}$ after $T$ iterations,
provided that the mirror map $\psi$ is $\frac{\p}{\p-1}$-uniformly
convex with respect to the underlying general norm $\left\Vert \cdot\right\Vert $
(which needs not be $\left\Vert \cdot\right\Vert _{2}$) \citep{ZALINESCU1983344,AFST_1995_6_4_4_705_0}.
Later, \citep{liu2024revisiting} extends the idea to the standard
heavy-tailed noise setting and establishes the same optimal rate (including
for the last iterate) for composite convex optimization over a broader
class of objective functions satisfying $(L,M)$-smoothness (see Equation
(4.2.2) in \citep{lan2020first}).

However, a known issue is that even in the case $\left\Vert \cdot\right\Vert =\left\Vert \cdot\right\Vert _{2}$,
the $\SMD$ algorithm in \citep{pmlr-v178-vural22a} cannot recover
the standard $\SGD$ method due to the requirement on $\psi$. Moreover,
the corresponding update rule is complicated compared with standard
$\SGD$ (see Section 3.1 of \citep{pmlr-v178-vural22a}). More critically,
the implementation of $\SMD$ in \citep{pmlr-v178-vural22a} requires
the exact value of the tail index $\p$, making the algorithm hard
to tune, since even estimating $\p$ is known to be difficult in practice.

\item \citep{fatkhullin2025can} studies the standard projected $\SGD$
method and establishes the optimal $1/T^{1-\frac{1}{\p}}$ rate in
expectation. Moreover, the authors also investigate high-probability
lower bounds, providing further evidence for the intrinsic difficulty
of stochastic optimization under heavy-tailed noise.

Although \citep{fatkhullin2025can} provides results for any closed
convex set $\X\subseteq\R^{d}$, it does not yield a practical output
unless $\X$ is compact. In addition, \citep{fatkhullin2025can} also
considers heavy-tailed stochastic gradients, which implicitly requires
the objective to be Lipschitz and hence restricts its applicability.

\item \citep{liu2025online} studies a broader problem, online convex
optimization with heavy tails, for different classical algorithms.
In particular, the technique of \citep{liu2025online} can be applied
to stochastic convex optimization, also yielding an optimal $1/T^{1-\frac{1}{\p}}$
rate in expectation for both the average and last iterates of projected
$\SGD$ for $(L,M)$-smooth functions. Moreover, as pointed out in
Remarks 1 and 3 of \citep{liu2025online}, the results in \citep{liu2025online}
can be provably extended to $\SMD$ for composite objectives.

However, as discussed in \citep{liu2025online}, the analysis developed
there critically relies on the boundedness of the domain $\X$. This
compactness requirement can be restrictive in stochastic convex optimization,
since many standard constrained, unconstrained, and regularized problems
naturally involve unbounded feasible regions. Therefore, removing
this requirement is important for extending the applicability of such
results to more general and practical settings.

\end{itemize}

\paragraph{Nonconvex optimization.}

Both \citep{fatkhullin2025can} and \citep{liu2025online} study nonconvex
optimization as well. However, we limit our attention to \citep{fatkhullin2025can},
since the algorithm developed in \citep{liu2025online} contains additional
operations that go far beyond standard $\SGD$.

More precisely, under the assumption of heavy-tailed stochastic gradients,
\citep{fatkhullin2025can} shows that $\E[\sum_{t=1}^{T}\left\Vert \nabla F(\bx_{t})\right\Vert _{2}^{2}/T]$
converges at the rate of $1/T^{1-\frac{1}{\p}}$, but under the unusual
condition that the objective $F$ is $\p$-H\"{o}lder smooth. In
this same setting, they also show a lower sample complexity bound
specific to the $\SGD$ algorithm for finding stationary points of
convex objectives.

Under heavy-tailed noise and when $F$ is smooth in the classical
sense, i.e., $\nabla F$ is Lipschitz, they show that $\Mini$ guarantees
$\E[\sum_{t=1}^{T}\left\Vert \nabla F(\bx_{t})\right\Vert _{2}^{\p}/T]\leq\varepsilon^{\p}$
using at most $(1/\varepsilon)^{\frac{2\p}{\p-1}}$ stochastic gradients.
Note that their convergence metric is weaker than the commonly considered
$\E[\sum_{t=1}^{T}\left\Vert \nabla F(\bx_{t})\right\Vert _{2}^{2}/T]$
in the literature, and therefore does not contradict \citep{NEURIPS2020_b05b57f6}.
However, as explicitly discussed at the end of Section 5.3 in \citep{fatkhullin2025can},
their analysis cannot be applied to standard $\SGD$ without gradient
batching and thus leaves open an unsolved problem for $\SGD$.

Taken together, these results suggest that classical first-order methods
such as $\SMD$ and $\SGD$ can still guarantee convergence without
algorithmic changes. However, as mentioned previously, existing results
often require restrictive assumptions and hence limit their value.
Naturally, one may ask:
\begin{center}
\textit{Can classical stochastic gradient methods converge in both
convex and nonconvex optimization under heavy-tailed noise, without
imposing additional restrictive assumptions?}
\par\end{center}

\paragraph{Our contributions.}

The key finding of this work is that, under a properly chosen convergence
criterion, two classical first-order stochastic gradient methods,
$\SMD$ and $\SGD$, converge in expectation without imposing additional
restrictive assumptions, answering the above question affirmatively. 

Concretely, inspired by existing works (e.g., \citep{fatkhullin2025can}
for strongly convex optimization), let $\metric$ denote one of the
commonly adopted convergence metrics, such as $F(\bar{\bx}_{T+1})-F_{\star}$
and $F(\bx_{T+1})-F_{\star}$ for (strongly) convex optimization (where
$\bar{\bx}_{T+1}$ denotes an average iterate over the optimization
trajectory up to time $T+1$ and $F_{\star}$ denotes the optimal
function value), and $\sum_{t=1}^{T}\left\Vert \nabla F(\bx_{t})\right\Vert _{2}^{2}/T$
for nonconvex optimization. We show that $\E[\left(\metric\right)^{\frac{\p}{2}}]$
converges at a rate of the form $\left(\rate\right)^{\frac{\p}{2}}$.

More specifically, under the standard heavy-tailed noise condition,
we establish the following results:

\begin{itemize}[leftmargin=*]

\item For convex optimization, $\SMD$ converges at $\rate=1/T^{1-\frac{1}{\p}}$
for composite objectives over a general function class satisfying
the $(L,M)$-relative smoothness condition (see Assumption \ref{assu:SMD}).
Notably, our results hold for both the average and last iterates,
require only that the mirror map $\psi$ satisfies the standard $1$-strong
convexity condition, and apply to any domain that is not necessarily
compact. Under an additional Bregman growth condition (see Assumption
\ref{assu:SMD}), we obtain an improved bound with $\rate=1/T^{2-\frac{2}{\p}}$.
Additionally, we extend our results to Accelerated Stochastic Mirror
Descent ($\ASMD$), obtaining a better rate with a higher-order term
that decays faster than that of $\SMD$. We also show that both $\SMD$
and $\ASMD$ guarantee convergence even when the value of $\p$ is
unknown in advance, provided that $\p\in\left(\nicefrac{4}{3},2\right]$.

\item For nonconvex optimization, $\SGD$ converges at $\rate=1/T^{1-\frac{1}{\p}}$
for $L$-upper smooth functions (see Assumption \ref{assu:SGD}).
Notably, our results do not require gradient batching, thereby resolving
the problem left by \citep{fatkhullin2025can}, and they also yield
a better sample complexity bound than the one in \citep{fatkhullin2025can}.
Under an additional condition, we show that $\SGD$ guarantees convergence
in the form of $\E\left[\metric\right]\leq\rate$, removing the previously
mentioned $\frac{\p}{2}$ exponent. Furthermore, we extend our results
to Stochastic Gradient Descent with Momentum ($\SGDM$), establishing
its first convergence bound under heavy-tailed noise. We also show
that both $\SGD$ and $\SGDM$ guarantee convergence even when the
value of $\p$ is unknown in advance, provided that $\p\in\left(\nicefrac{4}{3},2\right]$.

\end{itemize}

More importantly, our analysis provides a new, elegant, and powerful
framework for studying stochastic optimization under the heavy-tailed
noise assumption, opening a new route to understanding first-order
stochastic gradient methods.

\subsection{Related Work\label{subsec:related-work}}

A large body of work has studied heavy-tailed noise and produced a
rich literature. In particular, two mechanisms, namely gradient clipping
\citep{pmlr-v28-pascanu13,NEURIPS2020_b05b57f6,cutkosky2021high,NEURIPS2022_349956de,liu2023stochastic,pmlr-v195-liu23c,NEURIPS2023_4c454d34,pmlr-v202-sadiev23a,pmlr-v235-liu24bo,pmlr-v235-gorbunov24a,parletta2024improved,pmlr-v238-puchkin24a,pmlr-v258-armacki25a}
and gradient normalization \citep{nesterov1984minimization,NEURIPS2020_b05b57f6,pmlr-v258-hubler25a,liu2025nonconvex,JMLR:v26:24-1991},
are known to handle heavy-tailed noise both in theory and practice.
However, this work focuses on the original $\SMD$ and $\SGD$ methods
without algorithmic modifications. Therefore, \citep{pmlr-v178-vural22a,fatkhullin2025can,liu2025online}
are the three most related works as discussed above. In addition,
we would like to mention one additional work \citep{NEURIPS2021_9cdf2656},
which also studies vanilla $\SGD$ for convex optimization under heavy-tailed
noise, but under a rare condition called the $\p$-positive definite
Hessian assumption (see Definition 1 and Assumption 1 in \citep{NEURIPS2021_9cdf2656}).
To keep this work focused on the standard setting, we do not discuss
it further.

We also summarize the existing minimax in-expectation lower bounds
in different settings. For Lipschitz convex and strongly convex optimization,
we have $\E\left[F(\bx_{T+1})-F_{\star}\right]\gtrsim1/T^{1-\frac{1}{\p}}$
\citep{nemirovskij1983problem,pmlr-v178-vural22a,liu2026clipped}
and $\E\left[F(\bx_{T+1})-F_{\star}\right]\gtrsim1/T^{2-\frac{2}{\p}}$
\citep{liu2026clipped}, respectively. To the best of our knowledge,
there is no known lower bound for smooth convex optimization. As for
smooth strongly convex optimization, we have $\E\left[F(\bx_{T+1})-F_{\star}\right]\gtrsim1/T^{2-\frac{2}{\p}}$
as proved in \citep{NEURIPS2020_b05b57f6} but under the assumption
of heavy-tailed stochastic gradients. For smooth nonconvex optimization,
we have $\E\left[\sum_{t=1}^{T}\left\Vert \nabla F(\bx_{t})\right\Vert _{2}/T\right]\gtrsim1/T^{\frac{\p-1}{3\p-2}}$
\citep{NEURIPS2020_b05b57f6,liu2025nonconvex}. \citep{fatkhullin2025can}
also shows an algorithm-dependent lower bound for $\SGD$ to find
stationary points (see Theorem 5.5 in \citep{fatkhullin2025can} for
details).

\section{Preliminary}

\paragraph{Notation.}

$\N$ is the set of natural numbers excluding $0$. For any $T\in\N$,
we denote $\left[T\right]\defeq\left\{ 1,\mydots,T\right\} $. We
write $a\land b\defeq\min\left\{ a,b\right\} $. $\left\lceil \cdot\right\rceil $
and $\left\lfloor \cdot\right\rfloor $ are the ceiling and floor
functions, respectively. $\left\langle \cdot,\cdot\right\rangle $
denotes the standard Euclidean inner product on $\R^{d}$. $\left\Vert \cdot\right\Vert $
is a general norm on $\R^{d}$, and $\left\Vert \cdot\right\Vert _{\star}$
represents its dual norm, i.e., $\left\Vert \cdot\right\Vert _{\star}\defeq\sup_{\bv\in\R^{d}:\left\Vert \bv\right\Vert =1}\left\langle \cdot,\bv\right\rangle $.
For a set $\X\subseteq\R^{d}$, $\interior\X$ denotes the interior
of $\X$, and $\iota_{\X}$ is the indicator function of $\X$ (i.e.,
$\iota_{\X}(\bx)=0$ if $\bx\in\X$ and $\iota_{\X}(\bx)=+\infty$,
otherwise). Given a function $h:\R^{d}\to\R\cup\left\{ +\infty\right\} $,
$\dom h\defeq\left\{ \bx\in\R^{d}:h(\bx)<+\infty\right\} $ denotes
its domain, and $\partial h(\by)$ stands for the subdifferential
of $h$ at $\by\in\R^{d}$. Moreover, we write $\Breg_{h}(\bx,\by)\defeq h(\bx)-h(\by)-\left\langle \nabla h(\by),\bx-\by\right\rangle ,\forall\bx\in\R^{d},\by\in\dom h$,
where $\nabla h(\by)$ is the gradient of $h$ at $\by$ when $h$
is differentiable, or an arbitrary element of $\partial h(\by)$ when
$h$ is non-differentiable and $\partial h(\by)$ is nonempty.

\paragraph{Assumption.}

Our analysis relies on Assumption \ref{assu:oracle} below.
\begin{assumption}
\label{assu:oracle}We make the following assumption on the stochastic
gradients:

\begin{itemize}[leftmargin=*]

\item For a fixed function $h$, given a point $\bx_{t}$ at the
$t$-th iteration, one can query a stochastic gradient $\bg_{t}$
satisfying $\E_{t-1}\left[\bg_{t}\right]=\nabla h(\bx_{t})$, where
$\E_{t}\left[\cdot\right]\defeq\E\left[\cdot\mid\F_{t}\right]$ denotes
conditional expectation and $\F_{t}\defeq\sigma\left(\bg_{1},\mydots,\bg_{t}\right)$
is the natural filtration.

\item The stochastic noise $\bxi_{t}\defeq\bg_{t}-\nabla h(\bx_{t})\in\F_{t}$
satisfies $\E_{t-1}\left[\left\Vert \bxi_{t}\right\Vert _{\star}^{\p}\right]\leq\sigma^{\p}$,
where $\p\in\left(1,2\right]$ is the tail index and $\sigma\geq0$
denotes the noise level.

\end{itemize}
\end{assumption}
The first requirement in Assumption \ref{assu:oracle} is known as
unbiasedness and is standard in stochastic optimization \citep{nemirovskij1983problem,Solr-KOHA-OAI-TEST:19722,doi:10.1137/16M1080173,lan2020first}.
The second condition in Assumption \ref{assu:oracle} is commonly
referred to as the heavy-tailed noise condition, which was introduced
in the pioneering works \citep{pmlr-v97-simsekli19a,NEURIPS2020_b05b57f6}.

\section{Convex Optimization\label{sec:cvx}}

In this section, we study convex optimization. Before proceeding,
we introduce some preliminaries.

\paragraph{Objective.}

We are interested in minimizing an objective in the form of $F\defeq f+r$
over $\R^{d}$, where $f:\R^{d}\to\R$ and $r:\R^{d}\to\R\cup\left\{ +\infty\right\} $.
We remark that $\dom F=\dom r$ in this setting.

\paragraph{Mirror map.}

There exists a function $\psi:\R^{d}\to\R\cup\left\{ +\infty\right\} $
such that $\psi$ is strictly convex and differentiable on $\interior\dom\psi$.
Moreover, we assume that $\dom F\subseteq\interior\dom\psi$, and
that $\psi$ is $1$-strongly convex with respect to $\left\Vert \cdot\right\Vert $
on $\dom F$, i.e., $\Breg_{\psi}(\bx,\by)\geq\frac{1}{2}\left\Vert \bx-\by\right\Vert ^{2},\forall\bx,\by\in\dom F$.
\begin{rem}
Alternative sets of conditions for the mirror map can be found in,
for example, \citep{doi:10.1137/S0363012995281742}, Section 6.3 of
\citep{orabona2019modern}, and Section 3.2 of \citep{lan2020first}.
\end{rem}

\paragraph{Convergence metric.}

For convex optimization, as mentioned before and as will be seen later,
we prove our convergence rates in the form of $\E\left[\left(\metric\right)^{\frac{\p}{2}}\right]\lesssim\left(\rate\right)^{\frac{\p}{2}}$
(up to possible logarithmic factors), where $\metric$ could be, for
example, the standard function value gap evaluated at a weighted average
iterate $\bar{\bx}_{T+1}$. Certainly, such results are weaker than
bounds of the form $\E\left[\metric\right]\leq\rate$ provided for
some results in the three most relevant works \citep{pmlr-v178-vural22a,fatkhullin2025can,liu2025online},
because $\E\left[\left(\metric\right)^{\frac{\p}{2}}\right]\leq\left(\E\left[\metric\right]\right)^{\frac{\p}{2}}$
holds by H\"{o}lder's inequality since $\frac{\p}{2}\leq1$ . However,
when discussing rates and comparing our results with existing ones,
we focus on the $\rate$ term and ignore its exponent $\frac{\p}{2}$
in our bound to avoid unnecessary complexity.

\subsection{Convergence of Stochastic Mirror Descent}

\begin{algorithm}[H]
\caption{\label{alg:SMD}Stochastic Mirror Descent ($\protect\SMD$)}

\begin{algorithmic}[1]

\STATE\textbf{Input:} initial point $\bx_{1}\in\dom F$, stepsize
$\eta_{t}>0$, mirror map $\psi$

\FOR{$t=1$ \textbf{to} $T$}

\STATE$\bx_{t+1}=\argmin_{\bx\in\R^{d}}r(\bx)+\left\langle \bg_{t},\bx\right\rangle +\frac{\Breg_{\psi}(\bx,\bx_{t})}{\eta_{t}}$

\ENDFOR

\end{algorithmic}
\end{algorithm}

\begin{rem}
As one may observe, Algorithm \ref{alg:SMD} in fact coincides with
the composite version of $\SMD$ proposed in \citep{duchi2010composite}.
However, for simplicity, we call it $\SMD$.
\end{rem}
We consider the classical Stochastic Mirror Descent ($\SMD$) method
\citep{nemirovskij1983problem,BECK2003167}, presented in Algorithm
\ref{alg:SMD}. When $\left\Vert \cdot\right\Vert =\left\Vert \cdot\right\Vert _{2}$,
Algorithm \ref{alg:SMD} recovers the classical $\SGD$ method (with
a proximal update) by taking $\psi=\frac{1}{2}\left\Vert \cdot\right\Vert ^{2}$.

\paragraph{Assumption.}

To analyze $\SMD$ for convex optimization, we require Assumption
\ref{assu:SMD} below.
\begin{assumption}
\label{assu:SMD}We make the following assumptions on the objective
function $F$:

\begin{itemize}[leftmargin=*]

\item $F$ has an optimal solution, i.e., $\exists\bx_{\star}\in\R^{d}$
such that $F_{\star}\defeq F(\bx_{\star})=\inf_{\bx\in\R^{d}}F(\bx)$.
In addition, we let $\X_{\star}\defeq\argmin_{\bx\in\R^{d}}F(\bx)$
be the set of all optimal solutions.

\item $f$ is convex, $r$ is proper closed convex, and $F$ satisfies
the $\mu$-Bregman growth condition, i.e., $\exists\mu\geq0$ such
that $F(\bx)-F_{\star}\geq\mu\Breg_{\psi}(\X_{\star},\bx),\forall\bx\in\dom F$,
where $\Breg_{\psi}(\X_{\star},\bx)\defeq\inf_{\bx_{\star}\in\X_{\star}}\Breg_{\psi}(\bx_{\star},\bx)$.

\item $f$ is $(L,M)$-relatively smooth, i.e., $\exists L\geq0,M\geq0$
such that $L+M>0$ and $\Breg_{f}(\bx,\by)\leq L\Breg_{\psi}(\bx,\by)+M\sqrt{2\Breg_{\psi}(\bx,\by)},\forall\bx,\by\in\dom F$.

\end{itemize}
\end{assumption}
\begin{rem}
The Bregman growth condition was introduced in \citep{gutman2018unified,bauschke2019linear,doi:10.1287/moor.2019.1047}
and can be viewed as a natural generalization of the classical quadratic
growth condition under the Euclidean norm, which has been extensively
studied in the literature (see, e.g., \citep{necoara2019linear}).
\end{rem}
\begin{rem}
The notion of $(L,M)$-relative smoothness is a further relaxation
of $L$-relative smoothness \citep{doi:10.1287/moor.2016.0817,doi:10.1137/16M1099546}
and $M$-relative Lipschitz continuity \citep{doi:10.1287/ijoo.2018.0008,NEURIPS2020_b67fb336},
and also generalizes the concept of $(L,M)$-smoothness introduced
in \citep{lan2020first}.
\end{rem}
With Assumption \ref{assu:SMD}, let $\Breg\defeq\Breg_{\psi}(\X_{\star},\bx_{1})$.
We are now ready to present our results.

\paragraph{The case $\mu=0$.}

We first consider general convex optimization.
\begin{thm}
\label{thm:main-SMD-cvx-exp}Under Assumptions \ref{assu:oracle}
(with $h=f$) and \ref{assu:SMD} (with $\mu=0$), for any $T\in\N$,
$\SMD$ (Algorithm \ref{alg:SMD}) guarantees that, for $\bar{\bx}_{T+1}\defeq\frac{\sum_{t=1}^{T}\bx_{t+1}}{T}$,

\begin{itemize}[leftmargin=*]

\item $\E\left[\left(F(\bar{\bx}_{T+1})-F_{\star}\right)^{\frac{\p}{2}}\right]\lesssim\left(\frac{L\Breg}{T}+\frac{M\sqrt{\Breg}}{\sqrt{T}}+\frac{\sigma\sqrt{\Breg}}{T^{1-\frac{1}{\p}}}\right)^{\frac{\p}{2}}\log T$
under $\left\{ \eta_{t}=\frac{1}{2L}\land\frac{\sqrt{\Breg}}{M\sqrt{t}}\land\frac{\sqrt{\Breg}}{\sigma t^{\frac{1}{\p}}}\right\} _{t=1}^{T}$.

\item $\E\left[\left(F(\bar{\bx}_{T+1})-F_{\star}\right)^{\frac{\p}{2}}\right]\lesssim\left(\frac{L\Breg}{T}+\frac{M\sqrt{\Breg}}{\sqrt{T}}+\frac{\sigma\sqrt{\Breg}}{T^{1-\frac{1}{\p}}}\right)^{\frac{\p}{2}}$
under $\left\{ \eta_{t}=\frac{1}{2L}\land\frac{\sqrt{\Breg}}{M\sqrt{T}}\land\frac{\sqrt{\Breg}}{\sigma T^{\frac{1}{\p}}}\right\} _{t=1}^{T}$.

\item $\E\left[\left(F(\bx_{T+1})-F_{\star}\right)^{\frac{\p}{2}}\right]\lesssim\left(\frac{L\Breg}{T}+\frac{M\sqrt{\Breg}}{\sqrt{T}}+\frac{\sigma\sqrt{\Breg}}{T^{1-\frac{1}{\p}}}\right)^{\frac{\p}{2}}$
under $\left\{ \eta_{t}=\frac{l_{t}}{2LT}\land\frac{\sqrt{\Breg}l_{t}}{MT^{\frac{3}{2}}}\land\frac{\sqrt{\Breg}l_{t}}{\sigma T^{1+\frac{1}{\p}}}\right\} _{t=1}^{T}$
and $\left\{ l_{t}=T-t+1\right\} _{t=1}^{T}$.

\item $\E\left[\left(F(\bar{\bx}_{T+1})-F_{\star}\right)^{\frac{\p}{2}}\right]\lesssim\left(\frac{\Breg}{\eta\sqrt{T}}+\frac{M^{2}\eta}{\sqrt{T}}+\frac{\sigma^{2}\eta\log^{\frac{2\chi(\p)}{\p}}T}{T^{\frac{3}{2}-\frac{2}{\p}}}\right)^{\frac{\p}{2}}$
under $\left\{ \eta_{t}=\frac{\eta}{\sqrt{t}}\right\} _{t=1}^{T}$
for $\eta\in\left(0,\frac{1}{2L}\right]$ where $\chi(\p)\defeq\mathds{1}\left[\p=2\right]$,
when $\p\in\left(\nicefrac{4}{3},2\right]$.

\item $\E\left[\left(F(\bar{\bx}_{T+1})-F_{\star}\right)^{\frac{\p}{2}}\right]\lesssim\left(\frac{\Breg}{\eta\sqrt{T}}+\frac{M^{2}\eta}{\sqrt{T}}+\frac{\sigma^{2}\eta}{T^{\frac{3}{2}-\frac{2}{\p}}}\right)^{\frac{\p}{2}}$
under $\left\{ \eta_{t}=\frac{\eta}{\sqrt{T}}\right\} _{t=1}^{T}$
for $\eta\in\left(0,\frac{1}{2L}\right]$, when $\p\in\left(\nicefrac{4}{3},2\right]$.

\end{itemize}
\end{thm}
\begin{rem}
The last two results also extend to the last iterate, up to additional
logarithmic factors. For simplicity, we omit these extensions here.
\end{rem}

\paragraph{Discussion on rates with known $\protect\p$.}

Theorem \ref{thm:main-SMD-cvx-exp} provides the first provable and
optimal convergence rates for both the average and last iterates of
$\SMD$ under only standard, in fact even weaker, requirements, since
the notion of $(L,M)$-relative smoothness further relaxes several
conditions considered in the existing literature.

As a sanity check, when $\p=2$ and $\left\Vert \cdot\right\Vert =\left\Vert \cdot\right\Vert _{2}$,
all rates coincide with the best-known ones for the corresponding
stepsize choices. Moreover, for general $\p\in\left(1,2\right]$,
we remark that these three kinds of stepsize are inspired by \citep{liu2025online}.
In addition, when $\sigma\to0$, all three rates recover the optimal
bound $L\Breg/T+M\sqrt{\Breg}/\sqrt{T}$ (up to $\log T$ for the
time-varying stepsize) for deterministic optimization.

To better understand Theorem \ref{thm:main-SMD-cvx-exp}, we compare
it with the three most related works \citep{pmlr-v178-vural22a,fatkhullin2025can,liu2025online}
below.

\begin{itemize}[leftmargin=*]

\item Compared with \citep{pmlr-v178-vural22a}, a key distinction
is that our mirror map $\psi$ is assumed only to be $1$-strongly
convex, in contrast to the condition of $\frac{\p}{\p-1}$-uniform
convexity for $\psi$ required in \citep{pmlr-v178-vural22a}. Even
in the case $\left\Vert \cdot\right\Vert =\left\Vert \cdot\right\Vert _{2}$
and $r=0$, the resulting update rule of \citep{pmlr-v178-vural22a}
(see Section 3.1 of \citep{pmlr-v178-vural22a}) is complicated compared
with standard $\SGD$, and its implementation requires additional
computational expense. More critically, the update rule of $\SMD$
in \citep{pmlr-v178-vural22a} requires the exact value of the tail
index $\p$, which makes the algorithm hard to tune, since even estimating
$\p$ is difficult in practice.

\item Compared with \citep{fatkhullin2025can}, our results lead
to several improvements in different aspects. 1. The rates in \citep{fatkhullin2025can}
are established only for the standard projected $\SGD$ algorithm
(i.e., $\SMD$ under the $2$-norm and $r=\iota_{\X}$ for a convex
set $\X\subseteq\R^{d}$). In contrast, Theorem \ref{thm:main-SMD-cvx-exp}
works for $\SMD$ under any general norm with a general $r$. 2. Although
\citep{fatkhullin2025can} provides a convergence result on unbounded
domains, unfortunately, it does not yield a practical output, as discussed
in Section \ref{sec:introduction}. In comparison, our rates are proved
for the feasible choices, namely the average and last iterates. 3.
The assumption made by \citep{fatkhullin2025can} in (strongly) convex
optimization is $\E_{t-1}\left[\left\Vert \bg_{t}\right\Vert _{2}^{\p}\right]\leq G^{\p}$,
which implicitly requires the objective to be $G$-Lipschitz and hence
restricts its applicability.

\item Compared with \citep{liu2025online}, we do not require the
domain to be bounded. This removes a restrictive compactness assumption
and allows our results to apply to more common settings. As a result,
our framework is more broadly applicable and better aligned with practical
optimization scenarios.

\end{itemize}

Lastly, we highlight that our analysis differs significantly from
the three works discussed above, as well as from other works related
to stochastic optimization under heavy-tailed noise. In particular,
our proof does not follow the standard arguments used in prior works
on heavy-tailed noise, but instead relies on a new framework. For
details, we refer the reader to Section \ref{sec:analysis}.

\paragraph{Discussion on rates with unknown $\protect\p$.}

When $\p$ is unknown in advance, Theorem \ref{thm:main-SMD-cvx-exp}
suggests that the classical choices of stepsize $\left\{ \eta_{t}=\frac{\eta}{\sqrt{t}}\right\} _{t=1}^{T}$
and $\left\{ \eta_{t}=\frac{\eta}{\sqrt{T}}\right\} _{t=1}^{T}$ already
guarantee convergence if $\p\in\left(\nicefrac{4}{3},2\right]$. As
far as we know, there are no prior results of this kind for convex
optimization under heavy-tailed noise. If a bounded domain is additionally
imposed, unfortunately, these rates are worse than the bound achieved
by $\AdaGradNorm$ \citep{liu2025online}, which provably converges
at a rate of $1/T^{1-\frac{1}{\p}}$ for all $\p\in\left(1,2\right]$
without knowing $\p$.

\paragraph{The case $\mu>0$.}

We next consider a positive parameter in the Bregman growth condition.
\begin{thm}
\label{thm:main-SMD-str-exp}Under Assumptions \ref{assu:oracle}
(with $h=f$) and \ref{assu:SMD} (with $\mu>0$), for any $T\in\N$,
$\SMD$ (Algorithm \ref{alg:SMD}) guarantees that, for $\kappa\defeq\frac{L}{\mu}$,
$\bar{\bx}_{T+1}\defeq\frac{\sum_{t=1}^{T}2(t+8\kappa+1)\bx_{t+1}}{T(T+16\kappa+3)}$,
and $T^{\prime}\defeq\left\lceil \frac{T}{2}\right\rceil $,

\begin{itemize}[leftmargin=*]

\item $\E\left[\left(F(\bar{\bx}_{T+1})-F_{\star}\right)^{\frac{\p}{2}}\right]\lesssim\left(\frac{\mu(1+\kappa)^{2}\Breg}{T(T+\kappa)}+\frac{M^{2}}{\mu(T+\kappa)}+\frac{\sigma^{2}}{\mu T^{1-\frac{2}{\p}}(T+\kappa)}\right)^{\frac{\p}{2}}$
under $\left\{ \eta_{t}=\frac{4}{\mu(t+8\kappa)}\right\} _{t=1}^{T}$.

\item $\E\left[\left(\Breg_{\psi}(\X_{\star},\bx_{T+1})\right)^{\frac{\p}{2}}\right]\lesssim\left(\frac{(1+\kappa)^{2}\Breg}{(T+\kappa)^{2}}+\frac{M^{2}}{\mu^{2}(T+\kappa)}+\frac{\sigma^{2}}{\mu^{2}(T+\kappa)^{2-\frac{2}{\p}}}\right)^{\frac{\p}{2}}$
under $\left\{ \eta_{t}=\frac{4}{\mu(t+8\kappa)}\right\} _{t=1}^{T}$.

\item $\E\left[\left(F(\bx_{T+1})-F_{\star}\right)^{\frac{\p}{2}}\right]\lesssim\left(\frac{\mu(1+\kappa)^{2}\Breg}{T(T+\kappa)}+\frac{M^{2}}{\mu T}+\frac{\sigma^{2}}{\mu T^{1-\frac{2}{\p}}(T+\kappa)}\right)^{\frac{\p}{2}}$under
$\left\{ \eta_{t}=\frac{4}{\mu(t+8\kappa)}\right\} _{t=1}^{T^{\prime}}$
and $\left\{ \eta_{t}=\frac{T-t+1}{\mu T(T+\kappa)}\right\} _{t=T^{\prime}+1}^{T}$.

\end{itemize}
\end{thm}
Theorem \ref{thm:main-SMD-str-exp} shows the first optimal convergence
results for both the average and last iterates of $\SMD$ under heavy-tailed
noise. More importantly, all rates are proved under the Bregman growth
condition and the $(L,M)$-relative smoothness condition, which are
weaker than almost all assumptions in the literature. Even for the
Euclidean norm, the quadratic growth condition (which is equivalent
to the Bregman growth condition when $\psi=\frac{1}{2}\left\Vert \cdot\right\Vert _{2}^{2}$)
is known as the weakest notion among all the assumptions that serve
as relaxations of strong convexity \citep{pmlr-v130-guille-escuret21a}.
Therefore, under heavy-tailed noise, Theorem \ref{thm:main-SMD-str-exp}
is the most general and strongest result to our knowledge.

Regarding the rates, one can again take $\p=2$ as a sanity check.
The first kind of stepsize $\eta_{t}=\frac{4}{\mu(t+8\kappa)},\forall t\in\left[T\right]$
is standard in the literature on strongly convex optimization. The
second kind of stepsize is inspired by \citep{liu2024revisiting}.
Specifically, in the first stage of the procedure, we employ the classical
stepsize; in the second stage of the procedure, we switch to a linearly
decaying stepsize introduced in \citep{doi:10.1137/24M1717762}. Consequently,
we obtain the first last-iterate rate without any extra logarithmic
factors, improving upon \citep{liu2025online}. A notable feature
is that neither kind of stepsize requires knowledge of $\p$, similar
to the existing literature \citep{fatkhullin2025can,liu2025online}.

To better understand Theorem \ref{thm:main-SMD-str-exp}, we further
compare it with the two most closely related studies, \citep{fatkhullin2025can}
and \citep{liu2025online}, below.

\begin{itemize}[leftmargin=*]

\item Compared with \citep{fatkhullin2025can}, the first and third
points discussed before in the convex setting still hold, i.e., the
results in \citep{fatkhullin2025can} are established only for projected
$\SGD$ and require $\E_{t-1}\left[\left\Vert \bg_{t}\right\Vert _{2}^{\p}\right]\leq G^{\p}$.
In addition, we remark that the rates in \citep{fatkhullin2025can}
are also established for the weak convergence metrics $\E\left[\left(F(\bar{\bx}_{T+1})-F_{\star}\right)^{\frac{\p}{2}}\right]$
and $\E\left[\left\Vert \bx_{\star}-\bx_{T+1}\right\Vert _{2}^{\p}\right]$,
which are similar to the first two metrics in Theorem \ref{thm:main-SMD-str-exp}.

\item Compared with \citep{liu2025online}, we still do not require
a compactness assumption on the domain. This allows our results to
cover unbounded feasible regions and therefore makes our setting more
general and more broadly applicable than the one studied in \citep{liu2025online}.

\end{itemize}

Again, we emphasize that our proof technique is novel relative to
the existing literature. More details will be delivered in Section
\ref{sec:analysis}.

\subsection{Convergence of Accelerated Stochastic Mirror Descent}

\begin{algorithm}[H]
\caption{\label{alg:ASMD}Accelerated Stochastic Mirror Descent ($\protect\ASMD$)}

\begin{algorithmic}[1]

\STATE \textbf{Input:} initial point $\bx_{1}=\by_{1}=\bz_{1}\in\dom F$,
weight $w_{t}\in\left[0,1\right]$, stepsize $\eta_{t}>0$, mirror
map $\psi$

\FOR{$t=1$ \textbf{to} $T$}

\STATE $\bx_{t}=w_{t}\by_{t}+(1-w_{t})\bz_{t}$

\STATE $\by_{t+1}=\argmin_{\by\in\R^{d}}r(\by)+\left\langle \bg_{t},\by\right\rangle +\frac{\Breg_{\psi}(\by,\by_{t})}{\eta_{t}}$

\STATE $\bz_{t+1}=w_{t}\by_{t+1}+(1-w_{t})\bz_{t}$

\ENDFOR

\end{algorithmic}
\end{algorithm}

We study the Accelerated Stochastic Mirror Descent ($\ASMD$) method
described in Algorithm \ref{alg:ASMD}. There are several ways to
accelerate (Stochastic) Mirror Descent. The particular form presented
here is due to \citep{lan2012optimal,lan2020first}. Note that Algorithm
\ref{alg:ASMD} recovers $\SMD$ when $w_{t}=1,\forall t\in\left[T\right]$.

\paragraph{Assumption.}

To analyze $\ASMD$ for convex optimization, we require Assumption
\ref{assu:ASMD} below.
\begin{assumption}
\label{assu:ASMD}We make the following assumptions on the objective
function $F$:

\begin{itemize}[leftmargin=*]

\item $F$ has an optimal solution, i.e., $\exists\bx_{\star}\in\R^{d}$
such that $F_{\star}\defeq F(\bx_{\star})=\inf_{\bx\in\R^{d}}F(\bx)$.
In addition, we let $\X_{\star}\defeq\argmin_{\bx\in\R^{d}}F(\bx)$
be the set of all optimal solutions.

\item $f$ is convex and $r$ is proper closed convex.

\item $f$ is $(L,M)$-smooth, i.e., $\exists L\geq0,M\geq0$ such
that $L+M>0$ and $\Breg_{f}(\bx,\by)\leq\frac{L}{2}\left\Vert \bx-\by\right\Vert ^{2}+M\left\Vert \bx-\by\right\Vert ,\forall\bx,\by\in\dom F$.

\end{itemize}
\end{assumption}
\begin{rem}
Compared with Assumption \ref{assu:SMD}, Assumption \ref{assu:ASMD}
differs in two ways. First, we no longer impose the Bregman growth
condition. Second, we replace $(L,M)$-relative smoothness with $(L,M)$-smoothness.
This is because it is known that, in the deterministic case (i.e.,
$\sigma=0$ in Assumption \ref{assu:oracle}) with $r=0$, functions
satisfying the Polyak-{\L}ojasiewicz condition \citep{POLYAK1963864}
(which implies the quadratic growth condition \citep{pmlr-v130-guille-escuret21a},
equivalently, the Bregman growth condition when $\psi=\frac{1}{2}\left\Vert \cdot\right\Vert _{2}^{2}$)
cannot be accelerated \citep{pmlr-v195-yue23a}. Similarly, $(L,0)$-relative
smoothness cannot yield acceleration \citep{dragomir2022optimal}.
\end{rem}
With Assumption \ref{assu:ASMD}, let $\Breg\defeq\Breg_{\psi}(\X_{\star},\bx_{1})$.
We are now ready to present our results.
\begin{thm}
\label{thm:main-ASMD-cvx-exp}Under Assumptions \ref{assu:oracle}
(with $h=f$) and \ref{assu:ASMD}, for any $T\in\N$, $\ASMD$ (Algorithm
\ref{alg:ASMD}) guarantees that

\begin{itemize}[leftmargin=*]

\item $\E\left[\left(F(\bz_{T+1})-F_{\star}\right)^{\frac{\p}{2}}\right]\lesssim\left(\frac{L\Breg}{T^{2}}+\frac{M\sqrt{\Breg}}{\sqrt{T}}+\frac{\sigma\sqrt{\Breg}}{T^{1-\frac{1}{\p}}}\right)^{\frac{\p}{2}}\log T$
under $\left\{ w_{t}=\frac{2}{t+1}\right\} _{t=1}^{T}$ and $\left\{ \eta_{t}=\frac{1}{2w_{t}L}\land\frac{\sqrt{\Breg}}{w_{t}Mt^{\frac{3}{2}}}\land\frac{\sqrt{\Breg}}{w_{t}\sigma t^{1+\frac{1}{\p}}}\right\} _{t=1}^{T}$.

\item $\E\left[\left(F(\bz_{T+1})-F_{\star}\right)^{\frac{\p}{2}}\right]\lesssim\left(\frac{L\Breg}{T^{2}}+\frac{M\sqrt{\Breg}}{\sqrt{T}}+\frac{\sigma\sqrt{\Breg}}{T^{1-\frac{1}{\p}}}\right)^{\frac{\p}{2}}$
under $\left\{ w_{t}=\frac{2}{t+1}\right\} _{t=1}^{T}$ and $\left\{ \eta_{t}=\frac{1}{2w_{t}L}\land\frac{\sqrt{\Breg}}{w_{t}MT^{\frac{3}{2}}}\land\frac{\sqrt{\Breg}}{w_{t}\sigma T^{1+\frac{1}{\p}}}\right\} _{t=1}^{T}$.

\item $\E\left[\left(F(\bz_{T+1})-F_{\star}\right)^{\frac{\p}{2}}\right]\lesssim\left(\frac{\Breg}{\eta\sqrt{T}}+\frac{M^{2}\eta}{\sqrt{T}}+\frac{\sigma^{2}\eta}{T^{\frac{3}{2}-\frac{2}{\p}}}\right)^{\frac{\p}{2}}$
under $\left\{ w_{t}=\frac{2}{t+1}\right\} _{t=1}^{T}$ and $\left\{ \eta_{t}=\frac{\eta}{w_{t}T^{\frac{3}{2}}}\right\} _{t=1}^{T}$
for $\eta\in\left(0,\frac{1}{2L}\right]$, when $\p\in\left(\nicefrac{4}{3},2\right]$.

\end{itemize}
\end{thm}
\begin{rem}
In Theorem \ref{thm:main-ASMD-cvx-exp}, when $\p$ is unknown, we
can also consider a time-varying stepsize $\left\{ \eta_{t}=\frac{\eta}{w_{t}t^{\frac{3}{2}}}\right\} _{t=1}^{T}$
for $\eta\in\left(0,\frac{1}{2L}\right]$ . For simplicity, we omit
this case here.
\end{rem}
Theorem \ref{thm:main-ASMD-cvx-exp} generalizes the classical bounds
for $\ASMD$ from the finite-variance case (i.e., $\p=2$) \citep{lan2012optimal,lan2020first}
to the heavy-tailed setting. Compared with the convergence results
for $\SMD$ in Theorem \ref{thm:main-SMD-cvx-exp}, the first two
rates for $\ASMD$ established in Theorem \ref{thm:main-ASMD-cvx-exp}
contain a faster-decaying higher-order term. In particular, the leading
smoothness-dependent term is improved from the rates of $\SMD$ to
its accelerated counterpart, while the remaining terms match those
of $\SMD$. As for the last rate, it is of the same order as the corresponding
result for $\SMD$ when the problem-dependent parameters are not assumed
to be known.

\section{Nonconvex Optimization\label{sec:ncvx}}

In this section, we shift our focus to nonconvex optimization. Again,
we introduce some preliminaries.

\paragraph{Notation.}

In nonconvex optimization, $\left\Vert \cdot\right\Vert $ denotes
the standard Euclidean norm on $\R^{d}$, i.e., $\left\Vert \cdot\right\Vert \defeq\sqrt{\left\langle \cdot,\cdot\right\rangle }$.
In this case, the dual norm $\left\Vert \cdot\right\Vert _{\star}$
is also the Euclidean norm.

\paragraph{Objective.}

We are interested in optimizing $F$ over $\R^{d}$, where $F:\R^{d}\to\R$
is differentiable and possibly nonconvex. Since finding a global optimal
solution is generally infeasible in nonconvex optimization, we focus
instead on finding stationary points.

\paragraph{Convergence metric.}

For nonconvex optimization, we still prove convergence rates in the
form of $\E\left[\left(\metric\right)^{\frac{\p}{2}}\right]\lesssim\left(\rate\right)^{\frac{\p}{2}}$
(up to possible logarithmic factors), where $\metric=\sum_{t=1}^{T}\left\Vert \nabla F(\bx_{t})\right\Vert ^{2}/T$
now. However, as mentioned in Section \ref{sec:introduction}, this
measurement is somewhat subtle, as it lies between the classical criterion
$\sum_{t=1}^{T}\left\Vert \nabla F(\bx_{t})\right\Vert ^{2}/T$ for
$\SGD$ and the commonly studied metric $\sum_{t=1}^{T}\left\Vert \nabla F(\bx_{t})\right\Vert /T$
for $\NSGD\text{(}\textsf{M}\text{)}$ \citep{pmlr-v258-hubler25a,liu2025nonconvex,JMLR:v26:24-1991}
and adaptive gradient methods, e.g., $\Adam$ \citep{kingma2014adam,NEURIPS2023_7ac19fdc},
because
\begin{equation}
\left(\E\left[\frac{\sum_{t=1}^{T}\left\Vert \nabla F(\bx_{t})\right\Vert }{T}\right]\right)^{\p}\leq\E\left[\left(\frac{\sum_{t=1}^{T}\left\Vert \nabla F(\bx_{t})\right\Vert ^{2}}{T}\right)^{\frac{\p}{2}}\right]\leq\left(\E\left[\frac{\sum_{t=1}^{T}\left\Vert \nabla F(\bx_{t})\right\Vert ^{2}}{T}\right]\right)^{\frac{\p}{2}}.\label{eq:main-ncvx-metric}
\end{equation}
Again, for simplicity, when discussing our results, we focus on the
$\rate$ term and ignore its exponent $\frac{\p}{2}$ in our bound
in most cases.

\subsection{Convergence of Stochastic Gradient Descent}

\begin{algorithm}[H]
\caption{\label{alg:SGD}Stochastic Gradient Descent ($\protect\SGD$)}

\begin{algorithmic}[1]

\STATE\textbf{Input:} initial point $\bx_{1}\in\R^{d}$, stepsize
$\eta_{t}>0$

\FOR{$t=1$ \textbf{to} $T$}

\STATE$\bx_{t+1}=\bx_{t}-\eta_{t}\bg_{t}$

\ENDFOR

\end{algorithmic}
\end{algorithm}

We consider the standard Stochastic Gradient Descent ($\SGD$) method
\citep{10.1214/aoms/1177729586}, as presented in Algorithm \ref{alg:SGD}.
As discussed previously, $\SGD$ can be recovered from $\SMD$ (Algorithm
\ref{alg:SMD}).

\paragraph{Assumption.}

To analyze $\SGD$ for nonconvex optimization, we require Assumption
\ref{assu:SGD} below.
\begin{assumption}
\label{assu:SGD}We make the following assumptions on the objective
function $F$:

\begin{itemize}[leftmargin=*]

\item $F$ is lower bounded, i.e., $F_{\star}\defeq\inf_{\bx\in\R^{d}}F(\bx)>-\infty$.

\item $F$ is $L$-upper smooth, i.e., $\exists L>0$ such that $\Breg_{F}(\bx,\by)\leq\frac{L}{2}\left\Vert \bx-\by\right\Vert ^{2},\forall\bx,\by\in\R^{d}$.

\end{itemize}
\end{assumption}
\begin{rem}
Note that $F$ under Assumption \ref{assu:SGD} need not be smooth
in the classical sense, meaning that it does not necessarily satisfy
$\left\Vert \nabla F(\bx)-\nabla F(\by)\right\Vert \leq L\left\Vert \bx-\by\right\Vert ,\forall\bx,\by\in\R^{d}$.
\end{rem}
Under Assumption \ref{assu:SGD}, let $\Delta\defeq F(\bx_{1})-F_{\star}$.
We are now ready to state our results.
\begin{thm}
\label{thm:main-SGD-exp}Under Assumptions \ref{assu:oracle} (with
$h=F$) and \ref{assu:SGD}, for any $T\in\N$, $\SGD$ (Algorithm
\ref{alg:SGD}) guarantees that

\begin{itemize}[leftmargin=*]

\item $\E\left[\left(\frac{\sum_{t=1}^{T}\left\Vert \nabla F(\bx_{t})\right\Vert ^{2}}{T}\right)^{\frac{\p}{2}}\right]\lesssim\left(\frac{L\Delta\log^{\frac{2}{\p}}T}{T}+\frac{\sigma^{2}+\sigma\sqrt{L\Delta}\log^{\frac{2}{\p}}T}{T^{1-\frac{1}{\p}}}\right)^{\frac{\p}{2}}$
under $\left\{ \eta_{t}=\frac{1}{L}\land\frac{\sqrt{\Delta/L}}{\sigma t^{\frac{1}{\p}}}\right\} _{t=1}^{T}$.

\item $\E\left[\left(\frac{\sum_{t=1}^{T}\left\Vert \nabla F(\bx_{t})\right\Vert ^{2}}{T}\right)^{\frac{\p}{2}}\right]\lesssim\left(\frac{L\Delta}{T}+\frac{\sigma^{2}}{T^{2-\frac{2}{\p}}}+\frac{\sigma\sqrt{L\Delta}}{T^{1-\frac{1}{\p}}}\right)^{\frac{\p}{2}}$
under $\left\{ \eta_{t}=\frac{1}{L}\land\frac{\sqrt{\Delta/L}}{\sigma T^{\frac{1}{\p}}}\right\} _{t=1}^{T}$.

\item $\E\left[\left(\frac{\sum_{t=1}^{T}\left\Vert \nabla F(\bx_{t})\right\Vert ^{2}}{T}\right)^{\frac{\p}{2}}\right]\lesssim\left(\frac{\Delta}{\eta\sqrt{T}}+\frac{\sigma^{2}}{\sqrt{T}}+\frac{\sigma^{2}L\eta\log^{\frac{2\chi(\p)}{\p}}T}{T^{\frac{3}{2}-\frac{2}{\p}}}\right)^{\frac{\p}{2}}$
under $\left\{ \eta_{t}=\frac{\eta}{\sqrt{t}}\right\} _{t=1}^{T}$
for $\eta\in\left(0,\frac{1}{L}\right]$ where $\chi(\p)\defeq\mathds{1}\left[\p=2\right]$,
when $\p\in\left(\nicefrac{4}{3},2\right]$.

\item $\E\left[\left(\frac{\sum_{t=1}^{T}\left\Vert \nabla F(\bx_{t})\right\Vert ^{2}}{T}\right)^{\frac{\p}{2}}\right]\lesssim\left(\frac{\Delta}{\eta\sqrt{T}}+\frac{\sigma^{2}}{T^{2-\frac{2}{\p}}}+\frac{\sigma^{2}L\eta}{T^{\frac{3}{2}-\frac{2}{\p}}}\right)^{\frac{\p}{2}}$
under $\left\{ \eta_{t}=\frac{\eta}{\sqrt{T}}\right\} _{t=1}^{T}$
for $\eta\in\left(0,\frac{1}{L}\right]$, when $\p\in\left(\nicefrac{4}{3},2\right]$.

\end{itemize}
\end{thm}
Theorem \ref{thm:main-SGD-exp} gives the first convergence guarantee
for $\SGD$ under heavy-tailed noise without any additional procedures,
such as gradient batching \citep{fatkhullin2025can}, gradient clipping
\citep{pmlr-v202-sadiev23a,NEURIPS2023_4c454d34}, or gradient normalization
\citep{pmlr-v258-hubler25a,liu2025nonconvex,JMLR:v26:24-1991}.

\paragraph{Discussion on rates with known $\protect\p$.}

To gain some insight into Theorem \ref{thm:main-SGD-exp}, we first
set $\p=2$, which yields two convergence rates $L\Delta\log(T)/T+(\sigma\sqrt{L\Delta}\log T+\sigma^{2})/\sqrt{T}$
and $(L\Delta+\sigma^{2})/T+\sigma\sqrt{L\Delta}/\sqrt{T}$. We note
that both rates are nearly optimal, since they match the optimal rate
$L\Delta/T+\sigma\sqrt{L\Delta}/\sqrt{T}$ \citep{doi:10.1137/120880811,arjevani2023lower}
up to the additive terms $\sigma^{2}/\sqrt{T}$ and $\sigma^{2}/T$.

For general $\p\in\left(1,2\right]$, we compare Theorem \ref{thm:main-SGD-exp}
with the closest related work, \citep{fatkhullin2025can}. In \citep{fatkhullin2025can},
the authors prove that, for $F$ satisfying Assumption \ref{assu:SGD}
and, additionally, $\ell$-lower smoothness with $\ell>0$ (i.e.,
$-\frac{\ell}{2}\left\Vert \bx-\by\right\Vert ^{2}\leq\Breg_{F}(\bx,\by),\forall\bx,\by\in\R^{d}$),
$\Mini$ can ensure $\E\left[\sum_{t=1}^{T}\left\Vert \nabla F(\bx_{t})\right\Vert ^{\p}/T\right]\leq\varepsilon^{\p}$
after using $(L+\ell)\Delta/\varepsilon^{2}+(\sigma\sqrt{L^{-1}(\ell+L)^{2}\Delta}/\varepsilon^{2})^{\frac{\p}{\p-1}}$
number of stochastic gradients when $T$ is known. In contrast, our
result for the case of known $T$ (the second rate) has different
advantages in the following respects. \textbf{1. Algorithm.} Our result
is proved for $\SGD$ without mini-batching, which directly addresses
the problem explicitly described at the end of Section 5.3 in \citep{fatkhullin2025can}.
\textbf{2. Metric.} Our result is established for $\E\left[(\sum_{t=1}^{T}\left\Vert \nabla F(\bx_{t})\right\Vert ^{2}/T)^{\frac{\p}{2}}\right]$,
which is a stronger metric than $\E\left[\sum_{t=1}^{T}\left\Vert \nabla F(\bx_{t})\right\Vert ^{\p}/T\right]$
considered in \citep{fatkhullin2025can}, since $\sum_{t=1}^{T}\left\Vert \nabla F(\bx_{t})\right\Vert ^{\p}/T\leq(\sum_{t=1}^{T}\left\Vert \nabla F(\bx_{t})\right\Vert ^{2}/T)^{\frac{\p}{2}}$.
\textbf{3. Objective.} Our result is proved for a broader function
class, since it does not require the additional assumption of $\ell$-lower
smoothness. \textbf{4. Complexity.} Our result can be translated into
a sample complexity bound $L\Delta/\varepsilon^{2}+(\sigma\sqrt{L\Delta}/\varepsilon^{2})^{\frac{\p}{\p-1}}$\footnote{One may think that the correct sample complexity should be $L\Delta/\varepsilon^{2}+\left(\sigma/\varepsilon\right)^{\frac{\p}{\p-1}}+(\sigma\sqrt{L\Delta}/\varepsilon^{2})^{\frac{\p}{\p-1}}$.
However, without loss of generality, we can assume $\varepsilon\leq\sqrt{2L\Delta}$,
and note that $\left(\sigma/\varepsilon\right)^{\frac{\p}{\p-1}}\lesssim(\sigma\sqrt{L\Delta}/\varepsilon^{2})^{\frac{\p}{\p-1}}$
under this condition. To justify the condition $\varepsilon\leq\sqrt{2L\Delta}$,
we remark that $L$-upper smoothness implies $\left\Vert \nabla F(\bx)\right\Vert \leq\sqrt{2L(F(\bx)-F_{\star})}$.
Thus, if $\varepsilon>\sqrt{2L\Delta}$, the algorithm can simply
output $\bx_{1}$, so no nontrivial sample complexity bound is needed.}, which is better than the bound given in \citep{fatkhullin2025can}.
In particular, the bound in \citep{fatkhullin2025can} can be arbitrarily
worse than ours due to the additional parameter $\ell$.

Finally, we remark that our proof has a fundamentally different flavor
from that of \citep{fatkhullin2025can}, as discussed in the next
section.

\paragraph{Discussion on rates with unknown $\protect\p$.}

The last two rates in Theorem \ref{thm:main-SGD-exp} indicate that,
surprisingly, $\SGD$ converges without prior knowledge of the tail
index $\p$, provided $\p\in\left(\nicefrac{4}{3},2\right]$. In other
words, the corresponding stepsize choices are independent of $\p$,
although the resulting rates still depend on $\p$ through the analysis.
Interestingly, this phenomenon also applies to $\AdaGrad$ and $\AdaGradNorm$
as recently proved by \citep{liu2026can}, which however measures
convergence under the weaker metric $\E\left[\sum_{t=1}^{T}\left\Vert \nabla F(\bx_{t})\right\Vert /T\right]$.
\begin{thm}
\label{thm:main-SGD-exp-smo+lip}Under Assumptions \ref{assu:oracle}
(with $h=F$) and \ref{assu:SGD}, and an additional condition $\left\Vert \nabla F(\bx)\right\Vert \leq G,\forall\bx\in\R^{d}$,
for any $T\in\N$, $\SGD$ (Algorithm \ref{alg:SGD}) guarantees that

\begin{itemize}[leftmargin=*]

\item $\E\left[\frac{\sum_{t=1}^{T}\left\Vert \nabla F(\bx_{t})\right\Vert ^{2}}{T}\right]\lesssim\frac{G^{\frac{4}{\p}-2}(L\Delta)^{2-\frac{2}{\p}}\log T}{T^{2-\frac{2}{\p}}}+\frac{\sigma G^{\frac{2}{\p}-1}(L\Delta)^{1-\frac{1}{\p}}\log T}{T^{1-\frac{1}{\p}}}$
under $\left\{ \eta_{t}=\frac{\Delta^{\frac{2}{\p}-1}}{G^{\frac{4}{\p}-2}L^{2-\frac{2}{\p}}t^{\frac{2}{\p}-1}}\land\frac{\Delta^{\frac{1}{\p}}}{\sigma G^{\frac{2}{\p}-1}L^{1-\frac{1}{\p}}t^{\frac{1}{\p}}}\right\} _{t=1}^{T}$.

\item $\E\left[\frac{\sum_{t=1}^{T}\left\Vert \nabla F(\bx_{t})\right\Vert ^{2}}{T}\right]\lesssim\frac{G^{\frac{4}{\p}-2}(L\Delta)^{2-\frac{2}{\p}}}{T^{2-\frac{2}{\p}}}+\frac{\sigma G^{\frac{2}{\p}-1}(L\Delta)^{1-\frac{1}{\p}}}{T^{1-\frac{1}{\p}}}$
under $\left\{ \eta_{t}=\frac{\Delta^{\frac{2}{\p}-1}}{G^{\frac{4}{\p}-2}L^{2-\frac{2}{\p}}T^{\frac{2}{\p}-1}}\land\frac{\Delta^{\frac{1}{\p}}}{\sigma G^{\frac{2}{\p}-1}L^{1-\frac{1}{\p}}T^{\frac{1}{\p}}}\right\} _{t=1}^{T}$.

\item $\E\left[\frac{\sum_{t=1}^{T}\left\Vert \nabla F(\bx_{t})\right\Vert ^{2}}{T}\right]\lesssim\frac{\Delta}{\eta\sqrt{T}}+\frac{G^{2}L^{\frac{2\p-2}{2-\p}}\eta^{\frac{2\p-2}{2-\p}}}{T^{\frac{\p-1}{2-\p}}}+\frac{\sigma^{\p}G^{2-\p}L^{\p-1}\eta^{\p-1}}{T^{\frac{\p-1}{2}}}$
under $\left\{ \eta_{t}=\frac{\eta}{\sqrt{T}}\right\} _{t=1}^{T}$
for $\eta\in\left(0,+\infty\right)$.

\end{itemize}
\end{thm}
\begin{rem}
Similar to Theorem \ref{thm:main-SGD-exp}, we can consider a time-varying
stepsize $\left\{ \eta_{t}=\frac{\eta}{\sqrt{t}}\right\} _{t=1}^{T}$
in Theorem \ref{thm:main-SGD-exp-smo+lip}. However, this leads to
a piecewise convergence rate, which we omit for brevity.
\end{rem}
One shortcoming of Theorem \ref{thm:main-SGD-exp} is its convergence
metric, which is weaker than the usual criterion $\E\left[\sum_{t=1}^{T}\left\Vert \nabla F(\bx_{t})\right\Vert ^{2}/T\right]$
for $\SGD$ as discussed previously (see (\ref{eq:main-ncvx-metric})).
In Theorem \ref{thm:main-SGD-exp-smo+lip}, we show that this is no
longer an issue when $F$ is also $G$-Lipschitz. To the best of our
knowledge, this is the first convergence result for nonconvex optimization
under heavy-tailed noise that uses $\E\left[\sum_{t=1}^{T}\left\Vert \nabla F(\bx_{t})\right\Vert ^{2}/T\right]$
as the metric. Notably, in this case, the classical stepsize $\left\{ \eta_{t}=\frac{\eta}{\sqrt{T}}\right\} _{t=1}^{T}$
ensures convergence without knowing $\p$ for every $\p\in\left(1,2\right]$.

\subsection{Convergence of Stochastic Gradient Descent with Momentum}

\begin{algorithm}[H]
\caption{\label{alg:SGDM}Stochastic Gradient Descent with Momentum ($\protect\SGDM$)}

\begin{algorithmic}[1]

\STATE \textbf{Input:} initial point $\bx_{1}\in\R^{d}$, stepsize
$\eta_{t}>0$, momentum parameter $\beta\in\left[0,1\right)$

\FOR{$t=1$ \textbf{to} $T$}

\STATE $\bx_{t+1}=\bx_{t}-\eta_{t}\bg_{t}+\beta(\bx_{t}-\bx_{t-1})$\hfill{}$\triangleright$where
$\bx_{0}\defeq\bx_{1}$

\ENDFOR

\end{algorithmic}
\end{algorithm}

We consider a popular practical variant of $\SGD$ in Algorithm \ref{alg:SGDM},
known as Stochastic Gradient Descent with Momentum ($\SGDM$), which
was introduced in \citep{POLYAK19641}. Clearly, $\SGDM$ recovers
$\SGD$ when the momentum parameter $\beta=0$.

\paragraph{Assumption.}

To analyze $\SGDM$ for nonconvex optimization, we require Assumption
\ref{assu:SGDM} below.
\begin{assumption}
\label{assu:SGDM}We make the following assumptions on the objective
function $F$:

\begin{itemize}[leftmargin=*]

\item $F$ is lower bounded, i.e., $F_{\star}\defeq\inf_{\bx\in\R^{d}}F(\bx)>-\infty$.

\item $F$ is $L$-smooth, i.e., $\exists L>0$ such that $\left\Vert \nabla F(\bx)-\nabla F(\by)\right\Vert \leq L\left\Vert \bx-\by\right\Vert ,\forall\bx,\by\in\R^{d}$.

\end{itemize}
\end{assumption}
\begin{rem}
In comparison to Assumption \ref{assu:SGD}, we now require $F$ to
be smooth in the classical sense.
\end{rem}
Under Assumption \ref{assu:SGDM}, still let $\Delta\defeq F(\bx_{1})-F_{\star}$.
We prove the following convergence rates for $\SGDM$.
\begin{thm}
\label{thm:main-SGDM-exp}Under Assumptions \ref{assu:oracle} (with
$h=F$) and \ref{assu:SGDM}, for any $T\in\N$, $\SGDM$ (Algorithm
\ref{alg:SGDM}) guarantees that

\begin{itemize}[leftmargin=*]

\item $\E\left[\left(\frac{\sum_{t=1}^{T}\left\Vert \nabla F(\bx_{t})\right\Vert ^{2}}{T}\right)^{\frac{\p}{2}}\right]\lesssim\left(\frac{L\Delta\log^{\frac{2}{\p}}T}{(1-\beta)T}+\frac{\sigma^{2}}{T^{1-\frac{1}{\p}}}+\frac{\sigma\sqrt{L\Delta}\log^{\frac{2}{\p}}T}{(1-\beta)^{\frac{1}{\p}-\frac{1}{2}}T^{1-\frac{1}{\p}}}\right)^{\frac{\p}{2}}$
under $\left\{ \eta_{t}=\frac{(1-\beta)^{2}}{2L}\land\frac{(1-\beta)^{\frac{1}{\p}+\frac{1}{2}}\sqrt{\Delta/L}}{\sigma t^{\frac{1}{\p}}}\right\} _{t=1}^{T}$
and $\beta\in\left[0,1\right)$.

\item $\E\left[\left(\frac{\sum_{t=1}^{T}\left\Vert \nabla F(\bx_{t})\right\Vert ^{2}}{T}\right)^{\frac{\p}{2}}\right]\lesssim\left(\frac{L\Delta}{(1-\beta)T}+\frac{\sigma^{2}}{T^{2-\frac{2}{\p}}}+\frac{\sigma\sqrt{L\Delta}}{(1-\beta)^{\frac{1}{\p}-\frac{1}{2}}T^{1-\frac{1}{\p}}}\right)^{\frac{\p}{2}}$
under $\left\{ \eta_{t}=\frac{(1-\beta)^{2}}{2L}\land\frac{(1-\beta)^{\frac{1}{\p}+\frac{1}{2}}\sqrt{\Delta/L}}{\sigma T^{\frac{1}{\p}}}\right\} _{t=1}^{T}$
and $\beta\in\left[0,1\right)$.

\item $\E\left[\left(\frac{\sum_{t=1}^{T}\left\Vert \nabla F(\bx_{t})\right\Vert ^{2}}{T}\right)^{\frac{\p}{2}}\right]\lesssim\left(\frac{\Delta}{\eta\sqrt{T}}+\frac{\sigma^{2}}{\sqrt{T}}+\frac{\sigma^{2}L\eta\log^{\frac{2\chi(\p)}{\p}}T}{(1-\beta)^{\frac{2}{\p}-1}T^{\frac{3}{2}-\frac{2}{\p}}}\right)^{\frac{\p}{2}}$
under $\left\{ \eta_{t}=\frac{(1-\beta)\eta}{\sqrt{t}}\right\} _{t=1}^{T}$
and $\beta\in\left[0,1\right)$ for $\eta\in\left(0,\frac{1-\beta}{2L}\right]$
where $\chi(\p)\defeq\mathds{1}\left[\p=2\right]$, when $\p\in\left(\nicefrac{4}{3},2\right]$.

\item $\E\left[\left(\frac{\sum_{t=1}^{T}\left\Vert \nabla F(\bx_{t})\right\Vert ^{2}}{T}\right)^{\frac{\p}{2}}\right]\lesssim\left(\frac{\Delta}{\eta\sqrt{T}}+\frac{\sigma^{2}}{T^{2-\frac{2}{\p}}}+\frac{\sigma^{2}L\eta}{(1-\beta)^{\frac{2}{\p}-1}T^{\frac{3}{2}-\frac{2}{\p}}}\right)^{\frac{\p}{2}}$
under $\left\{ \eta_{t}=\frac{(1-\beta)\eta}{\sqrt{T}}\right\} _{t=1}^{T}$
and $\beta\in\left[0,1\right)$ for $\eta\in\left(0,\frac{1-\beta}{2L}\right]$
, when $\p\in\left(\nicefrac{4}{3},2\right]$.

\end{itemize}
\end{thm}
Theorem \ref{thm:main-SGDM-exp} provides the first convergence rate
for $\SGDM$ under heavy-tailed noise, thereby extending the best
existing result due to \citep{NEURIPS2020_d3f5d4de}, which applies
to $\SGDM$ in the finite-variance case (i.e., $\p=2$). When $\beta=0$,
$\SGDM$ reduces to $\SGD$, and Theorem \ref{thm:main-SGDM-exp}
recovers the previous Theorem \ref{thm:main-SGD-exp}, albeit under
slightly different assumptions.

\section{A New Analysis Framework for Heavy-Tailed Stochastic Optimization\label{sec:analysis}}

In this section, we provide a new analysis framework for studying
stochastic gradient methods in heavy-tailed stochastic optimization.
To simplify the discussion and focus on the high-level idea, we assume
throughout that $\left\Vert \cdot\right\Vert =\left\Vert \cdot\right\Vert _{2}$
and $\left\{ \eta_{t}=\eta\right\} _{t=1}^{T}$. Moreover, for convex
optimization, we set the mirror map $\psi=\frac{1}{2}\left\Vert \cdot\right\Vert _{2}^{2}$
in $\SMD$ to recover $\SGD$ and take $r=0$ in the objective $F=f+r$.

Our proof relies critically on the following lemma, whose proof is
deferred to Appendix \ref{sec:proof-core-lemma}.
\begin{lem}
\label{lem:core}Under Assumption \ref{assu:oracle}, for any sequence
of random vectors $\left\{ \by_{t}\in\F_{t-1}\right\} _{t\in\N}$,
we have
\[
\E\left[\left|\sum_{s=1}^{t}\left\langle \bxi_{s},\by_{s}\right\rangle \right|^{\frac{\p}{2}}\right]\leq2^{1-\frac{\p}{2}}\sigma^{\frac{\p}{2}}\sqrt{\sum_{s=1}^{t}\E\left[\left\Vert \by_{s}\right\Vert ^{\p}\right]},\forall t\in\N.
\]
\end{lem}
We now describe the high-level idea behind our analysis. Our principle
is surprisingly simple and consists of only three steps:
\begin{eqnarray}
1.\text{ Sum the descent inequality.} & 2.\text{ Raise the bound to the power }\frac{\p}{2}. & 3.\text{ Take expectations.}\label{eq:main-idea}
\end{eqnarray}
To illustrate what this means, we use both convex and nonconvex optimization
as examples below.

\subsection{Convex Optimization}

Consider the classical analysis of $\SGD$, a common approach is to
first establish a descent inequality of the form (assuming $M=0$
in Assumption \ref{assu:SMD} and $\eta\lesssim\frac{1}{L}$):{\small 
\begin{equation}
\eta\Delta_{t+1}\lesssim\left\Vert \bx_{\star}-\bx_{t}\right\Vert ^{2}-\left\Vert \bx_{\star}-\bx_{t+1}\right\Vert ^{2}+\eta^{2}\left\Vert \bxi_{t}\right\Vert ^{2}+\eta\left\langle \bxi_{t},\bx_{\star}-\bx_{t}\right\rangle \text{ where }\Delta_{t}\defeq F(\bx_{t})-F_{\star}.\label{eq:main-cvx-descent}
\end{equation}
}Then, one takes expectations on both sides of (\ref{eq:main-cvx-descent})
and sums it from $t=1$ to $T$ to conclude.

However, under heavy-tailed noise, the second step breaks, since $\E[\left\Vert \bxi_{t}\right\Vert ^{2}]$
can be $+\infty$. Instead, our solution is to first sum (\ref{eq:main-cvx-descent})
from $t=1$ to $T$ before finally taking expectations. Although this
step seems simple, it allows us to perform additional operations.
Specifically, as described above, we raise the sum of the descent
inequalities to the power $\frac{\p}{2}$ to obtain (where $\Breg\defeq\left\Vert \bx_{\star}-\bx_{1}\right\Vert ^{2}$){\small 
\[
\left(\eta\sum_{t=1}^{T}\Delta_{t+1}\right)^{\frac{\p}{2}}\lesssim\left(\Breg+\eta^{2}\sum_{t=1}^{T}\left\Vert \bxi_{t}\right\Vert ^{2}+\eta\sum_{t=1}^{T}\left\langle \bxi_{t},\bx_{\star}-\bx_{t}\right\rangle \right)^{\frac{\p}{2}}\lesssim\Breg^{\frac{\p}{2}}+\eta^{\p}\sum_{t=1}^{T}\left\Vert \bxi_{t}\right\Vert ^{\p}+\left|\eta\sum_{t=1}^{T}\left\langle \bxi_{t},\bx_{\star}-\bx_{t}\right\rangle \right|^{\frac{\p}{2}},
\]
}where, in the second step, we repetitively use $(a+b)^{\frac{\p}{2}}\leq a^{\frac{\p}{2}}+b^{\frac{\p}{2}},\forall a,b\geq0$.
As one can see, the term $\left\Vert \bxi_{t}\right\Vert ^{\p}$ has
the correct exponent $\p$, and we can safely take expectations on
both sides of the above inequality. To handle the last term, we can
apply Lemma \ref{lem:core} with $\by_{t}=\eta(\bx_{\star}-\bx_{t})$.
As such, there is{\small 
\[
\E\left[\left(\eta\sum_{t=1}^{T}\Delta_{t+1}\right)^{\frac{\p}{2}}\right]\lesssim\Breg^{\frac{\p}{2}}+\sigma^{\p}\eta^{\p}T+\sigma^{\frac{\p}{2}}\sqrt{\sum_{t=1}^{T}\E\left[\eta^{\p}\left\Vert \bx_{\star}-\bx_{t}\right\Vert ^{\p}\right]}\lesssim\Breg^{\frac{\p}{2}}+\sigma^{\p}\eta^{\p}T+\Breg_{\max}^{\frac{\p}{4}}\sigma^{\frac{\p}{2}}\eta^{\frac{\p}{2}}\sqrt{T},
\]
}where $\Breg_{\max}\defeq\max_{t\in\left[T\right]}(\E[\left\Vert \bx_{\star}-\bx_{t}\right\Vert ^{\p}])^{\frac{2}{\p}}$.
Temporarily assuming $\Breg_{\max}\lesssim\Breg$, we complete the
proof by using $\Breg^{\frac{\p}{4}}\sigma^{\frac{\p}{2}}\eta^{\frac{\p}{2}}\sqrt{T}\lesssim\Breg^{\frac{\p}{2}}+\sigma^{\p}\eta^{\p}T$
and substituting the value of $\eta$ from Theorem \ref{thm:main-SMD-cvx-exp}.

We emphasize that the assumption $\Breg_{\max}\lesssim\Breg$ is made
only to simplify the discussion. It is not required in the full proofs.
For more details, we refer the reader to Appendix \ref{sec:proof-cvx}.

\subsection{Nonconvex Optimization}

For $\SGD$ in nonconvex optimization, the proof again follows the
same three-step strategy outlined above. Assuming $\eta\leq\frac{1}{L}$
and applying the first two steps described in (\ref{eq:main-idea}),
we obtain{\small 
\[
\left(\eta\sum_{t=1}^{T}\left\Vert \nabla F(\bx_{t})\right\Vert ^{2}\right)^{\frac{\p}{2}}\lesssim\Delta^{\frac{\p}{2}}+L^{\frac{\p}{2}}\eta^{\p}\sum_{t=1}^{T}\left\Vert \bxi_{t}\right\Vert ^{\p}+\left|\eta\sum_{t=1}^{T}\left\langle \bxi_{t},\nabla F(\bx_{t})\right\rangle \right|^{\frac{\p}{2}}\text{ where }\Delta\defeq F(\bx_{1})-F_{\star}.
\]
}Take expectations on both sides of the above inequality and apply
Lemma \ref{lem:core} with $\by_{t}=\eta\nabla F(\bx_{t})$ to have{\small 
\[
\E\left[\left(\eta\sum_{t=1}^{T}\left\Vert \nabla F(\bx_{t})\right\Vert ^{2}\right)^{\frac{\p}{2}}\right]\lesssim\Delta^{\frac{\p}{2}}+L^{\frac{\p}{2}}\sigma^{\p}\eta^{\p}T+\sigma^{\frac{\p}{2}}\sqrt{\sum_{t=1}^{T}\E\left[\eta^{\p}\left\Vert \nabla F(\bx_{t})\right\Vert ^{\p}\right]}.
\]
}Finally, using $\sum_{t=1}^{T}\left\Vert \nabla F(\bx_{t})\right\Vert ^{\p}\leq T^{\frac{2-\p}{2}}(\sum_{t=1}^{T}\left\Vert \nabla F(\bx_{t})\right\Vert ^{2})^{\frac{\p}{2}}$
(by Cauchy-Schwarz inequality) together with a self-bounding argument,
we obtain $\E[(\eta\sum_{t=1}^{T}\left\Vert \nabla F(\bx_{t})\right\Vert ^{2})^{\frac{\p}{2}}]\lesssim\Delta^{\frac{\p}{2}}+\sigma^{\p}\eta^{\frac{\p}{2}}T^{\frac{2-\p}{2}}+L^{\frac{\p}{2}}\sigma^{\p}\eta^{\p}T$.
Substituting the value of $\eta$ from Theorem \ref{thm:main-SGD-exp},
we complete the proof. 

For all omitted details, we kindly refer the reader to Appendix \ref{sec:proof-ncvx}.

\section{Limitations\label{sec:limitation}}

\paragraph{Convergence metric.}

As mentioned previously, all of our results are established for $\E[\left(\metric\right)^{\frac{\p}{2}}]$,
where $\metric$ is a commonly adopted convergence metric. Although
we believe that, in general, $\frac{\p}{2}$ is the best possible
exponent for ensuring a finite-time rate, it is important to understand
under what additional conditions it can be relaxed, which we leave
for future work.

\paragraph{$\protect\SGD$ in nonconvex optimization.}

Combining (\ref{eq:main-ncvx-metric}) with Theorem \ref{thm:main-SGD-exp}
shows that $\SGD$ guarantees only $\E[\sum_{t=1}^{T}\left\Vert \nabla F(\bx_{t})\right\Vert _{2}/T]\lesssim1/T^{\frac{\p-1}{2\p}}$.
However, this rate is strictly worse than the existing minimax lower
bound $1/T^{\frac{\p-1}{3\p-2}}$ mentioned in Section \ref{subsec:related-work}.
Closing this gap is also important in our view.

\bibliographystyle{alpha}
\bibliography{ref}

\newcommand{\etalchar}[1]{$^{#1}$}
\begin{thebibliography}{GEGGM21}

\bibitem[ACD{\etalchar{+}}23]{arjevani2023lower}
Yossi Arjevani, Yair Carmon, John~C Duchi, Dylan~J Foster, Nathan Srebro, and Blake Woodworth.
\newblock Lower bounds for non-convex stochastic optimization.
\newblock {\em Mathematical Programming}, 199(1-2):165--214, 2023.

\bibitem[ACS{\etalchar{+}}24]{ahn2024linear}
Kwangjun Ahn, Xiang Cheng, Minhak Song, Chulhee Yun, Ali Jadbabaie, and Suvrit Sra.
\newblock Linear attention is (maybe) all you need (to understand transformer optimization).
\newblock In {\em The Twelfth International Conference on Learning Representations}, 2024.

\bibitem[AP95]{AFST_1995_6_4_4_705_0}
Dominique Az\'e and Jean-Paul Penot.
\newblock Uniformly convex and uniformly smooth convex functions.
\newblock {\em Annales de la Facult\'e des sciences de Toulouse : Math\'ematiques}, Ser. 6, 4(4):705--730, 1995.

\bibitem[AYS{\etalchar{+}}25]{pmlr-v258-armacki25a}
Aleksandar Armacki, Shuhua Yu, Pranay Sharma, Gauri Joshi, Dragana Bajovic, Dusan Jakovetic, and Soummya Kar.
\newblock High-probability convergence bounds for online nonlinear stochastic gradient descent under heavy-tailed noise.
\newblock In Yingzhen Li, Stephan Mandt, Shipra Agrawal, and Emtiyaz Khan, editors, {\em Proceedings of The 28th International Conference on Artificial Intelligence and Statistics}, volume 258 of {\em Proceedings of Machine Learning Research}, pages 1774--1782. PMLR, 03--05 May 2025.

\bibitem[BBC{\etalchar{+}}19]{bauschke2019linear}
Heinz~H Bauschke, J{\'e}r{\^o}me Bolte, Jiawei Chen, Marc Teboulle, and Xianfu Wang.
\newblock On linear convergence of non-euclidean gradient methods without strong convexity and lipschitz gradient continuity.
\newblock {\em Journal of Optimization Theory and Applications}, 182(3):1068--1087, 2019.

\bibitem[BBT17]{doi:10.1287/moor.2016.0817}
Heinz~H. Bauschke, J\'{e}r\^{o}me Bolte, and Marc Teboulle.
\newblock A descent lemma beyond lipschitz gradient continuity: First-order methods revisited and applications.
\newblock {\em Mathematics of Operations Research}, 42(2):330--348, 2017.

\bibitem[BCN18]{doi:10.1137/16M1080173}
L\'{e}on Bottou, Frank~E. Curtis, and Jorge Nocedal.
\newblock Optimization methods for large-scale machine learning.
\newblock {\em SIAM Review}, 60(2):223--311, 2018.

\bibitem[BT03]{BECK2003167}
Amir Beck and Marc Teboulle.
\newblock Mirror descent and nonlinear projected subgradient methods for convex optimization.
\newblock {\em Operations Research Letters}, 31(3):167--175, 2003.

\bibitem[BWL24]{pmlr-v238-battash24a}
Barak Battash, Lior Wolf, and Ofir Lindenbaum.
\newblock Revisiting the noise model of stochastic gradient descent.
\newblock In Sanjoy Dasgupta, Stephan Mandt, and Yingzhen Li, editors, {\em Proceedings of The 27th International Conference on Artificial Intelligence and Statistics}, volume 238 of {\em Proceedings of Machine Learning Research}, pages 4780--4788. PMLR, 02--04 May 2024.

\bibitem[CM21]{cutkosky2021high}
Ashok Cutkosky and Harsh Mehta.
\newblock High-probability bounds for non-convex stochastic optimization with heavy tails.
\newblock {\em Advances in Neural Information Processing Systems}, 34:4883--4895, 2021.

\bibitem[DSSST10]{duchi2010composite}
John~C Duchi, Shai Shalev-Shwartz, Yoram Singer, and Ambuj Tewari.
\newblock Composite objective mirror descent.
\newblock In {\em COLT}, volume~10, pages 14--26. Citeseer, 2010.

\bibitem[DTdB22]{dragomir2022optimal}
Radu-Alexandru Dragomir, Adrien~B Taylor, Alexandre d’Aspremont, and J{\'e}r{\^o}me Bolte.
\newblock Optimal complexity and certification of bregman first-order methods.
\newblock {\em Mathematical Programming}, 194(1):41--83, 2022.

\bibitem[FHL25]{fatkhullin2025can}
Ilyas Fatkhullin, Florian H{\"u}bler, and Guanghui Lan.
\newblock Can sgd handle heavy-tailed noise?
\newblock {\em arXiv preprint arXiv:2508.04860}, 2025.

\bibitem[GEGGM21]{pmlr-v130-guille-escuret21a}
Charles Guille-Escuret, Manuela Girotti, Baptiste Goujaud, and Ioannis Mitliagkas.
\newblock A study of condition numbers for first-order optimization.
\newblock In Arindam Banerjee and Kenji Fukumizu, editors, {\em Proceedings of The 24th International Conference on Artificial Intelligence and Statistics}, volume 130 of {\em Proceedings of Machine Learning Research}, pages 1261--1269. PMLR, 13--15 Apr 2021.

\bibitem[GFJ15]{7330562}
Euhanna Ghadimi, Hamid~Reza Feyzmahdavian, and Mikael Johansson.
\newblock Global convergence of the heavy-ball method for convex optimization.
\newblock In {\em 2015 European Control Conference (ECC)}, pages 310--315, 2015.

\bibitem[GL13]{doi:10.1137/120880811}
Saeed Ghadimi and Guanghui Lan.
\newblock Stochastic first- and zeroth-order methods for nonconvex stochastic programming.
\newblock {\em SIAM Journal on Optimization}, 23(4):2341--2368, 2013.

\bibitem[GP18]{gutman2018unified}
David~H Gutman and Javier~F Pena.
\newblock A unified framework for bregman proximal methods: subgradient, gradient, and accelerated gradient schemes.
\newblock {\em arXiv preprint arXiv:1812.10198}, 2018.

\bibitem[GSD{\etalchar{+}}24]{pmlr-v235-gorbunov24a}
Eduard Gorbunov, Abdurakhmon Sadiev, Marina Danilova, Samuel Horv\'{a}th, Gauthier Gidel, Pavel Dvurechensky, Alexander Gasnikov, and Peter Richt\'{a}rik.
\newblock High-probability convergence for composite and distributed stochastic minimization and variational inequalities with heavy-tailed noise.
\newblock In Ruslan Salakhutdinov, Zico Kolter, Katherine Heller, Adrian Weller, Nuria Oliver, Jonathan Scarlett, and Felix Berkenkamp, editors, {\em Proceedings of the 41st International Conference on Machine Learning}, volume 235 of {\em Proceedings of Machine Learning Research}, pages 15951--16070. PMLR, 21--27 Jul 2024.

\bibitem[GZP{\etalchar{+}}21]{pmlr-v139-garg21b}
Saurabh Garg, Joshua Zhanson, Emilio Parisotto, Adarsh Prasad, Zico Kolter, Zachary Lipton, Sivaraman Balakrishnan, Ruslan Salakhutdinov, and Pradeep Ravikumar.
\newblock On proximal policy optimization's heavy-tailed gradients.
\newblock In Marina Meila and Tong Zhang, editors, {\em Proceedings of the 38th International Conference on Machine Learning}, volume 139 of {\em Proceedings of Machine Learning Research}, pages 3610--3619. PMLR, 18--24 Jul 2021.

\bibitem[HFH25]{pmlr-v258-hubler25a}
Florian H{\"u}bler, Ilyas Fatkhullin, and Niao He.
\newblock From gradient clipping to normalization for heavy tailed sgd.
\newblock In Yingzhen Li, Stephan Mandt, Shipra Agrawal, and Emtiyaz Khan, editors, {\em Proceedings of The 28th International Conference on Artificial Intelligence and Statistics}, volume 258 of {\em Proceedings of Machine Learning Research}, pages 2413--2421. PMLR, 03--05 May 2025.

\bibitem[KB14]{kingma2014adam}
Diederik~P Kingma and Jimmy Ba.
\newblock Adam: A method for stochastic optimization.
\newblock {\em arXiv preprint arXiv:1412.6980}, 2014.

\bibitem[Kiw97]{doi:10.1137/S0363012995281742}
Krzysztof~C. Kiwiel.
\newblock Proximal minimization methods with generalized bregman functions.
\newblock {\em SIAM Journal on Control and Optimization}, 35(4):1142--1168, 1997.

\bibitem[Lan12]{lan2012optimal}
Guanghui Lan.
\newblock An optimal method for stochastic composite optimization.
\newblock {\em Mathematical Programming}, 133(1):365--397, 2012.

\bibitem[Lan20]{lan2020first}
Guanghui Lan.
\newblock {\em First-order and stochastic optimization methods for machine learning}.
\newblock Springer, 2020.

\bibitem[LFN18]{doi:10.1137/16M1099546}
Haihao Lu, Robert~M. Freund, and Yurii Nesterov.
\newblock Relatively smooth convex optimization by first-order methods, and applications.
\newblock {\em SIAM Journal on Optimization}, 28(1):333--354, 2018.

\bibitem[LGY20]{NEURIPS2020_d3f5d4de}
Yanli Liu, Yuan Gao, and Wotao Yin.
\newblock An improved analysis of stochastic gradient descent with momentum.
\newblock In H.~Larochelle, M.~Ranzato, R.~Hadsell, M.F. Balcan, and H.~Lin, editors, {\em Advances in Neural Information Processing Systems}, volume~33, pages 18261--18271. Curran Associates, Inc., 2020.

\bibitem[Liu25]{liu2025online}
Zijian Liu.
\newblock Online convex optimization with heavy tails: Old algorithms, new regrets, and applications.
\newblock {\em arXiv preprint arXiv:2508.07473}, 2025.

\bibitem[Liu26a]{liu2026can}
Zijian Liu.
\newblock Can adaptive gradient methods converge under heavy-tailed noise? a case study of adagrad.
\newblock {\em arXiv preprint arXiv:2605.18694}, 2026.

\bibitem[Liu26b]{liu2026clipped}
Zijian Liu.
\newblock Clipped gradient methods for nonsmooth convex optimization under heavy-tailed noise: A refined analysis.
\newblock In {\em The Fourteenth International Conference on Learning Representations}, 2026.

\bibitem[Lu19]{doi:10.1287/ijoo.2018.0008}
Haihao Lu.
\newblock “relative continuity” for non-lipschitz nonsmooth convex optimization using stochastic (or deterministic) mirror descent.
\newblock {\em INFORMS Journal on Optimization}, 1(4):288--303, 2019.

\bibitem[LWZ24]{pmlr-v235-liu24bo}
Langqi Liu, Yibo Wang, and Lijun Zhang.
\newblock High-probability bound for non-smooth non-convex stochastic optimization with heavy tails.
\newblock In Ruslan Salakhutdinov, Zico Kolter, Katherine Heller, Adrian Weller, Nuria Oliver, Jonathan Scarlett, and Felix Berkenkamp, editors, {\em Proceedings of the 41st International Conference on Machine Learning}, volume 235 of {\em Proceedings of Machine Learning Research}, pages 32122--32138. PMLR, 21--27 Jul 2024.

\bibitem[LZ23]{liu2023stochastic}
Zijian Liu and Zhengyuan Zhou.
\newblock Stochastic nonsmooth convex optimization with heavy-tailed noises: High-probability bound, in-expectation rate and initial distance adaptation.
\newblock {\em arXiv preprint arXiv:2303.12277}, 2023.

\bibitem[LZ24]{liu2024revisiting}
Zijian Liu and Zhengyuan Zhou.
\newblock Revisiting the last-iterate convergence of stochastic gradient methods.
\newblock In {\em The Twelfth International Conference on Learning Representations}, 2024.

\bibitem[LZ25]{liu2025nonconvex}
Zijian Liu and Zhengyuan Zhou.
\newblock Nonconvex stochastic optimization under heavy-tailed noises: Optimal convergence without gradient clipping.
\newblock In {\em The Thirteenth International Conference on Learning Representations}, 2025.

\bibitem[LZZ23]{pmlr-v195-liu23c}
Zijian Liu, Jiawei Zhang, and Zhengyuan Zhou.
\newblock Breaking the lower bound with (little) structure: Acceleration in non-convex stochastic optimization with heavy-tailed noise.
\newblock In Gergely Neu and Lorenzo Rosasco, editors, {\em Proceedings of Thirty Sixth Conference on Learning Theory}, volume 195 of {\em Proceedings of Machine Learning Research}, pages 2266--2290. PMLR, 12--15 Jul 2023.

\bibitem[Nes84]{nesterov1984minimization}
Yurii~E Nesterov.
\newblock Minimization methods for nonsmooth convex and quasiconvex functions.
\newblock {\em Matekon}, 29(3):519--531, 1984.

\bibitem[NNEN23]{NEURIPS2023_4c454d34}
Ta~Duy Nguyen, Thien~H Nguyen, Alina Ene, and Huy Nguyen.
\newblock Improved convergence in high probability of clipped gradient methods with heavy tailed noise.
\newblock In A.~Oh, T.~Naumann, A.~Globerson, K.~Saenko, M.~Hardt, and S.~Levine, editors, {\em Advances in Neural Information Processing Systems}, volume~36, pages 24191--24222. Curran Associates, Inc., 2023.

\bibitem[NNG19]{necoara2019linear}
Ion Necoara, Yu~Nesterov, and Francois Glineur.
\newblock Linear convergence of first order methods for non-strongly convex optimization.
\newblock {\em Mathematical programming}, 175(1):69--107, 2019.

\bibitem[NY83]{nemirovskij1983problem}
Arkadi Nemirovski and David Yudin.
\newblock Problem complexity and method efficiency in optimization.
\newblock {\em Wiley-Interscience}, 1983.

\bibitem[Ora19]{orabona2019modern}
Francesco Orabona.
\newblock A modern introduction to online learning.
\newblock {\em arXiv preprint arXiv:1912.13213}, 2019.

\bibitem[PGKG24]{pmlr-v238-puchkin24a}
Nikita Puchkin, Eduard Gorbunov, Nickolay Kutuzov, and Alexander Gasnikov.
\newblock Breaking the heavy-tailed noise barrier in stochastic optimization problems.
\newblock In Sanjoy Dasgupta, Stephan Mandt, and Yingzhen Li, editors, {\em Proceedings of The 27th International Conference on Artificial Intelligence and Statistics}, volume 238 of {\em Proceedings of Machine Learning Research}, pages 856--864. PMLR, 02--04 May 2024.

\bibitem[Pin15]{10.15352/afa/06-4-1}
Iosif Pinelis.
\newblock {Best possible bounds of the von Bahr--Esseen type}.
\newblock {\em Annals of Functional Analysis}, 6(4):1 -- 29, 2015.

\bibitem[PMB13]{pmlr-v28-pascanu13}
Razvan Pascanu, Tomas Mikolov, and Yoshua Bengio.
\newblock On the difficulty of training recurrent neural networks.
\newblock In Sanjoy Dasgupta and David McAllester, editors, {\em Proceedings of the 30th International Conference on Machine Learning}, volume~28 of {\em Proceedings of Machine Learning Research}, pages 1310--1318, Atlanta, Georgia, USA, 17--19 Jun 2013. PMLR.

\bibitem[Pol63]{POLYAK1963864}
B.T. Polyak.
\newblock Gradient methods for the minimisation of functionals.
\newblock {\em USSR Computational Mathematics and Mathematical Physics}, 3(4):864--878, 1963.

\bibitem[Pol64]{POLYAK19641}
B.T. Polyak.
\newblock Some methods of speeding up the convergence of iteration methods.
\newblock {\em USSR Computational Mathematics and Mathematical Physics}, 4(5):1--17, 1964.

\bibitem[Pol87]{Solr-KOHA-OAI-TEST:19722}
Boris~T. Polyak.
\newblock {\em Introduction to optimization}.
\newblock New York, Optimization Software, 1987.

\bibitem[PPS24]{parletta2024improved}
Daniela~Angela Parletta, Andrea Paudice, and Saverio Salzo.
\newblock An improved analysis of the clipped stochastic subgradient method under heavy-tailed noise.
\newblock {\em arXiv preprint arXiv:2410.00573}, 2024.

\bibitem[RM51]{10.1214/aoms/1177729586}
Herbert Robbins and Sutton Monro.
\newblock {A Stochastic Approximation Method}.
\newblock {\em The Annals of Mathematical Statistics}, 22(3):400 -- 407, 1951.

\bibitem[SDG{\etalchar{+}}23]{pmlr-v202-sadiev23a}
Abdurakhmon Sadiev, Marina Danilova, Eduard Gorbunov, Samuel Horv\'{a}th, Gauthier Gidel, Pavel Dvurechensky, Alexander Gasnikov, and Peter Richt\'{a}rik.
\newblock High-probability bounds for stochastic optimization and variational inequalities: the case of unbounded variance.
\newblock In Andreas Krause, Emma Brunskill, Kyunghyun Cho, Barbara Engelhardt, Sivan Sabato, and Jonathan Scarlett, editors, {\em Proceedings of the 40th International Conference on Machine Learning}, volume 202 of {\em Proceedings of Machine Learning Research}, pages 29563--29648. PMLR, 23--29 Jul 2023.

\bibitem[SLY25]{JMLR:v26:24-1991}
Tao Sun, Xinwang Liu, and Kun Yuan.
\newblock Revisiting gradient normalization and clipping for nonconvex sgd under heavy-tailed noise: Necessity, sufficiency, and acceleration.
\newblock {\em Journal of Machine Learning Research}, 26(237):1--42, 2025.

\bibitem[SSG19]{pmlr-v97-simsekli19a}
Umut Simsekli, Levent Sagun, and Mert Gurbuzbalaban.
\newblock A tail-index analysis of stochastic gradient noise in deep neural networks.
\newblock In Kamalika Chaudhuri and Ruslan Salakhutdinov, editors, {\em Proceedings of the 36th International Conference on Machine Learning}, volume~97 of {\em Proceedings of Machine Learning Research}, pages 5827--5837. PMLR, 09--15 Jun 2019.

\bibitem[vBE65]{8628f855-d581-37e5-a68f-5014db61e4b7}
Bengt von Bahr and Carl-Gustav Esseen.
\newblock Inequalities for the rth absolute moment of a sum of random variables, $1\leq r\leq 2$.
\newblock {\em The Annals of Mathematical Statistics}, 36(1):299--303, 1965.

\bibitem[VYB{\etalchar{+}}22]{pmlr-v178-vural22a}
Nuri~Mert Vural, Lu~Yu, Krishna Balasubramanian, Stanislav Volgushev, and Murat~A Erdogdu.
\newblock Mirror descent strikes again: Optimal stochastic convex optimization under infinite noise variance.
\newblock In Po-Ling Loh and Maxim Raginsky, editors, {\em Proceedings of Thirty Fifth Conference on Learning Theory}, volume 178 of {\em Proceedings of Machine Learning Research}, pages 65--102. PMLR, 02--05 Jul 2022.

\bibitem[WFZ{\etalchar{+}}23]{NEURIPS2023_7ac19fdc}
Bohan Wang, Jingwen Fu, Huishuai Zhang, Nanning Zheng, and Wei Chen.
\newblock Closing the gap between the upper bound and lower bound of adam\textquotesingle s iteration complexity.
\newblock In A.~Oh, T.~Naumann, A.~Globerson, K.~Saenko, M.~Hardt, and S.~Levine, editors, {\em Advances in Neural Information Processing Systems}, volume~36, pages 39006--39032. Curran Associates, Inc., 2023.

\bibitem[WGZ{\etalchar{+}}21]{NEURIPS2021_9cdf2656}
Hongjian Wang, Mert Gurbuzbalaban, Lingjiong Zhu, Umut Simsekli, and Murat~A Erdogdu.
\newblock Convergence rates of stochastic gradient descent under infinite noise variance.
\newblock In M.~Ranzato, A.~Beygelzimer, Y.~Dauphin, P.S. Liang, and J.~Wortman Vaughan, editors, {\em Advances in Neural Information Processing Systems}, volume~34, pages 18866--18877. Curran Associates, Inc., 2021.

\bibitem[YFL23]{pmlr-v195-yue23a}
Pengyun Yue, Cong Fang, and Zhouchen Lin.
\newblock On the lower bound of minimizing polyak-Łojasiewicz functions.
\newblock In Gergely Neu and Lorenzo Rosasco, editors, {\em Proceedings of Thirty Sixth Conference on Learning Theory}, volume 195 of {\em Proceedings of Machine Learning Research}, pages 2948--2968. PMLR, 12--15 Jul 2023.

\bibitem[ZC22]{NEURIPS2022_349956de}
Jiujia Zhang and Ashok Cutkosky.
\newblock Parameter-free regret in high probability with heavy tails.
\newblock In S.~Koyejo, S.~Mohamed, A.~Agarwal, D.~Belgrave, K.~Cho, and A.~Oh, editors, {\em Advances in Neural Information Processing Systems}, volume~35, pages 8000--8012. Curran Associates, Inc., 2022.

\bibitem[ZDGP21]{doi:10.1287/moor.2019.1047}
Hui Zhang, Yu-Hong Dai, Lei Guo, and Wei Peng.
\newblock Proximal-like incremental aggregated gradient method with linear convergence under bregman distance growth conditions.
\newblock {\em Mathematics of Operations Research}, 46(1):61--81, 2021.

\bibitem[ZG25]{doi:10.1137/24M1717762}
Moslem Zamani and Fran\c{c}ois Glineur.
\newblock Exact convergence rate of the last iterate in subgradient methods.
\newblock {\em SIAM Journal on Optimization}, 35(3):2182--2201, 2025.

\bibitem[ZKV{\etalchar{+}}20]{NEURIPS2020_b05b57f6}
Jingzhao Zhang, Sai~Praneeth Karimireddy, Andreas Veit, Seungyeon Kim, Sashank Reddi, Sanjiv Kumar, and Suvrit Sra.
\newblock Why are adaptive methods good for attention models?
\newblock In H.~Larochelle, M.~Ranzato, R.~Hadsell, M.F. Balcan, and H.~Lin, editors, {\em Advances in Neural Information Processing Systems}, volume~33, pages 15383--15393. Curran Associates, Inc., 2020.

\bibitem[ZSPSH20]{NEURIPS2020_b67fb336}
Yihan Zhou, Victor Sanches~Portella, Mark Schmidt, and Nicholas Harvey.
\newblock Regret bounds without lipschitz continuity: Online learning with relative-lipschitz losses.
\newblock In H.~Larochelle, M.~Ranzato, R.~Hadsell, M.F. Balcan, and H.~Lin, editors, {\em Advances in Neural Information Processing Systems}, volume~33, pages 15823--15833. Curran Associates, Inc., 2020.

\bibitem[Zǎ83]{ZALINESCU1983344}
C.~Zǎlinescu.
\newblock On uniformly convex functions.
\newblock {\em Journal of Mathematical Analysis and Applications}, 95(2):344--374, 1983.

\end{thebibliography}

\clearpage

\appendix

\section{Proof of Lemma \ref{lem:core}\label{sec:proof-core-lemma}}

\begin{proof}
Note that $\left\{ X_{t}\defeq\left\langle \bxi_{t},\by_{t}\right\rangle \right\} _{t\in\N}$
forms a martingale difference sequence adapted to $\left\{ \F_{t}\right\} _{t\in\N}$.
In addition, it satisfies
\begin{equation}
\E\left[\left|X_{t}\right|^{\p}\right]\overset{(a)}{=}\E\left[\E_{t-1}\left[\left|X_{t}\right|^{\p}\right]\right]\overset{(b)}{\leq}\E\left[\left\Vert \by_{t}\right\Vert ^{\p}\E_{t-1}\left[\left\Vert \bxi_{t}\right\Vert _{\star}^{\p}\right]\right]\overset{(c)}{\leq}\E\left[\sigma^{\p}\left\Vert \by_{t}\right\Vert ^{\p}\right]=\sigma^{\p}\E\left[\left\Vert \by_{t}\right\Vert ^{\p}\right],\label{eq:core-1}
\end{equation}
where $(a)$ is due to the tower rule, $(b)$ is by $\left|\left\langle \bxi_{t},\by_{t}\right\rangle \right|\leq\left\Vert \bxi_{t}\right\Vert _{\star}\left\Vert \by_{t}\right\Vert $
and $\by_{t}\in\F_{t-1}$, and $(c)$ follows from Assumption \ref{assu:oracle}.
Therefore, given $t\in\N$,
\begin{align*}
\E\left[\left|\sum_{s=1}^{t}\left\langle \bxi_{s},\by_{s}\right\rangle \right|^{\frac{\p}{2}}\right] & =\E\left[\left|\sum_{s=1}^{t}X_{s}\right|^{\frac{\p}{2}}\right]\overset{(d)}{\leq}\sqrt{\E\left[\left|\sum_{s=1}^{t}X_{s}\right|^{\p}\right]}\\
 & \overset{(e)}{\leq}\sqrt{2^{2-\p}\sum_{s=1}^{t}\E\left[\left|X_{s}\right|^{\p}\right]}\overset{(\ref{eq:core-1})}{\leq}\sqrt{2^{2-\p}\sigma^{\p}\sum_{s=1}^{t}\E\left[\left\Vert \by_{s}\right\Vert ^{\p}\right]},
\end{align*}
where $(d)$ is by H\"{o}lder's inequality and $(e)$ is due to the
von Bahr-Esseen inequality \citep{8628f855-d581-37e5-a68f-5014db61e4b7}
for martingales (see, e.g., \citep{10.15352/afa/06-4-1}).
\end{proof}

\section{Convex Optimization\label{sec:proof-cvx}}

\paragraph{Notation.}

Given $\bx\in\dom F$, we let $\left[\bx\right]_{\star}\in\X_{\star}$
denote the point satisfying $\Breg_{\psi}(\left[\bx\right]_{\star},\bx)=\Breg_{\psi}(\X_{\star},\bx)$,
where we recall that $\Breg_{\psi}(\X_{\star},\bx)=\inf_{\bx_{\star}\in\X_{\star}}\Breg_{\psi}(\bx_{\star},\bx)$.
The existence of $\left[\bx\right]_{\star}$ is guaranteed by the
conditions on $\psi$ and $F$. Moreover, for $\SMD$ (resp. $\ASMD$),
we denote by $\Delta_{t}\defeq F(\bx_{t})-F_{\star}$ ($\Delta_{t}\defeq F(\bz_{t})-F_{\star}$)
for any $t\in\left[T+1\right]$ in the proof.

\subsection{Full Theorems for $\protect\SMD$ and Proofs}

\paragraph{The case $\mu=0$.}

We first consider general convex optimization.
\begin{thm}[Full statement of Theorem \ref{thm:main-SMD-cvx-exp}]
\label{thm:SMD-cvx-exp}Under Assumptions \ref{assu:oracle} (with
$h=f$) and \ref{assu:SMD} (with $\mu=0$), for any $T\in\N$ and
$\left\{ \eta_{t}\leq\frac{1}{2L}\right\} _{t=1}^{T}$, $\SMD$ (Algorithm
\ref{alg:SMD}) guarantees that

\begin{itemize}[leftmargin=*]

\item $\E\left[\left(\frac{\sum_{t=1}^{T}\eta_{t}\Delta_{t+1}}{\sum_{t=1}^{T}\eta_{t}}\right)^{\frac{\p}{2}}\right]\lesssim\left(\frac{\Breg_{\psi}(\X_{\star},\bx_{1})}{\sum_{t=1}^{T}\eta_{t}}\right)^{\frac{\p}{2}}+\left(\frac{\sum_{t=1}^{T}M^{2}\eta_{t}^{2}}{\sum_{t=1}^{T}\eta_{t}}\right)^{\frac{\p}{2}}+\sum_{t=1}^{T}\left(\frac{\sigma^{2}\eta_{t}^{2}}{\sum_{t=1}^{T}\eta_{t}}\right)^{\frac{\p}{2}}.$

\item $\E\left[\left(\Delta_{T+1}\right)^{\frac{\p}{2}}\right]\lesssim\left(\frac{\Breg_{\psi}(\X_{\star},\bx_{1})}{\sum_{t=1}^{T}\eta_{t}}\right)^{\frac{\p}{2}}+\left(\sum_{t=1}^{T}\frac{M^{2}\eta_{t}^{2}}{\sum_{s=t}^{T}\eta_{s}}\right)^{\frac{\p}{2}}+\sum_{t=1}^{T}\left(\frac{\sigma^{2}\eta_{t}^{2}}{\sum_{s=t}^{T}\eta_{s}}\right)^{\frac{\p}{2}}.$

\end{itemize}In particular, by plugging in different choices of $\left\{ \eta_{t}\right\} _{t=1}^{T}$,
we recover Theorem \ref{thm:main-SMD-cvx-exp}.
\end{thm}
\begin{proof}
In the following, we write $\Breg_{t}\defeq\Breg_{\psi}(\X_{\star},\bx_{t})=\Breg_{\psi}(\left[\bx_{t}\right]_{\star},\bx_{t}),\forall t\in\left[T+1\right]$.
By $1$-strong convexity of $\psi$ on $\dom F$, we have
\begin{equation}
\frac{\left\Vert \left[\bx_{t}\right]_{\star}-\bx_{t}\right\Vert ^{2}}{2}\leq\Breg_{t},\forall t\in\left[T+1\right].\label{eq:SMD-cvx-exp-mirror}
\end{equation}

\subparagraph{Average-iterate convergence.}

We invoke the first inequality in Lemma \ref{lem:SMD-core} to have
(although $\gamma_{t}=\eta_{t},\forall t\in\left[T\right]$ when $\mu=0$,
we keep $\gamma_{s}$ in the derivation until the last step), for
any $t\in\left[T\right]$,
\begin{align*}
\frac{\gamma_{t+1}}{\eta_{t+1}}\Breg_{t+1}+\sum_{s=1}^{t}\frac{\gamma_{s}\Delta_{s+1}}{2} & \leq\Breg_{1}+\sum_{s=1}^{t}\eta_{s}\gamma_{s}\left(M+\left\Vert \bxi_{s}\right\Vert _{\star}\right)^{2}+\sum_{s=1}^{t}\gamma_{s}\left\langle \bxi_{s},\left[\bx_{s}\right]_{\star}-\bx_{s}\right\rangle \\
 & \leq\Breg_{1}+2\sum_{s=1}^{t}\eta_{s}\gamma_{s}\left(M^{2}+\left\Vert \bxi_{s}\right\Vert _{\star}^{2}\right)+\left|\sum_{s=1}^{t}\gamma_{s}\left\langle \bxi_{s},\left[\bx_{s}\right]_{\star}-\bx_{s}\right\rangle \right|,
\end{align*}
which implies that
\begin{align}
 & \left(\frac{\gamma_{t+1}}{\eta_{t+1}}\Breg_{t+1}+\sum_{s=1}^{t}\frac{\gamma_{s}\Delta_{s+1}}{2}\right)^{\frac{\p}{2}}\leq\left(\Breg_{1}+2\sum_{s=1}^{t}\eta_{s}\gamma_{s}\left(M^{2}+\left\Vert \bxi_{s}\right\Vert _{\star}^{2}\right)+\left|\sum_{s=1}^{t}\gamma_{s}\left\langle \bxi_{s},\left[\bx_{s}\right]_{\star}-\bx_{s}\right\rangle \right|\right)^{\frac{\p}{2}}\nonumber \\
\leq & \Breg_{1}^{\frac{\p}{2}}+2^{\frac{\p}{2}}M^{\p}\left(\sum_{s=1}^{t}\eta_{s}\gamma_{s}\right)^{\frac{\p}{2}}+2^{\frac{\p}{2}}\sum_{s=1}^{t}\eta_{s}^{\frac{\p}{2}}\gamma_{s}^{\frac{\p}{2}}\left\Vert \bxi_{s}\right\Vert _{\star}^{\p}+\left|\sum_{s=1}^{t}\gamma_{s}\left\langle \bxi_{s},\left[\bx_{s}\right]_{\star}-\bx_{s}\right\rangle \right|^{\frac{\p}{2}},\label{eq:SMD-cvx-exp-1}
\end{align}
where the second step is by repeatedly using $(a+b)^{\frac{\p}{2}}\leq a^{\frac{\p}{2}}+b^{\frac{\p}{2}},\forall a,b\geq0$
when $\p\in\left[0,2\right]$.

Next, we take expectations on both sides of (\ref{eq:SMD-cvx-exp-1})
to obtain
\begin{align}
 & \E\left[\left(\frac{\gamma_{t+1}}{\eta_{t+1}}\Breg_{t+1}+\sum_{s=1}^{t}\frac{\gamma_{s}\Delta_{s+1}}{2}\right)^{\frac{\p}{2}}\right]\nonumber \\
\leq & \Breg_{1}^{\frac{\p}{2}}+2^{\frac{\p}{2}}M^{\p}\left(\sum_{s=1}^{t}\eta_{s}\gamma_{s}\right)^{\frac{\p}{2}}+2^{\frac{\p}{2}}\sigma^{\p}\sum_{s=1}^{t}\eta_{s}^{\frac{\p}{2}}\gamma_{s}^{\frac{\p}{2}}+\E\left[\left|\sum_{s=1}^{t}\gamma_{s}\left\langle \bxi_{s},\left[\bx_{s}\right]_{\star}-\bx_{s}\right\rangle \right|^{\frac{\p}{2}}\right]\nonumber \\
\overset{(a)}{\leq} & \Breg_{1}^{\frac{\p}{2}}+2^{\frac{\p}{2}}M^{\p}\left(\sum_{s=1}^{t}\eta_{s}\gamma_{s}\right)^{\frac{\p}{2}}+2^{\frac{\p}{2}}\sigma^{\p}\sum_{s=1}^{t}\eta_{s}^{\frac{\p}{2}}\gamma_{s}^{\frac{\p}{2}}+2^{1-\frac{\p}{4}}\sigma^{\frac{\p}{2}}\sqrt{\sum_{s=1}^{t}\E\left[\gamma_{s}^{\p}\Breg_{s}^{\frac{\p}{2}}\right]}\nonumber \\
\overset{(b)}{\leq} & \Breg_{1}^{\frac{\p}{2}}+2^{\frac{\p}{2}}M^{\p}\left(\sum_{s=1}^{t}\eta_{s}\gamma_{s}\right)^{\frac{\p}{2}}+\left(2^{\frac{\p}{2}}+2^{1-\frac{\p}{2}}\right)\sigma^{\p}\sum_{s=1}^{t}\eta_{s}^{\frac{\p}{2}}\gamma_{s}^{\frac{\p}{2}}+\frac{\max_{s\in\left[t\right]}\E\left[\left(\frac{\gamma_{s}}{\eta_{s}}\Breg_{s}\right)^{\frac{\p}{2}}\right]}{2},\label{eq:SMD-cvx-exp-2}
\end{align}
where $(a)$ is by applying Lemma \ref{lem:core} with $\by_{t}=\gamma_{t}\left(\left[\bx_{t}\right]_{\star}-\bx_{t}\right)$
to have
\[
\E\left[\left|\sum_{s=1}^{t}\gamma_{s}\left\langle \bxi_{s},\left[\bx_{s}\right]_{\star}-\bx_{s}\right\rangle \right|^{\frac{\p}{2}}\right]\leq2^{1-\frac{\p}{2}}\sigma^{\frac{\p}{2}}\sqrt{\sum_{s=1}^{t}\E\left[\gamma_{s}^{\p}\left\Vert \left[\bx_{s}\right]_{\star}-\bx_{s}\right\Vert ^{\p}\right]}\overset{(\ref{eq:SMD-cvx-exp-mirror})}{\leq}2^{1-\frac{\p}{4}}\sigma^{\frac{\p}{2}}\sqrt{\sum_{s=1}^{t}\E\left[\gamma_{s}^{\p}\Breg_{s}^{\frac{\p}{2}}\right]},
\]
and $(b)$ follows from
\[
\sqrt{\sum_{s=1}^{t}\E\left[\gamma_{s}^{\p}\Breg_{s}^{\frac{\p}{2}}\right]}\leq\sqrt{\sum_{s=1}^{t}\eta_{s}^{\frac{\p}{2}}\gamma_{s}^{\frac{\p}{2}}\max_{s\in\left[t\right]}\E\left[\left(\frac{\gamma_{s}}{\eta_{s}}\Breg_{s}\right)^{\frac{\p}{2}}\right]}\leq\frac{\sigma^{\frac{\p}{2}}\sum_{s=1}^{t}\eta_{s}^{\frac{\p}{2}}\gamma_{s}^{\frac{\p}{2}}}{2^{\frac{\p}{4}}}+\frac{\max_{s\in\left[t\right]}\E\left[\left(\frac{\gamma_{s}}{\eta_{s}}\Breg_{s}\right)^{\frac{\p}{2}}\right]}{2^{2-\frac{\p}{4}}\sigma^{\frac{\p}{2}}},
\]
in which the last step is due to AM-GM inequality.

By an inductive argument, (\ref{eq:SMD-cvx-exp-2}) implies that
\[
\E\left[\left(\frac{\gamma_{t+1}}{\eta_{t+1}}\Breg_{t+1}\right)^{\frac{\p}{2}}\right]\lesssim\Breg_{1}^{\frac{\p}{2}}+M^{\p}\left(\sum_{s=1}^{t}\eta_{s}\gamma_{s}\right)^{\frac{\p}{2}}+\sigma^{\p}\sum_{s=1}^{t}\eta_{s}^{\frac{\p}{2}}\gamma_{s}^{\frac{\p}{2}},\forall t\in\left\{ 0\right\} \cup\left[T\right].
\]
Therefore, we invoke (\ref{eq:SMD-cvx-exp-2}) for $t=T$ to obtain
\begin{equation}
\E\left[\left(\frac{\gamma_{T+1}}{\eta_{T+1}}\Breg_{T+1}+\sum_{t=1}^{T}\frac{\gamma_{t}\Delta_{t+1}}{2}\right)^{\frac{\p}{2}}\right]\lesssim\Breg_{1}^{\frac{\p}{2}}+M^{\p}\left(\sum_{t=1}^{T}\eta_{t}\gamma_{t}\right)^{\frac{\p}{2}}+\sigma^{\p}\sum_{t=1}^{T}\eta_{t}^{\frac{\p}{2}}\gamma_{t}^{\frac{\p}{2}},\label{eq:SMD-cvx-core-3}
\end{equation}
Finally, we drop the nonnegative term $\frac{\gamma_{T+1}}{\eta_{T+1}}\Breg_{T+1}$,
then divide both sides by $\left(\frac{\sum_{t=1}^{T}\gamma}{2}\right)^{\frac{\p}{2}}$,
and use $\gamma_{t}=\eta_{t},\forall t\in\left[T\right]$ in this
case to conclude.

\subparagraph{Last-iterate convergence.}

We invoke the second inequality in Lemma \ref{lem:SMD-core} and then
follow essentially the same proof of (\ref{eq:SMD-cvx-core-3}) to
obtain
\[
\E\left[\left(\eta_{T}\Delta_{T+1}\right)^{\frac{\p}{2}}\right]\lesssim\left(v_{0}\Breg_{1}\right)^{\frac{\p}{2}}+M^{\p}\left(\sum_{t=1}^{T}v_{t}\eta_{t}^{2}\right)^{\frac{\p}{2}}+\sigma^{\p}\sum_{t=1}^{T}v_{t}^{\frac{\p}{2}}\eta_{t}^{\p},
\]
where $v_{t}=\frac{\eta_{T}}{\sum_{s=t}^{T}\eta_{s}},\forall t\in\left[T\right]$
and $v_{0}=v_{1}$. Finally, we divide both sides by $\eta_{T}^{\frac{\p}{2}}$
to conclude.
\end{proof}

\paragraph{The case $\mu>0$.}

We next consider a positive parameter in the Bregman growth condition.
\begin{thm}[Full statement of Theorem \ref{thm:main-SMD-str-exp}]
\label{thm:SMD-str-exp}Under Assumptions \ref{assu:oracle} (with
$h=f$) and \ref{assu:SMD} (with $\mu>0$), for any $T\in\N$ and
$\left\{ \eta_{t}\leq\frac{1}{2L}\right\} _{t=1}^{T}$, $\SMD$ (Algorithm
\ref{alg:SMD}) guarantees that, for $\left\{ P_{t}\defeq\prod_{s=1}^{t}\left(1+\frac{\mu\eta_{s}}{2}\right)\right\} _{t=0}^{T}$
and $T^{\prime}\defeq\left\lceil \frac{T}{2}\right\rceil $,

\begin{itemize}[leftmargin=*]

\item $\E\left[\left(\frac{\sum_{t=1}^{T}\gamma_{t}\Delta_{t+1}}{\sum_{t=1}^{T}\gamma_{t}}\right)^{\frac{\p}{2}}\right]\lesssim\left(\frac{\mu\Breg_{\psi}(\X_{\star},\bx_{1})}{P_{T}-1}\right)^{\frac{\p}{2}}+\left(\frac{\sum_{t=1}^{T}\mu M^{2}\eta_{t}^{2}P_{t-1}}{P_{T}-1}\right)^{\frac{\p}{2}}+\sum_{t=1}^{T}\left(\frac{\mu\sigma^{2}\eta_{t}^{2}P_{t-1}}{P_{T}-1}\right)^{\frac{\p}{2}}$.

\item $\E\left[\left(\Breg_{\psi}(\X_{\star},\bx_{T+1})\right)^{\frac{\p}{2}}\right]\lesssim\left(\frac{\Breg_{\psi}(\X_{\star},\bx_{1})}{\prod_{t=1}^{T}\left(1+\frac{\mu\eta_{t}}{2}\right)}\right)^{\frac{\p}{2}}+\left(\frac{\sum_{t=1}^{T}M^{2}\eta_{t}^{2}}{\prod_{s=t}^{T}\left(1+\frac{\mu\eta_{s}}{2}\right)}\right)^{\frac{\p}{2}}+\sum_{t=1}^{T}\left(\frac{\sigma^{2}\eta_{t}^{2}}{\prod_{s=t}^{T}\left(1+\frac{\mu\eta_{s}}{2}\right)}\right)^{\frac{\p}{2}}$.

\item $\E\left[\left(\Delta_{T+1}\right)^{\frac{\p}{2}}\right]\lesssim\left(\frac{\Breg_{\psi}(\X_{\star},\bx_{1})}{\prod_{t=1}^{T^{\prime}}\left(1+\frac{\mu\eta_{t}}{2}\right)\left(\sum_{t=T^{\prime}+1}^{T}\eta_{t}\right)}\right)^{\frac{\p}{2}}+\left(\frac{\sum_{t=1}^{T^{\prime}}M^{2}\eta_{t}^{2}}{\prod_{s=t}^{T^{\prime}}\left(1+\frac{\mu\eta_{s}}{2}\right)\left(\sum_{t=T^{\prime}+1}^{T}\eta_{t}\right)}\right)^{\frac{\p}{2}}+\left(\sum_{t=T^{\prime}+1}^{T}\frac{M^{2}\eta_{t}^{2}}{\sum_{s=t}^{T}\eta_{s}}\right)^{\frac{\p}{2}}\ +\sum_{t=1}^{T^{\prime}}\left(\frac{\sigma^{2}\eta_{t}^{2}}{\prod_{s=t}^{T^{\prime}}\left(1+\frac{\mu\eta_{s}}{2}\right)\left(\sum_{t=T^{\prime}+1}^{T}\eta_{t}\right)}\right)^{\frac{\p}{2}}+\sum_{t=T^{\prime}+1}^{T}\left(\frac{\sigma^{2}\eta_{t}^{2}}{\sum_{s=t}^{T}\eta_{s}}\right)^{\frac{\p}{2}}$.

\end{itemize}In particular, by plugging in different choices of $\left\{ \eta_{t}\right\} _{t=1}^{T}$,
we recover Theorem \ref{thm:main-SMD-str-exp}.
\end{thm}
\begin{proof}
In the following, we write $\Breg_{t}\defeq\Breg_{\psi}(\X_{\star},\bx_{t})=\Breg_{\psi}(\left[\bx_{t}\right]_{\star},\bx_{t}),\forall t\in\left[T+1\right]$.

\subparagraph{Average-iterate convergence and convergence in distance.}

Note that (\ref{eq:SMD-cvx-core-3}) is still applicable when $\mu>0$.
Therefore, we have
\[
\E\left[\left(\frac{\gamma_{T+1}}{\eta_{T+1}}\Breg_{T+1}+\sum_{t=1}^{T}\frac{\gamma_{t}\Delta_{t+1}}{2}\right)^{\frac{\p}{2}}\right]\lesssim\Breg_{1}^{\frac{\p}{2}}+M^{\p}\left(\sum_{t=1}^{T}\eta_{t}\gamma_{t}\right)^{\frac{\p}{2}}+\sigma^{\p}\sum_{t=1}^{T}\eta_{t}^{\frac{\p}{2}}\gamma_{t}^{\frac{\p}{2}},
\]
where $\gamma_{t}=\eta_{t}\prod_{s=1}^{t-1}\left(1+\frac{\mu\eta_{s}}{2}\right),\forall t\in\left[T+1\right]$.
We observe that
\[
\sum_{t=1}^{T}\gamma_{t}=\sum_{t=1}^{T}\eta_{t}\prod_{s=1}^{t-1}\left(1+\frac{\mu\eta_{s}}{2}\right)=\frac{2}{\mu}\sum_{t=1}^{T}\left[\prod_{s=1}^{t}\left(1+\frac{\mu\eta_{s}}{2}\right)-\prod_{s=1}^{t-1}\left(1+\frac{\mu\eta_{s}}{2}\right)\right]=\frac{2}{\mu}\left(\frac{\gamma_{T+1}}{\eta_{T+1}}-1\right),
\]
which implies that
\[
\E\left[\left(\frac{\gamma_{T+1}}{\eta_{T+1}}\Breg_{T+1}+\frac{\left(\frac{\gamma_{T+1}}{\eta_{T+1}}-1\right)}{\mu}\cdot\frac{\sum_{t=1}^{T}\gamma_{t}\Delta_{t+1}}{\sum_{t=1}^{T}\gamma_{t}}\right)^{\frac{\p}{2}}\right]\lesssim\Breg_{1}^{\frac{\p}{2}}+M^{\p}\left(\sum_{t=1}^{T}\eta_{t}\gamma_{t}\right)^{\frac{\p}{2}}+\sigma^{\p}\sum_{t=1}^{T}\eta_{t}^{\frac{\p}{2}}\gamma_{t}^{\frac{\p}{2}}
\]
Finally, we apply some direct calculations to conclude.

\subparagraph{Last-iterate convergence.}

To prove the last-iterate rate, we denote by $T^{\prime}\defeq\left\lceil \frac{T}{2}\right\rceil $
and observe that $\bx_{T+1}$ can be viewed as the output of Algorithm
\ref{alg:SMD} with the initial point $\bx_{T^{\prime}+1}$ after
$\left\lfloor \frac{T}{2}\right\rfloor $ iterations. Therefore, by
the second inequality in Theorem \ref{thm:SMD-cvx-exp}, we have
\[
\E\left[\left(\Delta_{T+1}\right)^{\frac{\p}{2}}\right]\lesssim\frac{\E\left[\Breg_{T^{\prime}+1}^{\frac{\p}{2}}\right]}{\left(\sum_{t=T^{\prime}+1}^{T}\eta_{t}\right)^{\frac{\p}{2}}}+\left(\sum_{t=T^{\prime}+1}^{T}\frac{M^{2}\eta_{t}^{2}}{\sum_{s=t}^{T}\eta_{s}}\right)^{\frac{\p}{2}}+\sum_{t=T^{\prime}+1}^{T}\frac{\sigma^{\p}\eta_{t}^{\p}}{\left(\sum_{s=t}^{T}\eta_{s}\right)^{\frac{\p}{2}}}.
\]
Finally, we apply the bound for $\E\left[\left(\Breg_{\psi}(\X_{\star},\bx_{T^{\prime}+1})\right)^{\frac{\p}{2}}\right]$
above to conclude.
\end{proof}

\subsection{Full Theorems for $\protect\ASMD$ and Proofs}
\begin{thm}[Full statement of Theorem \ref{thm:main-ASMD-cvx-exp}]
\label{thm:ASMD-cvx-exp}Under Assumptions \ref{assu:oracle} (with
$h=f$) and \ref{assu:ASMD}, for any $T\in\N$, $\left\{ w_{t}\in\left[0,1\right]\right\} _{t=1}^{T}$
satisfying $w_{1}=1$, and $\left\{ \eta_{t}\leq\frac{1}{2w_{t}L}\right\} _{t=1}^{T}$
satisfying $\frac{\eta_{t}}{w_{t}}\geq\frac{(1-w_{t+1})\eta_{t+1}}{w_{t+1}},\forall t\in\left[T-1\right]$,
$\ASMD$ (Algorithm \ref{alg:ASMD}) guarantees that
\[
\E\left[\left(\Delta_{T+1}\right)^{\frac{\p}{2}}\right]\lesssim\left(\frac{w_{T}\Breg_{\psi}(\X_{\star},\bx_{1})}{\eta_{T}}\right)^{\frac{\p}{2}}+\left(\frac{M^{2}w_{T}}{\eta_{T}}\sum_{t=1}^{T}\eta_{t}^{2}\right)^{\frac{\p}{2}}+\sum_{t=1}^{T}\left(\frac{\sigma^{2}w_{T}\eta_{t}^{2}}{\eta_{T}}\right)^{\frac{\p}{2}}.
\]
In particular, by plugging in different choices of $\left\{ \eta_{t}\right\} _{t=1}^{T}$,
we recover Theorem \ref{thm:main-ASMD-cvx-exp}.
\end{thm}
\begin{proof}
In the following, we write $\Breg_{t}\defeq\Breg_{\psi}(\X_{\star},\by_{t})=\Breg_{\psi}(\left[\by_{t}\right]_{\star},\by_{t}),\forall t\in\left[T+1\right]$.
By $1$-strong convexity of $\psi$ on $\dom F$, we have
\begin{equation}
\frac{\left\Vert \left[\by_{t}\right]_{\star}-\by_{t}\right\Vert ^{2}}{2}\leq\Breg_{t},\forall t\in\left[T+1\right].\label{eq:ASMD-cvx-exp-mirror}
\end{equation}
We invoke Lemma \ref{lem:ASMD-core} to have, for any $t\in\left[T\right]$,
\begin{align*}
\Breg_{t+1}+\frac{\eta_{t}\Delta_{t+1}}{w_{t}} & \leq\Breg_{1}+\sum_{s=1}^{t}\eta_{s}^{2}\left(M+\left\Vert \bxi_{s}\right\Vert _{\star}\right)^{2}+\sum_{s=1}^{t}\eta_{s}\left\langle \bxi_{s},\left[\by_{s}\right]_{\star}-\by_{s}\right\rangle \\
 & \leq\Breg_{1}+2\sum_{s=1}^{t}\eta_{s}^{2}\left(M^{2}+\left\Vert \bxi_{s}\right\Vert _{\star}^{2}\right)+\left|\sum_{s=1}^{t}\eta_{s}\left\langle \bxi_{s},\left[\by_{s}\right]_{\star}-\by_{s}\right\rangle \right|,
\end{align*}
which implies that
\begin{align}
 & \left(\Breg_{t+1}+\frac{\eta_{t}\Delta_{t+1}}{w_{t}}\right)^{\frac{\p}{2}}\leq\left(\Breg_{1}+2\sum_{s=1}^{t}\eta_{s}^{2}\left(M^{2}+\left\Vert \bxi_{s}\right\Vert _{\star}^{2}\right)+\left|\sum_{s=1}^{t}\eta_{s}\left\langle \bxi_{s},\left[\by_{s}\right]_{\star}-\by_{s}\right\rangle \right|\right)^{\frac{\p}{2}}\nonumber \\
\leq & \Breg_{1}^{\frac{\p}{2}}+2^{\frac{\p}{2}}M^{\p}\left(\sum_{s=1}^{t}\eta_{s}^{2}\right)^{\frac{\p}{2}}+2^{\frac{\p}{2}}\sum_{s=1}^{t}\eta_{s}^{\p}\left\Vert \bxi_{s}\right\Vert _{\star}^{\p}+\left|\sum_{s=1}^{t}\eta_{s}\left\langle \bxi_{s},\left[\by_{s}\right]_{\star}-\by_{s}\right\rangle \right|^{\frac{\p}{2}},\label{eq:ASMD-cvx-exp-1}
\end{align}
where the second step is by repeatedly using $(a+b)^{\frac{\p}{2}}\leq a^{\frac{\p}{2}}+b^{\frac{\p}{2}},\forall a,b\geq0$
when $\p\in\left[0,2\right]$.

Next, we take expectations on both sides of (\ref{eq:ASMD-cvx-exp-1})
to obtain
\begin{align*}
\E\left[\left(\Breg_{t+1}+\frac{\eta_{t}\Delta_{t+1}}{w_{t}}\right)^{\frac{\p}{2}}\right] & \leq\Breg_{1}^{\frac{\p}{2}}+2^{\frac{\p}{2}}M^{\p}\left(\sum_{s=1}^{t}\eta_{s}^{2}\right)^{\frac{\p}{2}}+2^{\frac{\p}{2}}\sigma^{\p}\sum_{s=1}^{t}\eta_{s}^{\p}+\E\left[\left|\sum_{s=1}^{t}\eta_{s}\left\langle \bxi_{s},\left[\by_{s}\right]_{\star}-\by_{s}\right\rangle \right|^{\frac{\p}{2}}\right]\\
 & \overset{(a)}{\leq}\Breg_{1}^{\frac{\p}{2}}+2^{\frac{\p}{2}}M^{\p}\left(\sum_{s=1}^{t}\eta_{s}^{2}\right)^{\frac{\p}{2}}+2^{\frac{\p}{2}}\sigma^{\p}\sum_{s=1}^{t}\eta_{s}^{\p}+2^{1-\frac{\p}{4}}\sigma^{\frac{\p}{2}}\sqrt{\sum_{s=1}^{t}\E\left[\eta_{s}^{\p}\Breg_{s}^{\frac{\p}{2}}\right]}\\
 & \overset{(b)}{\leq}\Breg_{1}^{\frac{\p}{2}}+2^{\frac{\p}{2}}M^{\p}\left(\sum_{s=1}^{t}\eta_{s}^{2}\right)^{\frac{\p}{2}}+\left(2^{\frac{\p}{2}}+2^{1-\frac{\p}{2}}\right)\sigma^{\p}\sum_{s=1}^{t}\eta_{s}^{\p}+\frac{\max_{s\in\left[t\right]}\E\left[\Breg_{s}^{\frac{\p}{2}}\right]}{2},
\end{align*}
where $(a)$ is by applying Lemma \ref{lem:core} with $\eta_{s}\left(\left[\by_{s}\right]_{\star}-\by_{s}\right)$
to have
\[
\E\left[\left|\sum_{s=1}^{t}\eta_{s}\left\langle \bxi_{s},\left[\by_{s}\right]_{\star}-\by_{s}\right\rangle \right|^{\frac{\p}{2}}\right]\leq2^{1-\frac{\p}{2}}\sigma^{\frac{\p}{2}}\sqrt{\sum_{s=1}^{t}\E\left[\eta_{s}^{\p}\left\Vert \left[\by_{s}\right]_{\star}-\by_{s}\right\Vert ^{\p}\right]}\overset{(\ref{eq:ASMD-cvx-exp-mirror})}{\leq}2^{1-\frac{\p}{4}}\sigma^{\frac{\p}{2}}\sqrt{\sum_{s=1}^{t}\E\left[\eta_{s}^{\p}\Breg_{s}^{\frac{\p}{2}}\right]},
\]
and $(b)$ follows from
\[
\sqrt{\sum_{s=1}^{t}\E\left[\eta_{s}^{\p}\Breg_{s}^{\frac{\p}{2}}\right]}\leq\sqrt{\sum_{s=1}^{t}\eta_{s}^{\p}\max_{s\in\left[t\right]}\E\left[\Breg_{s}^{\frac{\p}{2}}\right]}\leq\frac{\sigma^{\frac{\p}{2}}\sum_{s=1}^{t}\eta_{s}^{\p}}{2^{\frac{\p}{4}}}+\frac{\max_{s\in\left[t\right]}\E\left[\Breg_{s}^{\frac{\p}{2}}\right]}{2^{2-\frac{\p}{4}}\sigma^{\frac{\p}{2}}},
\]
in which the last step is due to AM-GM inequality.

Following essentially the same proof of (\ref{eq:SMD-cvx-core-3})
obtains
\[
\E\left[\left(\frac{\eta_{T}\Delta_{T+1}}{w_{T}}\right)^{\frac{\p}{2}}\right]\lesssim\Breg_{1}^{\frac{\p}{2}}+M^{\p}\left(\sum_{t=1}^{T}\eta_{t}^{2}\right)^{\frac{\p}{2}}+\sigma^{\p}\sum_{t=1}^{T}\eta_{t}^{\p}.
\]
Finally, we divide both sides by $(\eta_{T}/w_{T})^{\frac{\p}{2}}$
to conclude.
\end{proof}

\subsection{Helpful Lemmas}

We prove some helpful lemmas used in the proofs.
\begin{lem}
\label{lem:SMD-descent}Under Assumption \ref{assu:SMD}, for any
$\by\in\dom F$, $t\in\N$, and $\eta_{t}\leq\frac{1}{2L}$, $\SMD$
(Algorithm \ref{alg:SMD}) guarantees that
\[
F(\bx_{t+1})-F(\by)\leq\frac{\Breg_{\psi}(\by,\bx_{t})-\Breg_{\psi}(\by,\bx_{t+1})}{\eta_{t}}+\eta_{t}\left(M+\left\Vert \bxi_{t}\right\Vert _{\star}\right)^{2}+\left\langle \bxi_{t},\by-\bx_{t}\right\rangle .
\]
\end{lem}
\begin{proof}
In the following proof, we fix $\by\in\dom F$. By $(L,M)$-relative
smoothness, we know
\begin{equation}
f(\bx_{t+1})-f(\bx_{t})\leq\left\langle \nabla f(\bx_{t}),\bx_{t+1}-\bx_{t}\right\rangle +L\Breg_{\psi}(\bx_{t+1},\bx_{t})+M\sqrt{2\Breg_{\psi}(\bx_{t+1},\bx_{t})}.\label{eq:SMD-descent-1}
\end{equation}
Note that we have the decomposition
\begin{align}
\left\langle \nabla f(\bx_{t}),\bx_{t+1}-\bx_{t}\right\rangle  & =\left\langle \bxi_{t},\bx_{t}-\bx_{t+1}\right\rangle +\left\langle \bg_{t},\bx_{t+1}-\by\right\rangle +\left\langle \nabla f(\bx_{t}),\by-\bx_{t}\right\rangle +\left\langle \bxi_{t},\by-\bx_{t}\right\rangle \nonumber \\
 & \leq\underbrace{\left\langle \bxi_{t},\bx_{t}-\bx_{t+1}\right\rangle }_{\mathrm{I}}+\underbrace{\left\langle \bg_{t},\bx_{t+1}-\by\right\rangle }_{\mathrm{II}}+f(\by)-f(\bx_{t})+\left\langle \bxi_{t},\by-\bx_{t}\right\rangle ,\label{eq:SMD-descent-2}
\end{align}
where the last step is by the convexity of $f$. Next, we bound the
left two terms in the following.

\begin{itemize}[leftmargin=*]

\item Term $\mathrm{I}$. By Cauchy-Schwarz inequality and $1$-strong
convexity of $\psi$, we have
\begin{equation}
\left\langle \bxi_{t},\bx_{t}-\bx_{t+1}\right\rangle \leq\left\Vert \bxi_{t}\right\Vert _{\star}\left\Vert \bx_{t}-\bx_{t+1}\right\Vert \leq\left\Vert \bxi_{t}\right\Vert _{\star}\sqrt{2\Breg_{\psi}(\bx_{t+1},\bx_{t})}\label{eq:SMD-descent-I}
\end{equation}

\item Term $\mathrm{II}$. By the update rule of Algorithm \ref{alg:SMD},
there exists $\nabla r(\bx_{t+1})\in\partial r(\bx_{t+1})$ such that
\[
\left\langle \nabla r(\bx_{t+1})+\bg_{t}+\frac{\nabla\psi(\bx_{t+1})-\nabla\psi(\bx_{t})}{\eta_{t}},\bx_{t+1}-\by\right\rangle \leq0,
\]
which implies that
\begin{align}
\left\langle \bg_{t},\bx_{t+1}-\by\right\rangle  & \leq\frac{\left\langle \nabla\psi(\bx_{t})-\nabla\psi(\bx_{t+1}),\bx_{t+1}-\by\right\rangle }{\eta_{t}}+\left\langle \nabla r(\bx_{t+1}),\by-\bx_{t+1}\right\rangle \nonumber \\
 & \leq\frac{\left\langle \nabla\psi(\bx_{t})-\nabla\psi(\bx_{t+1}),\bx_{t+1}-\by\right\rangle }{\eta_{t}}+r(\by)-r(\bx_{t+1})\nonumber \\
 & =\frac{\Breg_{\psi}(\by,\bx_{t})-\Breg_{\psi}(\by,\bx_{t+1})-\Breg_{\psi}(\bx_{t+1},\bx_{t})}{\eta_{t}}+r(\by)-r(\bx_{t+1}),\label{eq:SMD-descent-II}
\end{align}
where the second inequality is due to the convexity of $r$.

\end{itemize}

Plug (\ref{eq:SMD-descent-I}) and (\ref{eq:SMD-descent-II}) back
into (\ref{eq:SMD-descent-2}) and rearrange terms to obtain
\begin{align*}
\left\langle \nabla f(\bx_{t}),\bx_{t+1}-\bx_{t}\right\rangle \leq & \frac{\Breg_{\psi}(\by,\bx_{t})-\Breg_{\psi}(\by,\bx_{t+1})}{\eta_{t}}+\left\Vert \bxi_{t}\right\Vert _{\star}\sqrt{2\Breg_{\psi}(\bx_{t+1},\bx_{t})}-\frac{\Breg_{\psi}(\bx_{t+1},\bx_{t})}{\eta_{t}}\\
 & +F(\by)-f(\bx_{t})-r(\bx_{t+1})+\left\langle \bxi_{t},\by-\bx_{t}\right\rangle .
\end{align*}
We combine the inequality above and (\ref{eq:SMD-descent-1}), rearrange
terms, and use $\eta_{t}\leq\frac{1}{2L}$ to conclude that
\begin{align*}
F(\bx_{t+1})-F(\by)\leq & \frac{\Breg_{\psi}(\by,\bx_{t})-\Breg_{\psi}(\by,\bx_{t+1})}{\eta_{t}}+\left\langle \bxi_{t},\by-\bx_{t}\right\rangle \\
 & +\left(M+\left\Vert \bxi_{t}\right\Vert _{\star}\right)\sqrt{2\Breg_{\psi}(\bx_{t+1},\bx_{t})}-\frac{\Breg_{\psi}(\bx_{t+1},\bx_{t})}{2\eta_{t}}\\
\leq & \frac{\Breg_{\psi}(\by,\bx_{t})-\Breg_{\psi}(\by,\bx_{t+1})}{\eta_{t}}+\left\langle \bxi_{t},\by-\bx_{t}\right\rangle +\eta_{t}\left(M+\left\Vert \bxi_{t}\right\Vert _{\star}\right)^{2},
\end{align*}
where the last step is due to AM-GM inequality.
\end{proof}

\begin{lem}
\label{lem:SMD-core}Under Assumption \ref{assu:SMD}, for any $T\in\N$
and $\left\{ \eta_{t}\leq\frac{1}{2L}\right\} _{t=1}^{T}$, $\SMD$
(Algorithm \ref{alg:SMD}) guarantees that, for any $t\in\left[T\right]$,
\begin{align*}
 & \frac{\gamma_{t+1}}{\eta_{t+1}}\Breg_{\psi}(\X_{\star},\bx_{t+1})+\sum_{s=1}^{t}\frac{\gamma_{s}\Delta_{s+1}}{2}\leq\Breg_{\psi}(\X_{\star},\bx_{1})+\sum_{s=1}^{t}\eta_{s}\gamma_{s}\left(M+\left\Vert \bxi_{s}\right\Vert _{\star}\right)^{2}+\sum_{s=1}^{t}\gamma_{s}\left\langle \bxi_{s},\left[\bx_{s}\right]_{\star}-\bx_{s}\right\rangle ,\\
 & v_{t}\Breg_{\psi}(\bz_{t},\bx_{t+1})+v_{t}\eta_{t}\Delta_{t+1}\leq v_{0}\Breg_{\psi}(\X_{\star},\bx_{1})+\sum_{s=1}^{t}v_{s}\eta_{s}^{2}\left(M+\left\Vert \bxi_{s}\right\Vert _{\star}\right)^{2}+\sum_{s=1}^{t}v_{s-1}\eta_{s}\left\langle \bxi_{s},\bz_{s-1}-\bx_{s}\right\rangle ,
\end{align*}
where $\left\{ \gamma_{t}\right\} _{t=1}^{T+1}$ is a sequence defined
in (\ref{eq:SMD-core-gamma-def}), $\left\{ v_{t}\right\} _{t=0}^{T}$
is a sequence defined in (\ref{eq:SMD-core-v-def}), and $\left\{ \bz_{t}\right\} _{t=0}^{T}$
is a sequence defined in (\ref{eq:SMD-core-z-def}).
\end{lem}
\begin{proof}
In the following, we fix $T\in\N$.

\begin{itemize}[leftmargin=*]

\item For the first inequality, we invoke Lemma \ref{lem:SMD-descent}
with $\by=\left[\bx_{t}\right]_{\star}$ and use
\begin{eqnarray*}
F(\left[\bx_{t}\right]_{\star})=F_{\star}, & \Breg_{\psi}(\left[\bx_{t}\right]_{\star},\bx_{t})=\Breg_{\psi}(\X_{\star},\bx_{t}), & \Breg_{\psi}(\left[\bx_{t}\right]_{\star},\bx_{t+1})\geq\Breg_{\psi}(\X_{\star},\bx_{t+1})
\end{eqnarray*}
to have
\[
\Delta_{t+1}\leq\frac{\Breg_{\psi}(\X_{\star},\bx_{t})-\Breg_{\psi}(\X_{\star},\bx_{t+1})}{\eta_{t}}+\eta_{t}\left(M+\left\Vert \bxi_{t}\right\Vert _{\star}\right)^{2}+\left\langle \bxi_{t},\left[\bx_{t}\right]_{\star}-\bx_{t}\right\rangle ,
\]
which, by the $\mu$-Bregman growth condition, further implies that
\[
\frac{\mu\Breg_{\psi}(\X_{\star},\bx_{t+1})+\Delta_{t+1}}{2}\leq\frac{\Breg_{\psi}(\X_{\star},\bx_{t})-\Breg_{\psi}(\X_{\star},\bx_{t+1})}{\eta_{t}}+\eta_{t}\left(M+\left\Vert \bxi_{t}\right\Vert _{\star}\right)^{2}+\left\langle \bxi_{t},\left[\bx_{t}\right]_{\star}-\bx_{t}\right\rangle .
\]
Multiply both sides of the above inequality by $\gamma_{t}$ and rearrange
terms to obtain
\[
\frac{\gamma_{t+1}}{\eta_{t+1}}\Breg_{\psi}(\X_{\star},\bx_{t+1})+\frac{\gamma_{t}\Delta_{t+1}}{2}\leq\frac{\gamma_{t}}{\eta_{t}}\Breg_{\psi}(\X_{\star},\bx_{t})+\eta_{t}\gamma_{t}\left(M+\left\Vert \bxi_{t}\right\Vert _{\star}\right)^{2}+\gamma_{t}\left\langle \bxi_{t},\left[\bx_{t}\right]_{\star}-\bx_{t}\right\rangle ,
\]
where $\left\{ \gamma_{t}\right\} _{t=1}^{T+1}$ is a sequence defined
as
\begin{equation}
\gamma_{t}\defeq\eta_{t}\prod_{s=1}^{t-1}\left(1+\frac{\mu\eta_{s}}{2}\right),\forall t\in\left[T+1\right].\label{eq:SMD-core-gamma-def}
\end{equation}
Finally, a direct calculation yields the desired result.

\item As for the second inequality, it can in fact be derived from
existing results in the literature, e.g., Lemma G.1 of \citep{liu2024revisiting}.
However, to make the paper self-contained, we reproduce the proof
below. We first define a nondecreasing sequence $\left\{ v_{t}\right\} _{t=0}^{T}$
as
\begin{eqnarray}
v_{t}\defeq\frac{\eta_{T}}{\sum_{s=t}^{T}\eta_{s}},\forall t\in\left[T\right] & \text{and} & v_{0}\defeq v_{1}.\label{eq:SMD-core-v-def}
\end{eqnarray}
Equipped with $\left\{ v_{t}\right\} _{t=0}^{T}$, we introduce another
sequence $\left\{ \bz_{t}\right\} _{t=0}^{T}$ recursively defined
as
\begin{eqnarray}
\bz_{t}\defeq\frac{v_{t-1}}{v_{t}}\bz_{t-1}+\left(1-\frac{v_{t-1}}{v_{t}}\right)\bx_{t},\forall t\in\left[T\right] & \text{where} & \bz_{0}\defeq\left[\bx_{1}\right]_{\star}.\label{eq:SMD-core-z-def}
\end{eqnarray}
Given $t\in\left\{ 0\right\} \cup\left[T\right]$, by expanding the
definition of $\bz_{t}$, we have
\begin{equation}
\bz_{t}=\frac{v_{0}}{v_{t}}\left[\bx_{1}\right]_{\star}+\sum_{s=1}^{t}\frac{v_{s}-v_{s-1}}{v_{t}}\bx_{s},\label{eq:SMD-core-z-prop}
\end{equation}
implying that $\bz_{t}\in\dom F$, as $\left\{ v_{t}\right\} _{t=0}^{T}$
is a nondecreasing sequence and $\left\{ \left[\bx_{1}\right]_{\star}\right\} \cup\left\{ \bx_{t}\right\} _{t=1}^{T}\subseteq\dom F$.
In particular, we have
\begin{eqnarray}
F(\bz_{0})=F(\left[\bx_{1}\right]_{\star})=F_{\star} & \text{and} & \Breg_{\psi}(\bz_{0},\bx_{1})=\Breg_{\psi}(\left[\bx_{1}\right]_{\star},\bx_{1})=\Breg_{\psi}(\X_{\star},\bx_{1}).\label{eq:SMD-core-z-value}
\end{eqnarray}

Now, we invoke Lemma \ref{lem:SMD-descent} with $\by=\bz_{t}$ at
$t\in\left[T\right]$ and multiply both sides by $v_{t}\eta_{t}$
to obtain
\begin{align}
 & v_{t}\eta_{t}\left(F(\bx_{t+1})-F(\bz_{t})\right)\nonumber \\
\leq & v_{t}\Breg_{\psi}(\bz_{t},\bx_{t})-v_{t}\Breg_{\psi}(\bz_{t},\bx_{t+1})+v_{t}\eta_{t}^{2}\left(M+\left\Vert \bxi_{t}\right\Vert _{\star}\right)^{2}+v_{t}\eta_{t}\left\langle \bxi_{t},\bz_{t}-\bx_{t}\right\rangle \nonumber \\
\leq & v_{t-1}\Breg_{\psi}(\bz_{t-1},\bx_{t})-v_{t}\Breg_{\psi}(\bz_{t},\bx_{t+1})+v_{t}\eta_{t}^{2}\left(M+\left\Vert \bxi_{t}\right\Vert _{\star}\right)^{2}+v_{t-1}\eta_{t}\left\langle \bxi_{t},\bz_{t-1}-\bx_{t}\right\rangle ,\label{eq:SMD-core-1}
\end{align}
where the second step is by $\Breg_{\psi}(\bz_{t},\bx_{t})\overset{(\ref{eq:SMD-core-z-def})}{\leq}\frac{v_{t-1}}{v_{t}}\Breg_{\psi}(\bz_{t-1},\bx_{t})+\left(1-\frac{v_{t-1}}{v_{t}}\right)\Breg_{\psi}(\bx_{t},\bx_{t})=\frac{v_{t-1}}{v_{t}}\Breg_{\psi}(\bz_{t-1},\bx_{t})$
and $v_{t}(\bz_{t}-\bx_{t})\overset{(\ref{eq:SMD-core-z-def})}{=}v_{t-1}(\bz_{t-1}-\bx_{t})$.
Relabel $t$ by $s$ in (\ref{eq:SMD-core-1}), sum the obtained inequality
from $s=1$ to $t\leq T$, rearrange terms, and use (\ref{eq:SMD-core-z-value})
to have
\begin{align*}
 & v_{t}\Breg_{\psi}(\bz_{t},\bx_{t+1})+\sum_{s=1}^{t}v_{s}\eta_{s}\left(F(\bx_{s+1})-F(\bz_{s})\right)\\
\leq & v_{0}\Breg_{\psi}(\X_{\star},\bx_{1})+\sum_{s=1}^{t}v_{s}\eta_{s}^{2}\left(M+\left\Vert \bxi_{s}\right\Vert _{\star}\right)^{2}+\sum_{s=1}^{t}v_{s-1}\eta_{s}\left\langle \bxi_{s},\bz_{s-1}-\bx_{s}\right\rangle .
\end{align*}

Therefore, we only need to show $\sum_{s=1}^{t}v_{s}\eta_{s}\left(F(\bx_{s+1})-F(\bz_{s})\right)\geq v_{t}\eta_{t}\Delta_{t+1}$.
This can be done by observing that
\begin{align*}
 & \sum_{s=1}^{t}v_{s}\eta_{s}\left(F(\bz_{s})-F_{\star}\right)\overset{(\ref{eq:SMD-core-z-prop})}{\leq}\sum_{s=1}^{t}v_{s}\eta_{s}\sum_{\ell=1}^{s}\frac{v_{\ell}-v_{\ell-1}}{v_{s}}\Delta_{\ell}=\sum_{\ell=1}^{t}\left(\sum_{s=\ell}^{t}\eta_{s}\right)\left(v_{\ell}-v_{\ell-1}\right)\Delta_{\ell}\\
\overset{(\ref{eq:SMD-core-v-def})}{=} & \sum_{\ell=2}^{t}\left(\sum_{s=\ell}^{t}\eta_{s}\right)\frac{\eta_{T}\eta_{\ell-1}}{\left(\sum_{k=\ell}^{T}\eta_{k}\right)\left(\sum_{k=\ell-1}^{T}\eta_{k}\right)}\Delta_{\ell}\overset{t\leq T}{\leq}\sum_{\ell=2}^{t}\frac{\eta_{T}\eta_{\ell-1}}{\sum_{k=\ell-1}^{T}\eta_{k}}\Delta_{\ell}=\sum_{\ell=2}^{t}v_{\ell-1}\eta_{\ell-1}\Delta_{\ell},
\end{align*}
which implies that
\begin{align*}
\sum_{s=1}^{t}v_{s}\eta_{s}\left(F(\bx_{s+1})-F(\bz_{s})\right) & =\sum_{s=1}^{t}v_{s}\eta_{s}\Delta_{s+1}-\sum_{s=1}^{t}v_{s}\eta_{s}\left(F(\bz_{s})-F_{\star}\right)\\
 & \geq\sum_{s=1}^{t}v_{s}\eta_{s}\Delta_{s+1}-\sum_{\ell=2}^{t}v_{\ell-1}\eta_{\ell-1}\Delta_{\ell}=v_{t}\eta_{t}\Delta_{t+1}.
\end{align*}
\end{itemize}
\end{proof}

\begin{lem}
\label{lem:ASMD-descent}Under Assumption \ref{assu:ASMD}, for any
$\by\in\dom F$, $t\in\N$, $w_{t}\in\left[0,1\right]$, and $\eta_{t}\leq\frac{1}{2w_{t}L}$,
$\ASMD$ (Algorithm \ref{alg:ASMD}) guarantees that
\[
F(\bz_{t+1})-F(\by)\leq(1-w_{t})(F(\bz_{t})-F(\by))+\frac{w_{t}\Breg_{\psi}(\by,\by_{t})-w_{t}\Breg_{\psi}(\by,\by_{t+1})}{\eta_{t}}+\eta_{t}w_{t}\left(M+\left\Vert \bxi_{t}\right\Vert _{\star}\right)^{2}+w_{t}\left\langle \bxi_{t},\by-\by_{t}\right\rangle .
\]
\end{lem}
\begin{proof}
In the following proof, we fix $\by\in\dom F$. By $(L,M)$-smoothness,
we know
\begin{align}
f(\bz_{t+1})-f(\bx_{t}) & \leq\left\langle \nabla f(\bx_{t}),\bz_{t+1}-\bx_{t}\right\rangle +\frac{L}{2}\left\Vert \bz_{t+1}-\bx_{t}\right\Vert ^{2}+M\left\Vert \bz_{t+1}-\bx_{t}\right\Vert \nonumber \\
 & \overset{(a)}{=}w_{t}\left\langle \nabla f(\bx_{t}),\by_{t+1}-\by_{t}\right\rangle +\frac{w_{t}^{2}L}{2}\left\Vert \by_{t+1}-\by_{t}\right\Vert ^{2}+w_{t}M\left\Vert \by_{t+1}-\by_{t}\right\Vert \nonumber \\
 & \overset{(b)}{\leq}w_{t}\left\langle \nabla f(\bx_{t}),\by_{t+1}-\by_{t}\right\rangle +w_{t}^{2}L\Breg_{\psi}(\by_{t+1},\by_{t})+w_{t}M\sqrt{2\Breg_{\psi}(\by_{t+1},\by_{t})},\label{eq:ASMD-descent-1}
\end{align}
where $(a)$ holds due to $\bz_{t+1}-\bx_{t}=w_{t}(\by_{t+1}-\by_{t})$
in $\ASMD$ and $(b)$ follows from $1$-strong convexity of $\psi$.
Note that we have the decomposition
\begin{align}
w_{t}\left\langle \nabla f(\bx_{t}),\by_{t+1}-\by_{t}\right\rangle = & w_{t}\left\langle \bxi_{t},\by_{t}-\by_{t+1}\right\rangle +w_{t}\left\langle \bg_{t},\by_{t+1}-\by\right\rangle \nonumber \\
 & +w_{t}\left\langle \nabla f(\bx_{t}),\by-\bx_{t}\right\rangle +(1-w_{t})\left\langle \nabla f(\bx_{t}),\bz_{t}-\bx_{t}\right\rangle +w_{t}\left\langle \bxi_{t},\by-\by_{t}\right\rangle \nonumber \\
\leq & \underbrace{w_{t}\left\langle \bxi_{t},\by_{t}-\by_{t+1}\right\rangle }_{\mathrm{I}}+\underbrace{w_{t}\left\langle \bg_{t},\by_{t+1}-\by\right\rangle }_{\mathrm{II}}\nonumber \\
 & +w_{t}(f(\by)-f(\bx_{t}))+(1-w_{t})(f(\bz_{t})-f(\bx_{t}))+w_{t}\left\langle \bxi_{t},\by-\by_{t}\right\rangle ,\label{eq:ASMD-descent-2}
\end{align}
where the last step is by the convexity of $f$. Next, we bound the
left two terms in the following.

\begin{itemize}[leftmargin=*]

\item Term $\mathrm{I}$. By Cauchy-Schwarz inequality and $1$-strong
convexity of $\psi$, we have
\begin{equation}
w_{t}\left\langle \bxi_{t},\by_{t}-\by_{t+1}\right\rangle \leq w_{t}\left\Vert \bxi_{t}\right\Vert _{\star}\left\Vert \by_{t}-\by_{t+1}\right\Vert \leq w_{t}\left\Vert \bxi_{t}\right\Vert _{\star}\sqrt{2\Breg_{\psi}(\by_{t+1},\by_{t})}.\label{eq:ASMD-descent-I}
\end{equation}

\item Term $\mathrm{II}$. By the update rule of Algorithm \ref{alg:ASMD},
there exists $\nabla r(\by_{t+1})\in\partial r(\by_{t+1})$ such that
\[
\left\langle \nabla r(\by_{t+1})+\bg_{t}+\frac{\nabla\psi(\by_{t+1})-\nabla\psi(\by_{t})}{\eta_{t}},\by_{t+1}-\by\right\rangle \leq0,
\]
which implies that
\begin{align}
w_{t}\left\langle \bg_{t},\by_{t+1}-\by\right\rangle  & \leq\frac{w_{t}\left\langle \nabla\psi(\by_{t})-\nabla\psi(\by_{t+1}),\by_{t+1}-\by\right\rangle }{\eta_{t}}+w_{t}\left\langle \nabla r(\by_{t+1}),\by-\by_{t+1}\right\rangle \nonumber \\
 & \leq\frac{w_{t}\left\langle \nabla\psi(\by_{t})-\nabla\psi(\by_{t+1}),\by_{t+1}-\by\right\rangle }{\eta_{t}}+w_{t}(r(\by)-r(\by_{t+1}))\nonumber \\
 & =\frac{w_{t}\Breg_{\psi}(\by,\by_{t})-w_{t}\Breg_{\psi}(\by,\by_{t+1})-w_{t}\Breg_{\psi}(\by_{t+1},\by_{t})}{\eta_{t}}+w_{t}(r(\by)-r(\by_{t+1})),\label{eq:ASMD-descent-II}
\end{align}
where the second inequality is due to the convexity of $r$.

\end{itemize}

Plug (\ref{eq:ASMD-descent-I}) and (\ref{eq:ASMD-descent-II}) back
into (\ref{eq:ASMD-descent-2}), rearrange terms, and use $r(\bz_{t+1})\leq w_{t}r(\by_{t+1})+(1-w_{t})r(\bz_{t})$
by convexity of $r$ and $\bz_{t+1}=w_{t}\by_{t+1}+(1-w_{t})\bz_{t}$
to obtain
\begin{align*}
w_{t}\left\langle \nabla f(\bx_{t}),\by_{t+1}-\by_{t}\right\rangle \leq & \frac{w_{t}\Breg_{\psi}(\by,\by_{t})-w_{t}\Breg_{\psi}(\by,\by_{t+1})}{\eta_{t}}+w_{t}\left\Vert \bxi_{t}\right\Vert _{\star}\sqrt{2\Breg_{\psi}(\by_{t+1},\by_{t})}-\frac{w_{t}\Breg_{\psi}(\by_{t+1},\by_{t})}{\eta_{t}}\\
 & +w_{t}F(\by)+(1-w_{t})F(\bz_{t})-f(\bx_{t})-r(\bz_{t+1})+w_{t}\left\langle \bxi_{t},\by-\by_{t}\right\rangle .
\end{align*}
We combine the inequality above and (\ref{eq:ASMD-descent-1}), rearrange
terms, and use $\eta_{t}\leq\frac{1}{2w_{t}L}$ to conclude that
\begin{align*}
F(\bz_{t+1})-F(\by)\leq & (1-w_{t})(F(\bz_{t})-F(\by))+\frac{w_{t}\Breg_{\psi}(\by,\by_{t})-w_{t}\Breg_{\psi}(\by,\by_{t+1})}{\eta_{t}}\\
 & +w_{t}\left\langle \bxi_{t},\by-\by_{t}\right\rangle +\left(M+\left\Vert \bxi_{t}\right\Vert _{\star}\right)w_{t}\sqrt{2\Breg_{\psi}(\by_{t+1},\by_{t})}-\frac{w_{t}\Breg_{\psi}(\by_{t+1},\by_{t})}{2\eta_{t}}\\
\leq & (1-w_{t})(F(\bz_{t})-F(\by))+\frac{w_{t}\Breg_{\psi}(\by,\by_{t})-w_{t}\Breg_{\psi}(\by,\by_{t+1})}{\eta_{t}}\\
 & +w_{t}\left\langle \bxi_{t},\by-\by_{t}\right\rangle +\eta_{t}w_{t}\left(M+\left\Vert \bxi_{t}\right\Vert _{\star}\right)^{2},
\end{align*}
where the last step is due to AM-GM inequality.
\end{proof}

\begin{lem}
\label{lem:ASMD-core}Under Assumption \ref{assu:ASMD}, for any $T\in\N$,
$\left\{ w_{t}\in\left[0,1\right]\right\} _{t=1}^{T}$ satisfying
$w_{1}=1$, and $\left\{ \eta_{t}\leq\frac{1}{2w_{t}L}\right\} _{t=1}^{T}$
satisfying $\frac{\eta_{t}}{w_{t}}\geq\frac{(1-w_{t+1})\eta_{t+1}}{w_{t+1}},\forall t\in\left[T-1\right]$,
$\ASMD$ (Algorithm \ref{alg:ASMD}) guarantees that, for any $t\in\left[T\right]$,
\[
\Breg_{\psi}(\X_{\star},\by_{t+1})+\frac{\eta_{t}\Delta_{t+1}}{w_{t}}\leq\Breg_{\psi}(\X_{\star},\by_{1})+\sum_{s=1}^{t}\eta_{s}^{2}\left(M+\left\Vert \bxi_{s}\right\Vert _{\star}\right)^{2}+\sum_{s=1}^{t}\eta_{s}\left\langle \bxi_{s},\left[\by_{s}\right]_{\star}-\by_{s}\right\rangle .
\]
\end{lem}
\begin{proof}
In the following, we fix $T\in\N$. We invoke Lemma \ref{lem:ASMD-descent}
with $\by=\left[\by_{t}\right]_{\star}$ and use
\begin{eqnarray*}
F(\left[\by_{t}\right]_{\star})=F_{\star}, & \Breg_{\psi}(\left[\by_{t}\right]_{\star},\by_{t})=\Breg_{\psi}(\X_{\star},\by_{t}), & \Breg_{\psi}(\left[\by_{t}\right]_{\star},\by_{t+1})\geq\Breg_{\psi}(\X_{\star},\by_{t+1})
\end{eqnarray*}
to have
\[
\Delta_{t+1}\leq(1-w_{t})\Delta_{t}+\frac{w_{t}\Breg_{\psi}(\X_{\star},\by_{t})-w_{t}\Breg_{\psi}(\X_{\star},\by_{t+1})}{\eta_{t}}+\eta_{t}w_{t}\left(M+\left\Vert \bxi_{t}\right\Vert _{\star}\right)^{2}+w_{t}\left\langle \bxi_{t},\left[\by_{t}\right]_{\star}-\by_{t}\right\rangle .
\]
Multiply both sides of the above inequality by $\frac{\eta_{t}}{w_{t}}$
and rearrange terms to obtain
\[
\Breg_{\psi}(\X_{\star},\by_{t+1})+\frac{\eta_{t}\Delta_{t+1}}{w_{t}}\leq\Breg_{\psi}(\X_{\star},\by_{t})+\frac{(1-w_{t})\eta_{t}\Delta_{t}}{w_{t}}+\eta_{t}^{2}\left(M+\left\Vert \bxi_{t}\right\Vert _{\star}\right)^{2}+\eta_{t}\left\langle \bxi_{t},\left[\by_{t}\right]_{\star}-\by_{t}\right\rangle .
\]
Finally, using $w_{1}=1$ and $\frac{\eta_{t}}{w_{t}}\geq\frac{(1-w_{t+1})\eta_{t+1}}{w_{t+1}},\forall t\in\left[T-1\right]$,
a direct calculation yields the desired result.
\end{proof}

\section{Nonconvex Optimization\label{sec:proof-ncvx}}

\subsection{Full Theorems for $\protect\SGD$ and Proofs}
\begin{thm}[Full statement of Theorem \ref{thm:main-SGD-exp}]
\label{thm:SGD-exp}Under Assumptions \ref{assu:oracle} (with $h=F$)
and \ref{assu:SGD}, for any $T\in\N$ and $\left\{ \eta_{t}\leq\frac{1}{L}\right\} _{t=1}^{T}$,
$\SGD$ (Algorithm \ref{alg:SGD}) guarantees that
\[
\E\left[\left(\sum_{t=1}^{T}\eta_{t}\left\Vert \nabla F(\bx_{t})\right\Vert ^{2}\right)^{\frac{\p}{2}}\right]\lesssim\left(F(\bx_{1})-F_{\star}\right)^{\frac{\p}{2}}+\sigma^{\p}\left(\sum_{t=1}^{T}\eta_{t}^{\frac{\p}{2-\p}}\right)^{\frac{2-\p}{2}}+L^{\frac{\p}{2}}\sigma^{\p}\sum_{t=1}^{T}\eta_{t}^{\p}.
\]
 In particular, by plugging in different choices of $\left\{ \eta_{t}\right\} _{t=1}^{T}$,
we recover Theorem \ref{thm:main-SGD-exp}.
\end{thm}
\begin{proof}
In the following, we write $\Delta\defeq F(\bx_{1})-F_{\star}$. By
Lemma \ref{lem:SGD-core}, we have
\begin{align*}
\sum_{t=1}^{T}\eta_{t}\left\Vert \nabla F(\bx_{t})\right\Vert ^{2} & \leq2\Delta+L\sum_{t=1}^{T}\eta_{t}^{2}\left\Vert \bxi_{t}\right\Vert ^{2}+2\sum_{t=1}^{T}\left(\eta_{t}^{2}L-\eta_{t}\right)\left\langle \bxi_{t},\nabla F(\bx_{t})\right\rangle \\
 & \leq2\Delta+L\sum_{t=1}^{T}\eta_{t}^{2}\left\Vert \bxi_{t}\right\Vert ^{2}+2\left|\sum_{t=1}^{T}\left(\eta_{t}^{2}L-\eta_{t}\right)\left\langle \bxi_{t},\nabla F(\bx_{t})\right\rangle \right|,
\end{align*}
which further implies that
\begin{align}
 & \left(\sum_{t=1}^{T}\eta_{t}\left\Vert \nabla F(\bx_{t})\right\Vert ^{2}\right)^{\frac{\p}{2}}\leq\left(2\Delta+L\sum_{t=1}^{T}\eta_{t}^{2}\left\Vert \bxi_{t}\right\Vert ^{2}+2\left|\sum_{t=1}^{T}\left(\eta_{t}^{2}L-\eta_{t}\right)\left\langle \bxi_{t},\nabla F(\bx_{t})\right\rangle \right|\right)^{\frac{\p}{2}}\nonumber \\
\leq & \left(2\Delta\right)^{\frac{\p}{2}}+L^{\frac{\p}{2}}\sum_{t=1}^{T}\eta_{t}^{\p}\left\Vert \bxi_{t}\right\Vert ^{\p}+2^{\frac{\p}{2}}\left|\sum_{t=1}^{T}\left(\eta_{t}^{2}L-\eta_{t}\right)\left\langle \bxi_{t},\nabla F(\bx_{t})\right\rangle \right|^{\frac{\p}{2}},\label{eq:SGD-exp-1}
\end{align}
where the second step is by repeatedly using $(a+b)^{\frac{\p}{2}}\leq a^{\frac{\p}{2}}+b^{\frac{\p}{2}},\forall a,b\geq0$
when $\p\in\left[0,2\right]$.

Next, we take expectations on both sides of (\ref{eq:SGD-exp-1})
to obtain
\begin{align}
 & \E\left[\left(\sum_{t=1}^{T}\eta_{t}\left\Vert \nabla F(\bx_{t})\right\Vert ^{2}\right)^{\frac{\p}{2}}\right]\nonumber \\
\leq & \left(2\Delta\right)^{\frac{\p}{2}}+L^{\frac{\p}{2}}\sigma^{\p}\sum_{t=1}^{T}\eta_{t}^{\p}+2^{\frac{\p}{2}}\E\left[\left|\sum_{t=1}^{T}\left(\eta_{t}^{2}L-\eta_{t}\right)\left\langle \bxi_{t},\nabla F(\bx_{t})\right\rangle \right|^{\frac{\p}{2}}\right]\nonumber \\
\overset{(a)}{\leq} & \left(2\Delta\right)^{\frac{\p}{2}}+L^{\frac{\p}{2}}\sigma^{\p}\sum_{t=1}^{T}\eta_{t}^{\p}+2^{\frac{\p}{2}}\sqrt{\E\left[\sum_{t=1}^{T}\left|\eta_{t}^{2}L-\eta_{t}\right|^{\p}\left\Vert \nabla F(\bx_{t})\right\Vert ^{\p}\right]}\nonumber \\
\overset{(b)}{\leq} & \left(2\Delta\right)^{\frac{\p}{2}}+L^{\frac{\p}{2}}\sigma^{\p}\sum_{t=1}^{T}\eta_{t}^{\p}+2\sigma^{\frac{\p}{2}}\sqrt{\left(\sum_{t=1}^{T}\eta_{t}^{\frac{\p}{2-\p}}\left|\eta_{t}L-1\right|^{\frac{2\p}{2-\p}}\right)^{\frac{2-\p}{2}}\E\left[\left(\sum_{t=1}^{T}\eta_{t}\left\Vert \nabla F(\bx_{t})\right\Vert ^{2}\right)^{\frac{\p}{2}}\right]}\nonumber \\
\overset{(c)}{\leq} & \left(2\Delta\right)^{\frac{\p}{2}}+L^{\frac{\p}{2}}\sigma^{\p}\sum_{t=1}^{T}\eta_{t}^{\p}+2\sigma^{\p}\left(\sum_{t=1}^{T}\eta_{t}^{\frac{\p}{2-\p}}\left|\eta_{t}L-1\right|^{\frac{2\p}{2-\p}}\right)^{\frac{2-\p}{2}}+\frac{\E\left[\left(\sum_{t=1}^{T}\eta_{t}\left\Vert \nabla F(\bx_{t})\right\Vert ^{2}\right)^{\frac{\p}{2}}\right]}{2},\label{eq:SGD-exp-2}
\end{align}
where $(a)$ is by applying Lemma \ref{lem:core} with $\by_{t}=\left(\eta_{t}^{2}L-\eta_{t}\right)\nabla F(\bx_{t})$,
$(b)$ is due to Cauchy-Schwarz inequality as follows
\begin{align}
\sum_{t=1}^{T}\left|\eta_{t}^{2}L-\eta_{t}\right|^{\p}\left\Vert \nabla F(\bx_{t})\right\Vert ^{\p} & =\sum_{t=1}^{T}\eta_{t}^{\frac{\p}{2}}\left|\eta_{t}L-1\right|^{\p}\cdot\eta_{t}^{\frac{\p}{2}}\left\Vert \nabla F(\bx_{t})\right\Vert ^{\p}\nonumber \\
 & \leq\left(\sum_{t=1}^{T}\eta_{t}^{\frac{\p}{2-\p}}\left|\eta_{t}L-1\right|^{\frac{2\p}{2-\p}}\right)^{\frac{2-\p}{2}}\left(\sum_{t=1}^{T}\eta_{t}\left\Vert \nabla F(\bx_{t})\right\Vert ^{2}\right)^{\frac{\p}{2}},\label{eq:SGD-exp-3}
\end{align}
and $(c)$ follows from AM-GM inequality.

Finally, we rearrange the terms in (\ref{eq:SGD-exp-2}) and use $\left|\eta_{t}L-1\right|\leq1$
to conclude.
\end{proof}

\begin{thm}[Full statement of Theorem \ref{thm:main-SGD-exp-smo+lip}]
\label{thm:SGD-exp-smo+lip}Under Assumptions \ref{assu:oracle}
(with $h=F$ and $\p\in\left(1,2\right)$) and \ref{assu:SGD}, and
an additional condition $\left\Vert \nabla F(\bx)\right\Vert \leq G,\forall\bx\in\R^{d}$,
for any $T\in\N$ and $\left\{ \eta_{t}\right\} _{t=1}^{T}$, $\SGD$
(Algorithm \ref{alg:SGD}) guarantees that
\[
\E\left[\sum_{t=1}^{T}\eta_{t}\left\Vert \nabla F(\bx_{t})\right\Vert ^{2}\right]\lesssim F(\bx_{1})-F_{\star}+G^{2}L^{\frac{2\p-2}{2-\p}}\sum_{t=1}^{T}\eta_{t}^{\frac{\p}{2-\p}}+G^{2-\p}L^{\p-1}\sigma^{\p}\sum_{t=1}^{T}\eta_{t}^{\p}.
\]
In particular, by plugging in different choices of $\left\{ \eta_{t}\right\} _{t=1}^{T}$,
we recover Theorem \ref{thm:main-SGD-exp-smo+lip} when $\p\in\left(1,2\right)$.
The rates in Theorem \ref{thm:main-SGD-exp-smo+lip} when $\p=2$
hold due to the existing literature, e.g., \citep{lan2020first}.
\end{thm}
\begin{proof}
In the following, we write $\Delta\defeq F(\bx_{1})-F_{\star}$. By
Lemma \ref{lem:SGD-core-smo+lip} (with $\nu=\p$), we have
\begin{align*}
\sum_{t=1}^{T}\eta_{t}\left\Vert \nabla F(\bx_{t})\right\Vert ^{2}\leq & 2\Delta+\frac{(2-\p)2^{\frac{2}{2-\p}}G^{2}L^{\frac{2\p-2}{2-\p}}}{\p}\sum_{t=1}^{T}\eta_{t}^{\frac{\p}{2-\p}}\\
 & +\frac{4G^{2-\p}L^{\p-1}}{\p}\sum_{t=1}^{T}\eta_{t}^{\p}\left\Vert \bxi_{t}\right\Vert ^{\p}-2\sum_{t=1}^{T}\eta_{t}\left\langle \bxi_{t},\nabla F(\bx_{t})\right\rangle .
\end{align*}
Next, we take expectations on both sides of the above inequality to
obtain
\[
\E\left[\sum_{t=1}^{T}\eta_{t}\left\Vert \nabla F(\bx_{t})\right\Vert ^{2}\right]\leq2\Delta+\frac{(2-\p)2^{\frac{2}{2-\p}}G^{2}L^{\frac{2\p-2}{2-\p}}}{\p}\sum_{t=1}^{T}\eta_{t}^{\frac{\p}{2-\p}}+\frac{4G^{2-\p}L^{\p-1}\sigma^{\p}}{\p}\sum_{t=1}^{T}\eta_{t}^{\p}.
\]
\end{proof}

\subsection{Full Theorems for $\protect\SGDM$ and Proofs}
\begin{thm}[Full statement of Theorem \ref{thm:main-SGDM-exp}]
\label{thm:SGDM-exp}Under Assumptions \ref{assu:oracle} (with $h=F$)
and \ref{assu:SGDM}, for any $T\in\N$ and nonincreasing $\left\{ \eta_{t}\leq\frac{(1-\beta)^{2}}{2L}\right\} _{t=1}^{T}$,
$\SGDM$ (Algorithm \ref{alg:SGDM}) guarantees that
\[
\E\left[\left(\sum_{t=1}^{T}\eta_{t}\left\Vert \nabla F(\bx_{t})\right\Vert ^{2}\right)^{\frac{\p}{2}}\right]\lesssim\left((1-\beta)(F(\bx_{1})-F_{\star})\right)^{\frac{\p}{2}}+\sigma^{\p}\left(\sum_{t=1}^{T}\eta_{t}^{\frac{\p}{2-\p}}\right)^{\frac{2-\p}{2}}+\frac{L^{\frac{\p}{2}}\sigma^{\p}}{1-\beta}\sum_{t=1}^{T}\eta_{t}^{\p}.
\]
In particular, by plugging in different choices of $\left\{ \eta_{t}\right\} _{t=1}^{T}$
and $\beta$, we recover Theorem \ref{thm:main-SGDM-exp}.
\end{thm}
\begin{rem}
In both Theorem \ref{thm:SGDM-exp} here and Lemma \ref{lem:SGDM-core}
below, it is possible to relax the condition on the stepsize to $\left\{ \eta_{t}\leq\frac{(1-\beta)(1-\beta^{2})}{2L(1+\beta^{2})}\right\} _{t=1}^{T}$.
However, we keep the current one due to its simple form.
\end{rem}
\begin{proof}
In the following, we write $\Delta\defeq F(\bx_{1})-F_{\star}$ and
recall from Lemma \ref{lem:SGDM-core} that $\left\{ \bz_{t}\right\} _{t=1}^{T}$
is a sequence defined as 
\[
\bz_{t}\defeq\frac{\bx_{t}-\beta\bx_{t-1}}{1-\beta},\forall t\in\left[T+1\right],
\]
and $\left\{ \gamma_{t}\right\} _{t=1}^{T}$ is a sequence defined
as
\begin{equation}
\gamma_{t}\defeq\frac{\eta_{t}}{1-\beta},\forall t\in\left[T\right].\label{eq:SGDM-exp-gamma}
\end{equation}
By Lemma \ref{lem:SGDM-core}, we have
\begin{align*}
 & \sum_{t=1}^{T}\gamma_{t}\left\Vert \nabla F(\bx_{t})\right\Vert ^{2}+2\sum_{t=1}^{T}\gamma_{t}\left\Vert \nabla F(\bz_{t})\right\Vert ^{2}\\
\leq & 4\Delta+2L\sum_{t=1}^{T}\gamma_{t}^{2}\left\Vert \bxi_{t}\right\Vert ^{2}+2\beta^{2}(1-\beta)L\sum_{t=1}^{T}\left\Vert \sum_{s=1}^{t-1}\beta^{t-1-s}\gamma_{s}\bxi_{s}\right\Vert ^{2}\\
 & +4L\sum_{t=1}^{T}\gamma_{t}^{2}\left\langle \bxi_{t},\nabla F(\bx_{t})\right\rangle -4\sum_{t=1}^{T}\gamma_{t}\left\langle \bxi_{t},\nabla F(\bz_{t})\right\rangle \\
\leq & 4\Delta+2L\sum_{t=1}^{T}\gamma_{t}^{2}\left\Vert \bxi_{t}\right\Vert ^{2}+2\beta^{2}(1-\beta)L\sum_{t=1}^{T}\left\Vert \sum_{s=1}^{t-1}\beta^{t-1-s}\gamma_{s}\bxi_{s}\right\Vert ^{2}\\
 & +4L\left|\sum_{t=1}^{T}\gamma_{t}^{2}\left\langle \bxi_{t},\nabla F(\bx_{t})\right\rangle \right|+4\left|\sum_{t=1}^{T}\gamma_{t}\left\langle \bxi_{t},\nabla F(\bz_{t})\right\rangle \right|,
\end{align*}
The above inequality further implies that
\begin{align}
 & \frac{\left(\sum_{t=1}^{T}\gamma_{t}\left\Vert \nabla F(\bx_{t})\right\Vert ^{2}\right)^{\frac{\p}{2}}}{2^{1-\frac{\p}{2}}}+\frac{\left(\sum_{t=1}^{T}\gamma_{t}\left\Vert \nabla F(\bz_{t})\right\Vert ^{2}\right)^{\frac{\p}{2}}}{2^{1-\p}}\nonumber \\
\leq & \left(\sum_{t=1}^{T}\gamma_{t}\left\Vert \nabla F(\bx_{t})\right\Vert ^{2}+2\sum_{t=1}^{T}\gamma_{t}\left\Vert \nabla F(\bz_{t})\right\Vert ^{2}\right)^{\frac{\p}{2}}\nonumber \\
\leq & \left[4\Delta+2L\sum_{t=1}^{T}\gamma_{t}^{2}\left\Vert \bxi_{t}\right\Vert ^{2}+2\beta^{2}(1-\beta)L\sum_{t=1}^{T}\left\Vert \sum_{s=1}^{t-1}\beta^{t-1-s}\gamma_{s}\bxi_{s}\right\Vert ^{2}\right.\nonumber \\
 & \left.+4L\left|\sum_{t=1}^{T}\gamma_{t}^{2}\left\langle \bxi_{t},\nabla F(\bx_{t})\right\rangle \right|+4\left|\sum_{t=1}^{T}\gamma_{t}\left\langle \bxi_{t},\nabla F(\bz_{t})\right\rangle \right|\right]^{\frac{\p}{2}}\nonumber \\
\leq & \left(4\Delta\right)^{\frac{\p}{2}}+\left(2L\right)^{\frac{\p}{2}}\sum_{t=1}^{T}\gamma_{t}^{\p}\left\Vert \bxi_{t}\right\Vert ^{\p}+\left(2\beta^{2}(1-\beta)L\right)^{\frac{\p}{2}}\sum_{t=1}^{T}\left\Vert \sum_{s=1}^{t-1}\beta^{t-1-s}\gamma_{s}\bxi_{s}\right\Vert ^{\p}\nonumber \\
 & +\left(4L\right)^{\frac{\p}{2}}\left|\sum_{t=1}^{T}\gamma_{t}^{2}\left\langle \bxi_{t},\nabla F(\bx_{t})\right\rangle \right|^{\frac{\p}{2}}+2^{\p}\left|\sum_{t=1}^{T}\gamma_{t}\left\langle \bxi_{t},\nabla F(\bz_{t})\right\rangle \right|^{\frac{\p}{2}},\label{eq:SGDM-exp-1}
\end{align}
where the second step is by repeatedly using $(a+b)^{\frac{\p}{2}}\leq a^{\frac{\p}{2}}+b^{\frac{\p}{2}},\forall a,b\geq0$
when $\p\in\left[0,2\right]$.

Next, we take expectations on both sides of (\ref{eq:SGDM-exp-1})
to obtain
\begin{align}
 & \E\left[\frac{\left(\sum_{t=1}^{T}\gamma_{t}\left\Vert \nabla F(\bx_{t})\right\Vert ^{2}\right)^{\frac{\p}{2}}}{2^{1-\frac{\p}{2}}}+\frac{\left(\sum_{t=1}^{T}\gamma_{t}\left\Vert \nabla F(\bz_{t})\right\Vert ^{2}\right)^{\frac{\p}{2}}}{2^{1-\p}}\right]\nonumber \\
\leq & \left(4\Delta\right)^{\frac{\p}{2}}+\left(2L\right)^{\frac{\p}{2}}\sigma^{\p}\sum_{t=1}^{T}\gamma_{t}^{\p}+\left(2\beta^{2}(1-\beta)L\right)^{\frac{\p}{2}}\sum_{t=1}^{T}\E\left[\left\Vert \sum_{s=1}^{t-1}\beta^{t-1-s}\gamma_{s}\bxi_{s}\right\Vert ^{\p}\right]\nonumber \\
 & +\left(4L\right)^{\frac{\p}{2}}\E\left[\left|\sum_{t=1}^{T}\gamma_{t}^{2}\left\langle \bxi_{t},\nabla F(\bx_{t})\right\rangle \right|^{\frac{\p}{2}}\right]+2^{\p}\E\left[\left|\sum_{t=1}^{T}\gamma_{t}\left\langle \bxi_{t},\nabla F(\bz_{t})\right\rangle \right|^{\frac{\p}{2}}\right]\nonumber \\
\overset{(a)}{\leq} & \left(4\Delta\right)^{\frac{\p}{2}}+\left(2L\right)^{\frac{\p}{2}}\sigma^{\p}\sum_{t=1}^{T}\gamma_{t}^{\p}+2\left(\beta^{2}(1-\beta)L\right)^{\frac{\p}{2}}\sigma^{\p}\sum_{t=1}^{T}\sum_{s=1}^{t-1}\beta^{\p(t-1-s)}\gamma_{s}^{\p}\nonumber \\
 & +2^{1+\frac{\p}{2}}L^{\frac{\p}{2}}\sigma^{\frac{\p}{2}}\sqrt{\sum_{t=1}^{T}\gamma_{t}^{2\p}\E\left[\left\Vert \nabla F(\bx_{t})\right\Vert ^{\p}\right]}+2^{1+\frac{\p}{2}}\sigma^{\frac{\p}{2}}\sqrt{\sum_{t=1}^{T}\gamma_{t}^{\p}\E\left[\left\Vert \nabla F(\bz_{t})\right\Vert ^{\p}\right]}\nonumber \\
\overset{(b)}{\leq} & \left(4\Delta\right)^{\frac{\p}{2}}+\left[2^{\frac{\p}{2}}+\frac{2\left(\beta^{2}(1-\beta)\right)^{\frac{\p}{2}}}{1-\beta^{\p}}\right]L^{\frac{\p}{2}}\sigma^{\p}\sum_{t=1}^{T}\gamma_{t}^{\p}\nonumber \\
 & +2(1-\beta)^{\frac{\p}{2}}\sigma^{\frac{\p}{2}}\sqrt{\left(\sum_{t=1}^{T}\gamma_{t}^{\frac{\p}{2-\p}}\right)^{\frac{2-\p}{2}}\E\left[\left(\sum_{t=1}^{T}\gamma_{t}\left\Vert \nabla F(\bx_{t})\right\Vert ^{2}\right)^{\frac{\p}{2}}\right]}\nonumber \\
 & +2^{1+\frac{\p}{2}}\sigma^{\frac{\p}{2}}\sqrt{\left(\sum_{t=1}^{T}\gamma_{t}^{\frac{\p}{2-\p}}\right)^{\frac{2-\p}{2}}\E\left[\left(\sum_{t=1}^{T}\gamma_{t}\left\Vert \nabla F(\bz_{t})\right\Vert ^{2}\right)^{\frac{\p}{2}}\right]}\nonumber \\
\overset{(c)}{\leq} & \left(4\Delta\right)^{\frac{\p}{2}}+\left[2^{\frac{\p}{2}}+\frac{2\left(\beta^{2}(1-\beta)\right)^{\frac{\p}{2}}}{1-\beta^{\p}}\right]L^{\frac{\p}{2}}\sigma^{\p}\sum_{t=1}^{T}\gamma_{t}^{\p}+\left[2^{2-\frac{\p}{2}}(1-\beta)^{\p}+2\right]\sigma^{\p}\left(\sum_{t=1}^{T}\gamma_{t}^{\frac{\p}{2-\p}}\right)^{\frac{2-\p}{2}}\nonumber \\
 & +\frac{\E\left[\left(\sum_{t=1}^{T}\gamma_{t}\left\Vert \nabla F(\bx_{t})\right\Vert ^{2}\right)^{\frac{\p}{2}}\right]}{2^{2-\frac{\p}{2}}}+\frac{\E\left[\left(\sum_{t=1}^{T}\gamma_{t}\left\Vert \nabla F(\bz_{t})\right\Vert ^{2}\right)^{\frac{\p}{2}}\right]}{2^{1-\p}},\label{eq:SGDM-exp-2}
\end{align}
where $(a)$ is by applying
\[
\E\left[\left\Vert \sum_{s=1}^{t-1}\beta^{t-1-s}\gamma_{s}\bxi_{s}\right\Vert ^{\p}\right]\leq2^{2-\p}\sum_{s=1}^{t-1}\beta^{\p(t-1-s)}\gamma_{s}^{\p}\E\left[\left\Vert \bxi_{s}\right\Vert ^{\p}\right]\leq2^{2-\p}\sigma^{\p}\sum_{s=1}^{t-1}\beta^{\p(t-1-s)}\gamma_{s}^{\p}
\]
and Lemma \ref{lem:core} with $\by_{t}=\gamma_{t}^{2}\nabla F(\bx_{t})$
and $\by_{t}=\gamma_{t}\nabla F(\bz_{t})$, $(b)$ is due to $\gamma_{t}=\frac{\eta_{t}}{1-\beta}\leq\frac{1-\beta}{2L},\forall t\in\left[T\right]$
by our requirement on $\left\{ \eta_{t}\right\} _{t=1}^{T}$ and Cauchy-Schwarz
inequality similar to (\ref{eq:SGD-exp-3}), and $(c)$ follows from
AM-GM inequality.

Finally, we rearrange the terms in (\ref{eq:SGDM-exp-2}), use $\beta\leq1$
and $\frac{1}{1-\beta^{\p}}\leq\frac{1}{1-\beta}$, and plug in $\gamma_{t}=\frac{\eta_{t}}{1-\beta},\forall t\in\left[T\right]$
to conclude.
\end{proof}

\subsection{Helpful Lemmas}

We prove some helpful lemmas used in the proofs.
\begin{lem}
\label{lem:SGD-core}Under Assumption \ref{assu:SGD}, for any $T\in\N$
and $\left\{ \eta_{t}\leq\frac{1}{L}\right\} _{t=1}^{T}$, $\SGD$
(Algorithm \ref{alg:SGD}) guarantees that
\[
\sum_{t=1}^{T}\eta_{t}\left\Vert \nabla F(\bx_{t})\right\Vert ^{2}\leq2\left(F(\bx_{1})-F_{\star}\right)+L\sum_{t=1}^{T}\eta_{t}^{2}\left\Vert \bxi_{t}\right\Vert ^{2}+2\sum_{t=1}^{T}\left(\eta_{t}^{2}L-\eta_{t}\right)\left\langle \bxi_{t},\nabla F(\bx_{t})\right\rangle .
\]
\end{lem}
\begin{proof}
Given $t\in\left[T\right]$, by upper smoothness of $F$ and the update
rule of $\SGD$, we have
\begin{align*}
F(\bx_{t+1}) & \leq F(\bx_{t})+\left\langle \nabla F(\bx_{t}),\bx_{t+1}-\bx_{t}\right\rangle +\frac{L}{2}\left\Vert \bx_{t+1}-\bx_{t}\right\Vert ^{2}\\
 & =F(\bx_{t})-\eta_{t}\left\langle \nabla F(\bx_{t}),\bg_{t}\right\rangle +\frac{\eta_{t}^{2}L}{2}\left\Vert \bg_{t}\right\Vert ^{2}\\
 & =F(\bx_{t})-\left(\eta_{t}-\frac{\eta_{t}^{2}L}{2}\right)\left\Vert \nabla F(\bx_{t})\right\Vert ^{2}+\frac{\eta_{t}^{2}L}{2}\left\Vert \bxi_{t}\right\Vert ^{2}+\left(\eta_{t}^{2}L-\eta_{t}\right)\left\langle \bxi_{t},\nabla F(\bx_{t})\right\rangle \\
 & \overset{\eta_{t}\leq\frac{1}{L}}{\leq}F(\bx_{t})-\frac{\eta_{t}}{2}\left\Vert \nabla F(\bx_{t})\right\Vert ^{2}+\frac{\eta_{t}^{2}L}{2}\left\Vert \bxi_{t}\right\Vert ^{2}+\left(\eta_{t}^{2}L-\eta_{t}\right)\left\langle \bxi_{t},\nabla F(\bx_{t})\right\rangle .
\end{align*}
Finally, we sum the above inequality from $t=1$ to $T$, rearrange
terms, and drop the nonnegative term $2\left(F(\bx_{T+1})-F_{\star}\right)$
to conclude.
\end{proof}

\begin{lem}
\label{lem:SGD-core-smo+lip}Under Assumption \ref{assu:SGD} and
an additional condition $\left\Vert \nabla F(\bx)\right\Vert \leq G,\forall\bx\in\R^{d}$,
for any $T\in\N$ and $\left\{ \eta_{t}\right\} _{t=1}^{T}$, $\SGD$
(Algorithm \ref{alg:SGD}) guarantees that, for any $\nu\in\left[1,2\right]$,
\begin{align*}
\sum_{t=1}^{T}\eta_{t}\left\Vert \nabla F(\bx_{t})\right\Vert ^{2}\leq & 2\left(F(\bx_{1})-F_{\star}\right)+\frac{(2-\nu)2^{\frac{\nu}{2-\nu}}G^{2}L^{\frac{2\nu-2}{2-\nu}}}{\nu}\sum_{t=1}^{T}\eta_{t}^{\frac{\nu}{2-\nu}}\\
 & +\frac{4G^{2-\nu}L^{\nu-1}}{\nu}\sum_{t=1}^{T}\eta_{t}^{\nu}\left\Vert \bxi_{t}\right\Vert ^{\nu}-2\sum_{t=1}^{T}\eta_{t}\left\langle \bxi_{t},\nabla F(\bx_{t})\right\rangle .
\end{align*}
\end{lem}
\begin{proof}
We fix $\nu\in\left[1,2\right]$ in the following proof. Given $t\in\left[T\right]$,
by Lemma \ref{lem:general-holder} and the update rule of $\SGD$,
we have
\begin{align}
F(\bx_{t+1}) & \leq F(\bx_{t})+\left\langle \nabla F(\bx_{t}),\bx_{t+1}-\bx_{t}\right\rangle +\frac{\left(2G\right)^{2-\nu}L^{\nu-1}}{\nu}\left\Vert \bx_{t+1}-\bx_{t}\right\Vert ^{\nu}\nonumber \\
 & =F(\bx_{t})-\eta_{t}\left\langle \nabla F(\bx_{t}),\bg_{t}\right\rangle +\frac{\eta_{t}^{\nu}\left(2G\right)^{2-\nu}L^{\nu-1}}{\nu}\left\Vert \bg_{t}\right\Vert ^{\nu}\nonumber \\
 & \overset{(a)}{\leq}F(\bx_{t})-\eta_{t}\left\Vert \nabla F(\bx_{t})\right\Vert ^{2}-\eta_{t}\left\langle \bxi_{t},\nabla F(\bx_{t})\right\rangle +\frac{2\eta_{t}^{\nu}G^{2-\nu}L^{\nu-1}}{\nu}\left(\left\Vert \nabla F(\bx_{t})\right\Vert ^{\nu}+\left\Vert \bxi_{t}\right\Vert ^{\nu}\right)\nonumber \\
 & \overset{(b)}{\leq}F(\bx_{t})-\frac{\eta_{t}}{2}\left\Vert \nabla F(\bx_{t})\right\Vert ^{2}+\frac{(2-\nu)2^{\frac{\nu}{2-\nu}}\eta_{t}^{\frac{\nu}{2-\nu}}G^{2}L^{\frac{2\nu-2}{2-\nu}}}{\nu}+\frac{2\eta_{t}^{\nu}G^{2-\nu}L^{\nu-1}}{\nu}\left\Vert \bxi_{t}\right\Vert ^{\nu}-\eta_{t}\left\langle \bxi_{t},\nabla F(\bx_{t})\right\rangle ,\label{eq:SGD-core-smo+lip-1}
\end{align}
where $(a)$ is due to $\left\Vert \bx+\by\right\Vert ^{\nu}\leq2^{\nu-1}\left(\left\Vert \bx\right\Vert ^{\nu}+\left\Vert \by\right\Vert ^{\nu}\right),\forall\bx,\by\in\R^{d}$,
and $(b)$ is by
\begin{align*}
 & \frac{2\eta_{t}^{\nu}G^{2-\nu}L^{\nu-1}\left\Vert \nabla F(\bx_{t})\right\Vert ^{\nu}}{\nu}=\frac{2\eta_{t}^{\frac{\nu}{2}}G^{2-\nu}L^{\nu-1}}{\nu^{1-\frac{\nu}{2}}}\cdot\frac{\eta_{t}^{\frac{\nu}{2}}\left\Vert \nabla F(\bx_{t})\right\Vert ^{\nu}}{\nu^{\frac{\nu}{2}}}\\
\overset{(c)}{\leq} & \frac{\left(2\eta_{t}^{\frac{\nu}{2}}G^{2-\nu}L^{\nu-1}/\nu^{1-\frac{\nu}{2}}\right)^{\frac{2}{2-\nu}}}{2/(2-\nu)}+\frac{\left(\eta_{t}^{\frac{\nu}{2}}\left\Vert \nabla F(\bx_{t})\right\Vert ^{\nu}/\nu^{\frac{\nu}{2}}\right)^{\frac{2}{\nu}}}{2/\nu}=\frac{(2-\nu)2^{\frac{\nu}{2-\nu}}\eta_{t}^{\frac{\nu}{2-\nu}}G^{2}L^{\frac{2\nu-2}{2-\nu}}}{\nu}+\frac{\eta_{t}\left\Vert \nabla F(\bx_{t})\right\Vert ^{2}}{2},
\end{align*}
in which $(c)$ follows from Young's inequality. Finally, we sum (\ref{eq:SGD-core-smo+lip-1})
from $t=1$ to $T$, rearrange terms, and drop the nonnegative term
$2\left(F(\bx_{T+1})-F_{\star}\right)$ to conclude.
\end{proof}

\begin{lem}
\label{lem:SGDM-core}Under Assumption \ref{assu:SGDM}, for any $T\in\N$
and nonincreasing $\left\{ \eta_{t}\leq\frac{(1-\beta)^{2}}{2L}\right\} _{t=1}^{T}$,
$\SGDM$ (Algorithm \ref{alg:SGDM}) guarantees that
\begin{align*}
 & \sum_{t=1}^{T}\gamma_{t}\left\Vert \nabla F(\bx_{t})\right\Vert ^{2}+2\sum_{t=1}^{T}\gamma_{t}\left\Vert \nabla F(\bz_{t})\right\Vert ^{2}\\
\leq & 4\left(F(\bx_{1})-F_{\star}\right)+2L\sum_{t=1}^{T}\gamma_{t}^{2}\left\Vert \bxi_{t}\right\Vert ^{2}+2\beta^{2}(1-\beta)L\sum_{t=1}^{T}\left\Vert \sum_{s=1}^{t-1}\beta^{t-1-s}\gamma_{s}\bxi_{s}\right\Vert ^{2}\\
 & +4L\sum_{t=1}^{T}\gamma_{t}^{2}\left\langle \bxi_{t},\nabla F(\bx_{t})\right\rangle -4\sum_{t=1}^{T}\gamma_{t}\left\langle \bxi_{t},\nabla F(\bz_{t})\right\rangle ,
\end{align*}
where $\left\{ \bz_{t}\right\} _{t=1}^{T+1}$ is a sequence defined
in (\ref{eq:SGDM-core-z}) and $\left\{ \gamma_{t}\right\} _{t=1}^{T}$
is a sequence defined in (\ref{eq:SGDM-core-z-update}).
\end{lem}
\begin{proof}
Inspired by \citep{7330562,NEURIPS2020_d3f5d4de}, we first introduce
an auxiliary sequence $\left\{ \bz_{t}\right\} _{t=1}^{T+1}$ defined
as
\begin{equation}
\bz_{t}\defeq\frac{\bx_{t}-\beta\bx_{t-1}}{1-\beta},\forall t\in\left[T+1\right],\label{eq:SGDM-core-z}
\end{equation}
which, combined with the update rule of $\SGDM$, implies that for
any $t\in\left[T\right]$,
\begin{eqnarray}
\bz_{t+1}=\bz_{t}-\gamma_{t}\bg_{t}, & \text{where} & \gamma_{t}\defeq\frac{\eta_{t}}{1-\beta}.\label{eq:SGDM-core-z-update}
\end{eqnarray}
A useful fact that will be applied in the later analysis is that
\[
\bz_{t}-\bx_{t}\overset{(\ref{eq:SGDM-core-z})}{=}\frac{\beta}{1-\beta}\left(\bx_{t}-\bx_{t-1}\right)=-\frac{\beta}{1-\beta}\sum_{s=1}^{t-1}\beta^{t-1-s}\eta_{s}\bg_{s}\overset{(\ref{eq:SGDM-core-z-update})}{=}-\beta\sum_{s=1}^{t-1}\beta^{t-1-s}\gamma_{s}\bg_{s},
\]
where the second equation is obtained by expanding the update rule
of $\SGDM$.

Given $t\in\left[T\right]$, by smoothness of $F$, we have
\begin{align}
F(\bz_{t+1})\leq & F(\bz_{t})+\left\langle \nabla F(\bz_{t}),\bz_{t+1}-\bz_{t}\right\rangle +\frac{L}{2}\left\Vert \bz_{t+1}-\bz_{t}\right\Vert ^{2}\nonumber \\
\overset{(\ref{eq:SGDM-core-z-update})}{=} & F(\bz_{t})-\gamma_{t}\left\langle \nabla F(\bz_{t}),\bg_{t}\right\rangle +\frac{\gamma_{t}^{2}L}{2}\left\Vert \bg_{t}\right\Vert ^{2}\nonumber \\
= & F(\bz_{t})-\gamma_{t}\left\langle \nabla F(\bz_{t}),\nabla F(\bx_{t})\right\rangle +\frac{\gamma_{t}^{2}L}{2}\left\Vert \nabla F(\bx_{t})\right\Vert ^{2}\nonumber \\
 & +\frac{\gamma_{t}^{2}L}{2}\left\Vert \bxi_{t}\right\Vert ^{2}+\gamma_{t}^{2}L\left\langle \bxi_{t},\nabla F(\bx_{t})\right\rangle -\gamma_{t}\left\langle \bxi_{t},\nabla F(\bz_{t})\right\rangle \nonumber \\
= & F(\bz_{t})+\frac{\gamma_{t}^{2}L-\gamma_{t}}{2}\left\Vert \nabla F(\bx_{t})\right\Vert ^{2}-\frac{\gamma_{t}}{2}\left\Vert \nabla F(\bz_{t})\right\Vert ^{2}+\frac{\gamma_{t}}{2}\left\Vert \nabla F(\bz_{t})-\nabla F(\bx_{t})\right\Vert ^{2}\nonumber \\
 & +\frac{\gamma_{t}^{2}L}{2}\left\Vert \bxi_{t}\right\Vert ^{2}+\gamma_{t}^{2}L\left\langle \bxi_{t},\nabla F(\bx_{t})\right\rangle -\gamma_{t}\left\langle \bxi_{t},\nabla F(\bz_{t})\right\rangle .\label{eq:SGDM-core-1}
\end{align}
Note that, by smoothness of $F$,
\begin{align*}
\left\Vert \nabla F(\bz_{t})-\nabla F(\bx_{t})\right\Vert  & \leq L\left\Vert \bz_{t}-\bx_{t}\right\Vert \overset{(\ref{eq:SGDM-core-z-update})}{=}\beta L\left\Vert \sum_{s=1}^{t-1}\beta^{t-1-s}\gamma_{s}\bg_{s}\right\Vert =\beta L\left\Vert \sum_{s=1}^{t-1}\beta^{t-1-s}\gamma_{s}\left(\nabla F(\bx_{s})+\bxi_{s}\right)\right\Vert \\
 & \leq\beta L\left\Vert \sum_{s=1}^{t-1}\beta^{t-1-s}\gamma_{s}\nabla F(\bx_{s})\right\Vert +\beta L\left\Vert \sum_{s=1}^{t-1}\beta^{t-1-s}\gamma_{s}\bxi_{s}\right\Vert ,
\end{align*}
which implies that
\begin{align}
\left\Vert \nabla F(\bz_{t})-\nabla F(\bx_{t})\right\Vert ^{2} & \leq2\beta^{2}L^{2}\left\Vert \sum_{s=1}^{t-1}\beta^{t-1-s}\gamma_{s}\nabla F(\bx_{s})\right\Vert ^{2}+2\beta^{2}L^{2}\left\Vert \sum_{s=1}^{t-1}\beta^{t-1-s}\gamma_{s}\bxi_{s}\right\Vert ^{2}\nonumber \\
 & \leq2\beta^{2}L^{2}\left(\sum_{s^{\prime}=1}^{t-1}\beta^{t-1-s^{\prime}}\right)\sum_{s=1}^{t-1}\beta^{t-1-s}\gamma_{s}^{2}\left\Vert \nabla F(\bx_{s})\right\Vert ^{2}+2\beta^{2}L^{2}\left\Vert \sum_{s=1}^{t-1}\beta^{t-1-s}\gamma_{s}\bxi_{s}\right\Vert ^{2}\nonumber \\
 & \leq\frac{2\beta^{2}L^{2}}{1-\beta}\sum_{s=1}^{t-1}\beta^{t-1-s}\gamma_{s}^{2}\left\Vert \nabla F(\bx_{s})\right\Vert ^{2}+2\beta^{2}L^{2}\left\Vert \sum_{s=1}^{t-1}\beta^{t-1-s}\gamma_{s}\bxi_{s}\right\Vert ^{2}.\label{eq:SGDM-core-2}
\end{align}
We combine (\ref{eq:SGDM-core-1}) and (\ref{eq:SGDM-core-2}), sum
the obtained inequality from $t=1$ to $T$, use $\bz_{1}=\bx_{1}$,
apply the nonincreasing property of $\left\{ \eta_{t}\right\} _{t=1}^{T}$,
rearrange terms, and drop the nonnegative term $F(\bz_{T+1})-F_{\star}$
to obtain that
\begin{align}
 & \sum_{t=1}^{T}\left(\frac{\gamma_{t}}{2}-\frac{\gamma_{t}^{2}L}{2}-\frac{\beta^{2}\gamma_{t}^{3}L^{2}}{(1-\beta)^{2}}\right)\left\Vert \nabla F(\bx_{t})\right\Vert ^{2}+\frac{1}{2}\sum_{t=1}^{T}\gamma_{t}\left\Vert \nabla F(\bz_{t})\right\Vert ^{2}\nonumber \\
\leq & F(\bx_{1})-F_{\star}+\frac{L}{2}\sum_{t=1}^{T}\gamma_{t}^{2}\left\Vert \bxi_{t}\right\Vert ^{2}+\beta^{2}L^{2}\sum_{t=1}^{T}\gamma_{t}\left\Vert \sum_{s=1}^{t-1}\beta^{t-1-s}\gamma_{s}\bxi_{s}\right\Vert ^{2}\nonumber \\
 & +L\sum_{t=1}^{T}\gamma_{t}^{2}\left\langle \bxi_{t},\nabla F(\bx_{t})\right\rangle -\sum_{t=1}^{T}\gamma_{t}\left\langle \bxi_{t},\nabla F(\bz_{t})\right\rangle .\label{eq:SGDM-core-3}
\end{align}
Observe that, by the requirement of $\eta_{t}\leq\frac{(1-\beta)^{2}}{2L},\forall t\in\left[T\right]$,
\begin{eqnarray*}
\gamma_{t}=\frac{\eta_{t}}{1-\beta}\leq\frac{1-\beta}{2L} & \text{and} & \frac{\gamma_{t}L}{2}+\frac{\beta^{2}\gamma_{t}^{2}L^{2}}{(1-\beta)^{2}}\leq\frac{1-\beta}{4}+\frac{\beta^{2}}{4}\overset{\beta<1}{\leq}\frac{1}{4}.
\end{eqnarray*}
Plugging the above inequality into (\ref{eq:SGDM-core-3}) and multiplying
both sides by $4$ completes the proof.
\end{proof}

\section{A General Lemma}
\begin{lem}
\label{lem:general-holder}Given a differentiable function $F:\R^{d}\to\R$
that is upper $L$-smooth and $G$-Lipschitz w.r.t. a general norm
$\left\Vert \cdot\right\Vert $, i.e., $\Breg_{F}(\bx,\by)\leq\frac{L}{2}\left\Vert \bx-\by\right\Vert ^{2}$
and $\left\Vert \nabla F(\bx)\right\Vert _{\star}\leq G$, then we
have
\[
F(\bx)\leq F(\by)+\left\langle \nabla F(\by),\bx-\by\right\rangle +\frac{\left(2G\right)^{2-\nu}L^{\nu-1}}{\nu}\left\Vert \bx-\by\right\Vert ^{\nu},\forall\bx,\by\in\R^{d},\nu\in\left[1,2\right].
\]
\end{lem}
\begin{proof}
Given $\bx,\by\in\R^{d}$, $\nu\in\left[1,2\right]$, and $t\in\left[0,1\right]$,
we claim there is always
\begin{equation}
\left\langle \nabla F(\by+t(\bx-\by))-\nabla F(\by),\bx-\by\right\rangle \leq\left(2G\right)^{2-\nu}L^{\nu-1}\left\Vert \bx-\by\right\Vert ^{\nu}t^{\nu-1}.\label{eq:general-holder-1}
\end{equation}
If $\left\langle \nabla F(\by+t(\bx-\by))-\nabla F(\by),\bx-\by\right\rangle \leq0$,
the above inequality clearly holds. Otherwise, we have
\begin{align*}
\left\langle \nabla F(\by+t(\bx-\by))-\nabla F(\by),\bx-\by\right\rangle  & =\frac{\left\langle \nabla F(\by+t(\bx-\by))-\nabla F(\by),t(\bx-\by)\right\rangle }{t}\\
 & \leq\frac{\left(2G\left\Vert t(\bx-\by)\right\Vert \right)^{2-\nu}\left(L\left\Vert t(\bx-\by)\right\Vert ^{2}\right)^{\nu-1}}{t}\\
 & =\left(2G\right)^{2-\nu}L^{\nu-1}\left\Vert \bx-\by\right\Vert ^{\nu}t^{\nu-1},
\end{align*}
where the second to last step is due to, by $G$-Lipschitzness,
\[
\left\langle \nabla F(\by+t(\bx-\by))-\nabla F(\by),t(\bx-\by)\right\rangle \leq\left\Vert \nabla F(\by+t(\bx-\by))-\nabla F(\by)\right\Vert _{\star}\left\Vert t(\bx-\by)\right\Vert \leq2G\left\Vert t(\bx-\by)\right\Vert ,
\]
and by upper $L$-smoothness,
\[
\left\langle \nabla F(\by+t(\bx-\by))-\nabla F(\by),t(\bx-\by)\right\rangle \leq L\left\Vert t(\bx-\by)\right\Vert ^{2}.
\]

Now, for any $\bx,\by\in\R^{d}$ and $\nu\in\left[1,2\right]$, by
the First Fundamental Theorem of Calculus,
\begin{align*}
F(\bx)-F(\by) & =\int_{0}^{1}\left\langle \nabla F(\by+t(\bx-\by)),\bx-\by\right\rangle \d t\\
 & =\left\langle \nabla F(\by),\bx-\by\right\rangle +\int_{0}^{1}\left\langle \nabla F(\by+t(\bx-\by))-\nabla F(\by),\bx-\by\right\rangle \d t\\
 & \overset{(\ref{eq:general-holder-1})}{\leq}\left\langle \nabla F(\by),\bx-\by\right\rangle +\left(2G\right)^{2-\nu}L^{\nu-1}\left\Vert \bx-\by\right\Vert ^{\nu}\int_{0}^{1}t^{\nu-1}\d t\\
 & \leq\left\langle \nabla F(\by),\bx-\by\right\rangle +\frac{\left(2G\right)^{2-\nu}L^{\nu-1}}{\nu}\left\Vert \bx-\by\right\Vert ^{\nu}.
\end{align*}
\end{proof}

\end{document}